\documentclass{amsart}
\usepackage{a4wide,amsmath,amsthm,amssymb,amsfonts}
\usepackage{enumitem}
\usepackage{mathdots}
\usepackage[all]{xy}
\usepackage{bm}
\usepackage{mathscinet}
\usepackage{tikz-cd}
\usepackage{cite}
\usepackage[pagebackref]{hyperref}
\usepackage{hyperref}
\usepackage{amsrefs}

\numberwithin{equation}{section}
\newtheorem{theorem}{Theorem}[section]
\newtheorem{proposition}[theorem]{Proposition}
\newtheorem{definition}[theorem]{Definition}
\newtheorem{corollary}[theorem]{Corollary}
\newtheorem{lemma}[theorem]{Lemma}
\newtheorem{remark}[theorem]{Remark}
\newtheorem{example}[theorem]{Example}

\newcommand{\GL}{\operatorname{GL}}
\newcommand{\SL}{\operatorname{SL}}
\newcommand{\Mat}{\operatorname{Mat}}

\newcommand{\Hom}{\operatorname{Hom}}
\newcommand{\diag}{\operatorname{diag}}
\newcommand{\Irr}{\operatorname{Irr}}
\newcommand{\Reps}{\mathcal{M}}
\newcommand{\adm}{\operatorname{adm}}
\newcommand{\fl}{\operatorname{fl}}
\newcommand{\modulus}{\bmod}
\newcommand{\JH}{\operatorname{JH}}
\newcommand{\soc}{\operatorname{soc}}
\newcommand{\Ind}{\operatorname{Ind}}
\newcommand{\ind}{\operatorname{ind}}
\newcommand{\Res}{\operatorname{Res}}
\newcommand{\Ker}{\operatorname{Ker}}
\newcommand{\WCusp}{\operatorname{WCusp}}
\newcommand{\id}{\operatorname{id}}
\newcommand{\chr}{\operatorname{char}}
\newcommand{\der}{\operatorname{der}}
\newcommand{\LInd}{\mathcal{LI}}                        %  Lagrangian induction
\newcommand{\LRes}{\mathcal{LR}}                        %  Lagrangian restriction
\newcommand{\ex}[2]{#1[#2]}                             % extension by center
\newcommand{\sml}{\operatorname{sml}}                   % scalaras in F^*n
\newcommand{\lrg}{\operatorname{lrg}}                   % determinant in F^*n
\newcommand{\lf}{\operatorname{lf}}                     % locally finite
\newcommand{\hlf}{\mbox{-}\operatorname{lf}}
\newcommand{\tnsr}{\mathfrak{T}}                         % metaplectic tensor
\newcommand{\PD}[1]{\operatorname{PD}(#1)}               % Pontriyagin dual
\newcommand{\bas}[1]{\underline{#1}}                     % base (covered) group
\newcommand{\trns}{\mathfrak{F}}                         % trnasfer for well-matched covering groups
\newcommand{\twst}{\sim}                               % twisting
\newcommand{\incvr}{\sim}                              % commutator upstairs
\newcommand{\sm}[4]{{\bigl(\begin{smallmatrix}{#1}&{#2}\\{#3}&{#4}\end{smallmatrix}\bigr)}}

\newcommand{\sqr}{\operatorname{sqr}}                  % square-integrable
\newcommand{\vol}{\operatorname{vol}}
\newcommand{\norm}[1]{\lVert#1\rVert}
\newcommand{\pr}{\bm{p}}
\newcommand{\sctn}{\bm{s}}
\newcommand{\bs}{\backslash}
\newcommand{\ndiv}{\nmid}
\newcommand{\sprod}[2]{\left\langle{#1},{#2}\right\rangle}        % inner product
\newcommand{\dsprod}[2]{\left\langle\langle{#1},{#2}\right\rangle\rangle}
\newcommand{\Z}{\mathbb{Z}}                             % integers
\newcommand{\C}{\mathbb{C}}                             % complex numbers
\newcommand{\R}{\mathbb{R}}                             % real numbers
\newcommand{\m}{\mathfrak{m}}                           % multisegment
\newcommand{\ordr}{\mathfrak{o}}                        % ring of integers
\newcommand{\twstoper}{\mathcal{S}}                     % certain intertwining operator
\newcommand{\TwstOper}{\mathcal{T}}
\newcommand{\intrchng}{\mathcal{R}}                     % interchange

\newcommand{\Jac}[2]{r_{{#1},{#2}}}%{r_{#2}^{#1}}%
\newcommand{\Jacn}[3]{r_{{#1},{#2}}^{#3}}%{r_{#2}^{#1}}%
\newcommand{\Jacbar}[2]{\overline r_{{#1},{#2}}}%{\overline r_{#2}^{#1}}%
\newcommand{\Jacbarn}[3]{\overline r_{{#1},{#2}}^{#3}}%{\overline r_{#2}^{#1}}%

\newcommand{\Pind}[2]{i_{{#1},{#2}}}%{i_{#1}^{#2}}%
\newcommand{\Pindn}[3]{i_{{#1},{#2}}^{#3}}
\newcommand{\Pindbar}[2]{\overline i_{{#1},{#2}}}%{\overline i_{#1}^{#2}}%
\newcommand{\Pindbarn}[3]{\overline i_{{#1},{#2}}^{#3}}%{\overline i_{#1}^{#2}}%

\newcommand{\Pgp}[2]{P_{{#1},{#2}}}%{P_{#2}^{#1}}%
\newcommand{\ordPgp}[2]{\bas{P}_{{#1},{#2}}}
\newcommand{\ordPgpbar}[2]{\bas{P}_{{#1},{#2}}^{-}}

\newcommand{\ordNgp}[2]{\bas{U}_{{#1},{#2}}}
\newcommand{\ordNgpbar}[2]{\bas{U}_{{#1},{#2}}^{-}}

\newcommand{\abs}[1]{\left|{#1}\right|}
\newcommand{\rest}{\lvert}

\begin{document}
	
\title[Classification of irreducible representations of covering groups]
{Classification of irreducible representations of metaplectic covers of the general linear group over a non-archimedean local field}

\author{Eyal Kaplan}
\address{Department of Mathematics, Bar-Ilan University, Ramat-Gan 5290002, Israel}
\thanks{E.K. partially supported by the Israel Science Foundation (grant numbers 376/21 and 421/17).}
\email{kaplaney@gmail.com}
\author{Erez Lapid}
\address{Department of Mathematics, Weizmann Institute of Science, Rehovot 7610001, Israel}
\email{erez.m.lapid@gmail.com}
\author{Jiandi Zou}
\thanks{J.Z. partially supported by the Israel Science Foundation (grant No. 737/20)}
\address{Mathematics Department, Technion -- Israel Institute of Technology, Haifa 3200003, Israel}
\email{idealzjd@gmail.com}

%\date{\today}

\begin{abstract}
Let $F$ be a non-archimedean local field.
The classification of the irreducible representations of $\GL_n(F)$, $n\ge0$
in terms of supercuspidal representations is one of the highlights of the Bernstein--Zelevinsky theory.
We give an analogous classification for metaplectic coverings of $\GL_n(F)$, $n\ge0$.
\end{abstract}

\maketitle
\setcounter{tocdepth}{1}
\tableofcontents

\section{Introduction}

A cornerstone in the representation theory of $p$-adic groups is the Bernstein--Zelevinsky theory,
culminating in the classification of the irreducible representations of the general linear groups $\GL_n(F)$, $n\ge0$
where $F$ is a non-archimedean local field. (We refer to them collectively as irreducible representations of $\GL$.
In this paper we only consider smooth, complex representations.)
The ``elementary particles'' in the classification are the irreducible supercuspidal representations,
which are treated as a black box.
Using them, one forms segments, which are simply finite sets of irreducible supercuspidal representations of the form
$\{\rho,\rho\cdot\abs{\cdot},\dots,\rho\cdot\abs{\cdot}^{k-1}\}$, $k\ge1$, where
$\rho\cdot\abs{\cdot}^i$ denotes twisting by the character $\abs{\det}^i$.
Remarkably, the irreducible representations of $\GL$ are classified by multisegments, which are nothing but
finite formal sums of segments -- an ostensibly simple combinatorial object.
This is the main result of Zelevinsky in \cite{MR584084}.
Subsequently, his proof was (slightly) simplified and extended to inner forms of $\GL$ and to representations
in characteristic $\ne p$.
See \cite{MR3573961}*{Appendix A} for a recent account and more details.

The importance of coverings of reductive $p$-adic (and adelic) groups (the so-called metaplectic groups) was realized early on by
Steinberg, Weil, Moore, Kubota, Matsumoto, Bass, Milnor and Serre among others \cites{MR0165033, MR204422, MR0255490, MR244258, MR0466335, MR240214, MR0244257}.
A broader scope for covering groups was later on conceived by Brylinski and Deligne in \cite{MR1896177}.
We refer to \cites{MR3802417, MR3802419} and the references therein for general conjectures and perspective.
The basic elements of the representation theory of reductive $p$-adic groups carry over to metaplectic groups
essentially without change (see \cites{MR3151110, MR3516189, MR3053009, 2206.06905}).

In this paper we will consider the metaplectic covers of $\GL_n(F)$ considered in the seminal work of Kazhdan and Patterson \cite{MR743816}.
The main goal is to extend the Bernstein--Zelevinsky classification to these groups.

Grosso modo, the classification scheme is similar to the linear case.
(See Theorem \ref{thm: main} for the precise statement.)
However, there are some subtle differences.
The main technical difficulty is to extend the notion of the Bernstein--Zelevinsky product of representations of $\GL$,
namely, the parabolic induction of the tensor product $\pi_1\otimes\dots\otimes\pi_k$
of representations $\pi_i$ of $\GL_{n_i}(F)$, viewed as a representation of the Levi subgroup of type $(n_1,\dots,n_k)$
of $\GL_n(F)$, where $n=n_1+\dots+n_k$.
While parabolic induction still makes sense in the context of covering groups and enjoys similar properties as in the linear case, the covering groups $G_i$ of the blocks
$\GL_{n_i}(F)$ do not pairwise commute in the covering group $G$ of $\GL_n(F)$.
Therefore, the covering group $M$ of the Levi subgroup is not simply a central quotient of the product of the $G_i$'s.

This problem was first dealt with by Banks--Levy--Sepanski who explicated 2-cocycles for the coverings defined by
Matsumoto and showed that they are ``block-compatible'' upon restriction from $\SL_{n+1}$ to $\GL_n$ \cite{MR1670203}.
Their work was used by Kable, Mezo and Takeda to construct
the Bernstein--Zelevinsky product, or more precisely, the covering analogue of the tensor product -- the so-called
``metaplectic tensor product'', at least in the case of irreducible representations \cites{MR2060495, MR1862025, MR3442519, MR3673084}.
We follow their approach and consolidate it by introducing the notion of a \emph{special subgroup} (see \S\ref{sec: specpairs}).
This is a normal subgroup of finite index satisfying certain conditions.
The main feature is that the ``genuine'' representation theory of the ambient group and its special subgroup are essentially
the same. An equivalence of categories is fulfilled by the so-called ``Lagrangian induction'' (see \S\ref{sec: LI}).

A typical case of a special subgroup is the center of a Heisenberg group $H$ over a finite field.
By the Stone--von-Neumann theorem, there is a unique irreducible representation of $H$
with a given nontrivial central character $\psi$. In fact, the category of \emph{all} representations of $H$ with central
character $\psi$ is equivalent to the category of vector spaces. Such an equivalence can be realized
(non-uniquely) by taking the Jacquet module with respect to $(L,\chi)$ where $L$ is a choice of a maximal abelian subgroup of $H$ and $\chi$ is
an extension of $\psi$ to a character of $L$. (See \cite{MR2552002} for a choice-free equivalence of categories.)

Using the notion of a special subgroup, we can compare the representation theories of any two covering groups
(of the same base group) that admit isomorphic special subgroups -- see \S\ref{sec: wellmatch}.
Specializing to the case at hand, the upshot is a \emph{metaplectic tensor product multifunctor}
\[
\Reps_{\omega_1}(G_{n_1})\times\dots\times\Reps_{\omega_k}(G_{n_k})\rightarrow\Reps_\omega(M)
\]
of categories of representations, which is multiadditive and multiexact (\S\ref{sec: mtp}).
Here $\omega_i$ and $\omega$ are characters of the center of $G_i$ and $M$ respectively,
which satisfy a certain compatibility relation.
The notation $\Reps_\chi(H)$ stands for the Serre subcategory of representations of $H$ that admit a finite filtration
such that the center $Z(H)$ acts by $\chi$ on the associated graded object.

This gives rise to a metaplectic analogue of the Bernstein--Zelevinsky product.
An important property of the metaplectic tensor product and Bernstein--Zelevinsky product is their associativity.

Another technical difficulty is the extension of the basic irreducibility result of Olshanski in the corank one case \cite{MR0499010}
to the metaplectic context.
Olshanski's result is based on a computation of the residue of an intertwining operator.
The covering case is more delicate, but eventually the result is almost the same.
It is carried out in \S\ref{sec: olshan}.

The argument yields the existence of a reducibility point $s=s_\rho>0$ of $\rho\times\rho\abs{\cdot}^s$.
The uniqueness of $s_\rho$ is proved separately. It is possible to determine $s_\rho$ using a trace formula
comparison with the linear groups. However, we will not carry this out in this paper. Instead, we will treat $s_\rho$
as a black box. At any rate, given $s_\rho$, one can construct segment representations as in the linear case.

With these two ingredients, the rest of the argument of the classification is essentially the same as in the linear case, following
Zelevinsky's approach, refined and extended by M\'{\i}nguez and S\'{e}cherre \cite{MR3178433} (see \S\ref{sec: seg}).
We will only sketch it, as it is very similar to \cite{MR3573961}*{Appendix A}.
The main difference is that one has to keep track of the central characters.
Alternatively, one works with irreducible representations up to twist
by characters whose order divides the degree of the covering.

The third-named author would like to thank Max Gurevich and Fan Gao for helpful correspondence.

\section{Notation and preliminaries}

Let $G$ be a group. We use the following notation.
\begin{itemize}
\item $Z(G)$ is the center of $G$.
\item $[\cdot,\cdot]:G\times G\rightarrow G$ is the commutator map $[x,y]=xyx^{-1}y^{-1}$.
\item For any elements $x,y\in G$, $x^y=y^{-1}xy$ is the conjugation of $x$ by $y$.
Also, ${}^yx=x^{y^{-1}}=yxy^{-1}$, so that $x^{yz}=(x^y)^z$ and ${}^{yz}x={}^y({}^zx)$.
\item Given subgroups $G_1,G_2$ of $G$, $Z_{G_2}(G_1)$ is the centralizer of $G_1$ in $G_2$
and $[G_1,G_2]$ is the subgroup of $G$ generated by the set of commutators $\{[g_1,g_2]\mid g_1\in G_1, g_2\in G_2\}$.
\end{itemize}

We will consider $\ell$-groups as defined by Bernstein--Zelevinsky.
We refer to \cite{MR0425030} and \cite{MR2567785} for standard facts about $\ell$-groups.
To simplify the discussion, we assume that all $\ell$-groups considered here are countable at infinity, i.e. they are a union of countably many compact subsets.

By convention, a subgroup of a topological group will always mean a \emph{closed} subgroup, unless specified otherwise.

Let $\alpha$ be an automorphism of $G$. We denote by $\modulus_G\alpha>0$ the module of $\alpha$ given by
\[
\int_Gf(\alpha^{-1}g)\ dg=\modulus_G\alpha\int_G f(g)\ dg
\]
where $dg$ is a (left or right) Haar measure on $G$ (whose choice is unimportant).
In particular, if $g\in G$ normalizes a subgroup $H$ of $G$, we denote by $\modulus_H(g)$ the module of the automorphism $h\mapsto{}^gh$
of $H$.

By a representation of an $\ell$-group we will always mean a complex smooth representation.
(By the same token, by a character of an $\ell$-group we always mean a locally constant homomorphism to $\C^{\times}$.)
Denote by $\Reps(G)$ the category of representations of $G$
and by $\Reps^{\adm}(G)$ and $\Reps^{\fl}(G)$ the full subcategories of admissible representations, and representations
of finite length, respectively.
Let $(\pi,V)$ be a representation of $G$. (By abuse of notation, we sometimes write $\pi\in\Reps(G)$.)
We denote by $m\cdot \pi$ the direct sum of $m$ copies of $\pi$ and
by $\pi^{\vee}$ the contragredient of $\pi$.
For an element $x\in G$, we denote by $\pi^x$ the representation on $V$ given by $\pi^x(g)=\pi({}^xg)$.
Note that $\pi^{xy}=(\pi^x)^y$.
If $\pi$ is of finite length, we denote by $\JH(\pi)$ the Jordan-H\"older series of $\pi$ as a multiset,
and by $\soc(\pi)$ (resp., $\cos(\pi)$) the socle (resp., cosocle) of $\pi$,
that is, the maximal semi-simple subrepresentation (resp., quotient) of $\pi$.

We denote by $\Irr(G)$ the set of equivalence classes of irreducible representations of $G$
and by $\Irr^{\sqr}(G)$ the subset of irreducible square-integrable representations.

Suppose that $(\pi,V)$ is an irreducible, essentially square-integrable representation of $G$.
We denote by $d_\pi$ the formal degree of $\pi$. We view it as the Haar measure on $Z(G)\bs G$ satisfying
\begin{equation} \label{eq: Schur}
\int_{Z(G)\bs G}\sprod{\pi(g)v}{v^\vee}\sprod{u}{\pi^\vee(g)u^\vee}\ d_\pi g=\sprod{v}{u^\vee}\sprod{u}{v^\vee},\ \ \forall u,v\in V, u^\vee,v^\vee\in V^\vee
\end{equation}
where $\sprod{\cdot}{\cdot}$ is the standard pairing on $V\times V^\vee$.
It is often useful to replace $Z(G)$ by a cocompact subgroup $A$ thereof and view $d_\pi$ as a Haar measure on $A\bs G$,
which we denote by $d_\pi^{A\bs G}$.
The pushforward of $d_\pi^{A\bs G}$ under $A\bs G\rightarrow Z(G)\bs G$ is $d_\pi^{Z(G)\bs G}$.

Let $H$ be a subgroup of $G$. We denote by
\[
\Ind_H^G,\ind_H^G:\Reps(H)\rightarrow\Reps(G),
\]
the (unnormalized) full induction and compact induction functors and by
\[
\Res_H^G:\Reps(G)\rightarrow\Reps(H)
\]
or simply $\rest_H$, the restriction functor. For any $\pi\in\Reps(G)$, $\tau\in\Reps(H)$ we have
\begin{subequations}
\begin{equation}
\label{eq: contrind} \modulus_G\cdot(\ind_H^G\tau)^\vee\simeq\Ind_H^G(\modulus_H\cdot\tau^\vee)
\end{equation}
and by Frobenius reciprocity
\begin{equation}
\label{eq: Frobrec} \Hom_G(\pi,\Ind_H^G\tau)=\Hom_H(\Res_H^G\pi,\tau).\\
\end{equation}
If $H$ is open in $G$ (in which case $\modulus_G=\modulus_H$ on $H$), then we have in addition
\begin{gather}
\label{eq: Frobrec2} \Hom_G(\ind_H^G\tau,\pi)=\Hom_H(\tau,\Res_H^G\pi),\\
\label{eq: indtensor} \ind_H^G(\pi\rest_H\otimes\tau)\simeq\pi\otimes\ind_H^G\tau.
\end{gather}
\end{subequations}

Let $A$ be an abelian $\ell$-group and let $\chi$ be a character of $A$.
We say that a representation $(\pi,V)$ of $A$ is \emph{locally-$\chi$} if for every $v\in V$ there exists an integer $m\ge0$
such that $(\pi(a)-\chi(a))^mv=0$ for all $a\in A$.
We denote by $\Reps_\chi(A)$ the full subcategory of $\Reps(A)$ consisting of locally-$\chi$ representations.

Every representation $(\pi,V)$ of $A$ admits a unique maximal locally-$\chi$ subrepresentation, denoted $(\pi^{(\chi)},V^{(\chi)})$, namely
\[
V^{(\chi)}=\bigcup_{m\ge0}\bigcap_{a\in A}\Ker(\pi(a)-\chi(a))^m.
\]
The sum
\[
\sum_\chi V^{(\chi)}
\]
over all the characters $\chi$ of $A$ is direct and is called the locally finite part of $\pi$, denoted $\pi_{\lf}$.
We say that $\pi$ is locally finite if $\pi=\pi_{\lf}$.
We denote by $\Reps_{\lf}(A)$ the full subcategory of locally finite representations of $A$. We have a decomposition
\[
\Reps_{\lf}(A)=\prod_\chi\Reps_\chi(A)
\]
over the characters of $A$.

More generally, let $G$ be an $\ell$-group and $A\subset G$ an abelian normal subgroup.
(Often, but not always, $A$ will be central in $G$.)
Let $\pi$ be a representation of $G$ and denote by $\tau$ its restriction to $A$.
For any character $\chi$ of $A$, the space
\[
\pi^{(\chi)}:=\sum_{g\in G/A}\tau^{(\chi^g)}=\sum_{g\in G/A}\pi(g)(\tau^{(\chi)})
\]
is a subrepresentation of $\pi$ that depends only on the $G$-orbit of $\chi$ under conjugation.
We say that $\pi$ is locally-$\chi$ if $\pi^{(\chi)}=\pi$.
We denote by $\Reps_\chi(G)$ the full subcategory of $\Reps(G)$ consisting of locally-$\chi$ representations.
(We do not include $A$ in the notation as it is encoded in $\chi$.)
Of course, $\Reps_\chi(G)$ depends only on the $G$-orbit of $\chi$.

In general, the sum
\[
\sum_\chi\pi^{(\chi)}
\]
over the $G$-orbits of characters of $A$ is direct and is called the $A$-locally finite part of $\pi$, denoted $\pi_{A\hlf}$.
Note that $\Res_A^G\pi_{A\hlf}=\tau_{\lf}$.
We say that $\pi$ is $A$-locally finite if $\pi=\pi_{A\hlf}$, or equivalently, $\tau$ is locally finite.
We denote by $\Reps_{A\hlf}(G)$ the full subcategory of $A$-locally finite representations of $G$.
We have a decomposition
\[
\Reps_{A\hlf}(G)=\prod_\chi\Reps_\chi(G)
\]
over the $G$-orbits of characters of $A$.
The subcategories $\Reps_\chi(G)$ and $\Reps_{A\hlf}(G)$ of $\Reps(G)$ are closed under subobjects, quotients and extensions, i.e.,
they are Serre subcategories.

We note that if $B$ is a finite index subgroup of $A$ which is normal in $G$, then
\begin{subequations}
\begin{equation} \label{eq: lfAB}
	\Reps_{B\hlf}(G)=\Reps_{A\hlf}(G).
\end{equation}
Moreover, for any character $\psi$ of $B$ we have
\begin{equation} \label{eq: extpsi}
	\Reps_\psi(G)=\oplus_\chi\Reps_\chi(G)
\end{equation}
\end{subequations}
where $\chi$ ranges over the $G$-orbits of the characters of $A$ that extend a character in the $G$-orbit of $\psi$.

In the special case where $A=Z(G)$, by Schur's lemma every representation of finite length is $A$-locally finite.

\section{Clifford theory}

\subsection{General settings}

We recall some standard results from Clifford theory.
Throughout this section, let $G$ be an $\ell$-group and let $H$ be a normal subgroup of finite index.
Let $\Gamma$ be the finite quotient group $G/H$.

By Mackey's formula, for any $\tau\in\Reps(H)$ we have
\begin{equation} \label{eq: Mackey}
\Res_H^G(\Ind_H^G\tau)=\oplus_{\gamma\in\Gamma}\tau^\gamma.
\end{equation}

For any representation $\tau$ of $H$ we denote by
\[
G_\tau=\{g\in G\mid\tau^g\simeq\tau\}
\]
the stabilizer of $\tau$ in $G$. Thus, $G_\tau$ is a subgroup of $G$ containing $H$.

Recall that by definition, a representation is completely reducible if every subrepresentation admits a complement.

\begin{lemma}\cite{MR0425030}*{Lemma 2.9} \label{lem: GHresmult}
\begin{enumerate}
\item Let $\pi$ be a representation of $G$. Then,
\begin{enumerate}
\item $\pi$ is completely reducible if and only if $\Res_H^G\pi$ is completely reducible.
\item $\pi$ is of finite length if and only if $\Res_H^G\pi$ is of finite length.
\end{enumerate}
\item Let $\tau$ be a representation of $H$. Then,
\begin{enumerate}
\item $\tau$ is completely reducible if and only if $\Ind_H^G\tau$ is completely reducible.
\item $\tau$ is of finite length if and only if $\Ind_H^G\tau$ is of finite length.
\end{enumerate}
\end{enumerate}
\end{lemma}

\begin{remark}\label{remarksubpairspecial}
The second part of Lemma \ref{lem: GHresmult} is not explicitly stated in \cite{MR0425030} but it follows from the first part by \eqref{eq: Mackey}.
In fact, the lemma holds without the assumption that $H$ is normal. One reduces to this case by
considering a finite index normal subgroup of $G$ contained in $H$.
\end{remark}

\begin{lemma} \label{lem: distcent}
Suppose that $L$ is a subgroup of $Z(H)$ that is normal in $G$. Let $\psi$ be a character of $L$.
Assume that $G_\psi=H$.
Then, the induction functor defines an equivalence of categories
\begin{equation} \label{eq: Indhere}
\Ind_H^G:\Reps_\psi(H)\rightarrow\Reps_\psi(G).
\end{equation}
An inverse is given by
\[
\pi\mapsto(\Res_H^G\pi)^{(\psi)}.
\]
Moreover, if $\pi\in\Reps_\psi(G)$ and $\tau=\Res_H^G\pi$, then
\begin{equation} \label{eq: restdecomp}
\tau=\oplus_{g\in \Gamma}\tau^{(\psi^g)}=\oplus_{g\in \Gamma}\pi(g)(\tau^{(\psi)}).
\end{equation}
Finally, suppose that $L\cap Z(G)$ is cocompact in $Z(H)$.
Let $\sigma\in\Irr_\psi(H)$ and $\pi=\Ind_H^G\sigma$.
Then, $\sigma$ is essentially square-integrable if and only if $\pi$ is essentially square-integrable. In this case,
the formal degree $d_\sigma^{(L\cap Z(G))\bs H}$ is the restriction of the formal degree $d_\pi^{(L\cap Z(G))\bs G}$.
\end{lemma}

\begin{proof}
By assumption, the subcategories $\Reps_{\psi^g}(H)$, $g\in \Gamma$ of $\Reps(H)$ are pairwise disjoint.
If $\tau\in\Reps_\psi(H)$, then by \eqref{eq: Mackey} we have $\Ind_H^G(\tau)\in\Reps_\psi(G)$ and
\[
(\Res_H^G\Ind_H^G\tau)^{(\psi)}\simeq\tau.
\]
Let $\pi\in\Reps_\psi(G)$ and $\tau=\Res_H^G\pi$.
The relation \eqref{eq: restdecomp} follows from the assumption on $\psi$.

Moreover, it is easy to see that the morphism
$\pi\rightarrow\Ind_H^G\tau^{(\psi)}$ corresponding under \eqref{eq: Frobrec} to the projection
$\tau\rightarrow\tau^{(\psi)}$, is an isomorphism.
(Up to a scalar, its inverse is the map corresponding to the embedding
$\tau^{(\psi)}\hookrightarrow\tau$ under \eqref{eq: Frobrec2}.)

Finally, for any $\sigma\in\Reps(H)$ we can realize any matrix coefficient of $\sigma$ (extended by $0$
outside $H$) as a matrix coefficient of $\Ind_H^G\sigma$. The last part follows.
\end{proof}

In the opposite extreme we have the following result.

\begin{lemma} \label{lem: centext}
Assume that $G=HL$ with $L\le Z(G)$. Let $\psi$ be a character of $L\cap H$.
Then,
\begin{enumerate}
\item \label{part: ext}
For any extension $\chi$ of $\psi$ to $L$, the restriction functor $\Res_H^G$ defines an isomorphism (and in particular, an equivalence) of categories
\[
\Reps_\chi(G)\rightarrow\Reps_\psi(H).
\]
Thus, every $\tau\in\Reps_\psi(H)$ admits a unique extension $\ex{\tau}{\chi}$ to a
locally-$\chi$ representation of $G$ (on the same space of $\tau$).
Moreover,
\begin{equation} \label{eq: tauxhicont}
(\ex{\tau}{\chi})^\vee=\ex{\tau^\vee}{\chi^{-1}}.
\end{equation}
\item \label{part: chi'}
If $\chi'$ is another character of $L$ extending $\psi$, then
\[
\ex{\tau}{\chi'}=\ex{\tau}{\chi}\cdot\mu
\]
where $\mu=\chi'\chi^{-1}$ (viewed as a character of $\Gamma\simeq L/(L\cap H)$).
\item \label{part: restau} For any $\tau\in\Reps_\psi(H)$ we have a natural isomorphism
\begin{equation} \label{eq: totind}
\Ind_H^G\tau\simeq\oplus_{\chi}\ex{\tau}{\chi}
\end{equation}
where $\chi$ ranges over the extensions of $\psi$ to $L$ (whose number is $[G:H]$).
\item \label{part: sqrintext} Suppose that $L$ is cocompact in $Z(G)$. Let $\tau\in\Irr_\psi(H)$.
Then, $\tau$ is essentially square-integrable if and only if $\ex{\tau}{\chi}$ is essentially square-integrable,
in which case $d_\tau^{(H\cap L)\bs H}=d_{\ex{\tau}{\chi}}^{L\bs G}$ under the isomorphism $L\bs G\simeq (H\cap L)\bs H$.
\end{enumerate}
\end{lemma}

\begin{proof}
We construct the inverse
\[
\Reps_\psi(H)\rightarrow\Reps_\chi(G).
\]

For any vector space $V$ and a nonzero integer $k$, the map $T\mapsto T^k$
is a bijection on the set of linear operators on $V$ such that $T-\id_V$ is locally nilpotent.
Its inverse, denoted $T\mapsto T^{1/k}$, is given by the Taylor series of $x^{1/k}$ around $1$. Thus,
\[
T^{1/k}v=\sum_{n=0}^\infty\binom{\frac1k}n(T-\id_V)^nv,\ \ v\in V,
\]
where only finitely many terms are nonzero for any $v$.

By assumption, $\Gamma\simeq L/(H\cap L)$.
Let $(\tau,V)\in\Reps_\psi(H)$. We extend $\tau$ to $G$ by setting
\begin{equation} \label{def: tau[chi]}
\ex{\tau}{\chi}(a)=\chi(a)(\psi(a^d)^{-1}\tau(a^d))^{1/d},\ \ a\in L,
\end{equation}
where $d\ne0$ is such that $a^d\in H$.
It is easy to see that $\ex{\tau}\chi$ is well defined and it is the unique extension of $\tau$ to a locally-$\chi$ representation of $G$.
Moreover, for any $\tau,\tau'\in\Reps_\psi(H)$ we have
\[
\Hom_G(\ex{\tau}\chi,\ex{\tau'}\chi)=\Hom_H(\tau,\tau').
\]
The relation \eqref{eq: tauxhicont} is clear.
This proves part \ref{part: ext}. Part \ref{part: chi'} is clear.

Let $\pi$ be a representation of $G$.
Since $H$ is normal in $G$ and $\Gamma$ is abelian, we have
\[
\Ind_H^G(\pi\rest_H)\simeq\oplus_{\mu}\pi\cdot\mu,
\]
where $\mu$ ranges over the characters of $\Gamma$.
(This follows from \eqref{eq: indtensor} by taking $\tau$ to be the identity one-dimensional representation.)
It follows that if $\tau\in\Reps_\psi(H)$, then the decomposition of $\Ind_H^G\tau$ according to \eqref{eq: extpsi}
(with respect to $B=H\cap L\subset A=L$) is given by
\[
\Ind_H^G\tau\simeq\oplus_{\chi}\ex{\tau}{\chi}
\]
where $\chi$ ranges over the extensions of $\psi$ to a character of $L$.
This proves part \ref{part: restau}.

Finally, part \ref{part: sqrintext} is straightforward.
\end{proof}	

\begin{remark}
Lemmas \ref{lem: distcent} and \ref{lem: centext} are dual in some sense.
In Lemma \ref{lem: distcent}, the induction restricts to an equivalence of categories once we fix $\psi$
and the restriction is given as a direct sum.
In Lemma \ref{lem: centext}, the roles of induction and restriction are interchanged.
\end{remark}

\subsection{Groups of Heisenberg type} \label{sec: Heisen}
Consider now the following situation.
Let $N$ be an $\ell$-group and let $A$ be a finite cyclic subgroup of $Z(N)$.
We will say that the pair $(N,A)$ (or simply $N$ itself) is of Heisenberg type if $[N,N]\subset A$ and the group $X=N/Z(N)$ is finite.

\begin{lemma} \label{lem: heisen}
Suppose that $(N,A)$ is of Heisenberg type. Then,
\begin{enumerate}
\item The group $X=N/Z(N)$ is abelian.
\item \label{part: commiso} The commutator induces a bimultiplicative, alternate and non-degenerate pairing
\[
[\cdot,\cdot]:X\times X\rightarrow A
\]
and hence a group isomorphism
\[
\iota:X\rightarrow\Hom(X,A)\simeq\PD{X},\ \ x\mapsto [x,\cdot],
\]
where $\PD{X}$ is the Pontryagin dual of $X$
\end{enumerate}
Moreover, let $L$ be a subgroup of $N$ containing $Z(N)$. Then,
\begin{enumerate}[resume]
\item The image of $L/Z(N)$ under $\iota$ is $\Hom(N/Z_N(L),A)$, i.e., the annihilator of $Z_N(L)/Z(N)$ in $\PD{X}$.
In particular,
\[
[L:Z(N)]=[N:Z_N(L)]
\]
and
\begin{equation} \label{eq: dcent}
Z_N(Z_N(L))=L.
\end{equation}
\item Suppose that there exists a genuine character $\chi$ of $L$, i.e., such that $\chi\rest_A$ is faithful.
Then, $N_\chi=Z_N(L)$. In particular, $L$ is abelian.

\item \label{part: st} Suppose that $L$ is abelian and let $\chi$ be a genuine character of $L$.
Then, $N/Z_N(L)$ acts simply transitively (under conjugation)
on the set of characters of $L$ that extend $\chi\rest_{Z(N)}$.
Moreover, the following conditions are equivalent.
\begin{enumerate}
\item $[L:Z(N)]=[N:L]$.
\item $Z_{N}(L)=L$.
\item $L$ is a maximal abelian subgroup of $N$.
\item The homomorphism $L/Z(N)\rightarrow\Hom(N/L,A)$, $x\mapsto [x,\cdot]$ is an isomorphism.
\end{enumerate}
In this case, $L$ is called a \emph{Lagrangian subgroup} of $N$.

In particular, a Lagrangian subgroup of $N$ exists and $[N:Z(N)]$ is a perfect square.

\item \label{part: SvN}
Let $\psi$ be a genuine character of $Z(N)$.
Suppose that $L$ is a Lagrangian subgroup of $N$ and $\chi$ is an extension of $\psi$ to $L$.
Then, the functor
\[
\Reps_\psi(Z(N))\rightarrow\Reps_\psi(N),\ \ \tau\mapsto\Ind_L^N\ex{\tau}{\chi}
\]
is an equivalence of categories.
\end{enumerate}
\end{lemma}

\begin{proof}
This is elementary.

For part \ref{part: commiso}, observe that since $\iota$ is (clearly) injective,
the order of every element of $X$ divides the order of $A$. Hence, by choosing an embedding of $A$ in $\C^*$, we may identify
$\Hom(X,A)$ with $\PD{X}$. Since $X$ is finite, $\iota$ is surjective.

For part \ref{part: st},  observe that $\chi^x\chi^{-1}=\chi([x,\cdot])$, $x\in N/Z_N(L)$ ranges over the characters of $L/Z(N)$.

Part \ref{part: SvN} follows by combining Lemma \ref{lem: centext} for $Z(G)\le L$ and Lemma \ref{lem: distcent}
for $L\le N$ -- the latter is applicable since $N_\chi=L$.

The rest of the statements are straightforward.
\end{proof}

\begin{remark}
The last part of the previous lemma is a version of the classical Stone--von-Nuemann theorem.
Note that the functor depends on the choice of $L$ and $\chi$ although its domain and codomain do not.

Note that if $N$ is finite, then the forgetful functor from $\Reps_\psi(Z(N))$ to the category of vector spaces
(i.e., representations of the trivial group) is an isomorphism of categories.
\end{remark}

\subsection{Special pairs} \label{sec: specpairs}

Suppose that $\bas{G}$ is an $\ell$-group and $A$ is a finite cyclic group.

By an $A$-covering of $\bas{G}$ we mean a central topological extension of $\bas{G}$ by $A$, i.e.
a short exact sequence of topological groups\footnote{In particular, $\pr$ is continuous and open.}
\[
1\rightarrow A\rightarrow G\xrightarrow{\pr}\bas{G}\rightarrow 1
\]
where $A$ is central in $G$. This automatically implies that $G$ is an $\ell$-group and every sufficiently small
open subgroup of $\bas{G}$ splits.\footnote{Indeed, since $\pr$ is continuous, open and its fibers are finite of constant size,
it is easy to see that $\pr$ is a topological covering map. Therefore, $G$ is an $\ell$-space, and hence an $\ell$-group.}

By general theory, the isomorphism classes of $A$-coverings of $\bas{G}$ are classified by the second cohomology group $H^2(G,A)$ defined in terms of
Borel cochains. (See \cite{MR414775} for more details.)\footnote{By the above, in the case at hand
we can use locally constant cochains instead.}

For the rest of the section we assume that $G$ is an $A$-covering of $\bas{G}$.
By our convention, if $\bas{H}$ is a subgroup of $\bas{G}$, we will denote by $H$ the preimage of $\bas{H}$
under $\pr$. We say that a character $\chi$ of $H$ is genuine if its restriction to $A$ is faithful.

We would like to enhance the discussion of \S\ref{sec: Heisen} to more general groups.

\begin{definition}
Let $\bas{H}$ be a finite index normal subgroup of $\bas{G}$.
We say that $H$ is \emph{special} if there exists a subgroup $\bas{N}$ of $Z(\bas{G})$
such that $N\le Z_G(H)$ and
\begin{equation}
\label{eq: zgzhnh} Z_G(H\cap N)=NH.
\end{equation}
In this case, we will also say that the pair $(H,N)$ is special.
\end{definition}

For instance, if $(N,A)$ is a Heisenberg pair, then $(Z(N),N)$ is special pair in $N$.

Note that in general, $N$ is not uniquely determined by $H$, but given $H$ there exists $N_{\max}$ such that $N\le N_{\max}$
for any $N$ such that $(H,N)$ is special (see below).

As we will soon see, if $H$ is special, then the representation theories of $G$ and $H$ are essentially the same (Proposition \ref{prop: lind}).

We start with some basic properties of special pairs.

\begin{lemma} \label{lem: specp}
Suppose that $(H,N)$ is a special pair and $N'$ is a subgroup of $N$. Then,
\begin{enumerate}
\item \label{part: hnzhzn} $H\cap N=Z(H)\cap Z(N)$.
\item \label{part: nzn} The group $N$ (or more precisely, the pair $(N,A)$) is of Heisenberg type.
In particular, $N/Z(N)$ is a finite abelian group whose order is a perfect square.
\item \label{part: nmax} Let $N_{\max}=Z_{\pr^{-1}(Z(\bas{G}))}(H)$.
Then, $N\le N_{\max}$ and $(H,N'')$ is special for any $N\le N''\le N_{\max}$.
\item \label{part: commd}
The commutator defines a non-degenerate bimultiplicative pairing
\[
G/Z_G(N')\times N'/Z_{N'}(G)\rightarrow A.
\]
Consequently, the groups $G/Z_G(N')$ and $N'/Z_{N'}(G)$ are finite abelian groups of exponent dividing $\abs{A}$,
that are in Pontryagin duality.
\item If $N'\supset N\cap H$, then
\begin{subequations}
\begin{equation} \label{eq: zgn'}
Z_G(N')=Z_N(N')H.
\end{equation}
In particular, (denoting the Pontryagin dual of a group $X$ by $\PD{X}$)
\begin{align}
\label{eq: zgznnh} Z_G(Z(N))&=NH,\\
\label{eq: zgnznh} Z_G(N)&=Z(N)H,\\
\label{eq: znhzn} \PD{G/Z(N)H}&\simeq N/Z_N(G),\\
\label{eq: pdgnh}
\begin{split}
\PD{G/NH}&\simeq Z(N)/Z_N(G)\\&\simeq (N\cap H)/Z_{N\cap H}(G).
\end{split}
\end{align}
Thus,
\begin{align}
\label{eq: znznghn} Z(N)&=Z_N(G)(N\cap H),\\
\label{eq: zgnzngh} Z_G(N)&=Z_N(G)H,\\
\label{eq: zgzngzhg} Z(G)&=Z_N(G)Z_H(G).
\end{align}
\end{subequations}
\item \label{part: all exsts}
Suppose that $N'$ contains $A$ and that there exists a genuine character $\chi$ of $N'$.
Then, $G_\chi=Z_G(N')$. In particular, $N'$ is abelian.
Denote by $\psi$ the restriction of $\chi$ to $Z_{N'}(G)$.
Then, the $G$-orbit of $\chi$ consists of the characters of $N'$ that extend $\psi$. Moreover,
\begin{equation} \label{eq: repspsichi}
\Reps_{\chi}(G)=\Reps_{\psi}(G).
\end{equation}
\item \label{part: stabtau}
Suppose that $\tau\in\Reps_\psi(H)$ where $\psi$ is a genuine character of $H\cap N$.
Then,
\begin{equation} \label{eq: stabtau}
G_\tau=HN.
\end{equation}
\item \label{part: hered}
Let $G'$ be a subgroup of $G$ containing $H$ (resp., $N$). Then,
\begin{equation} \label{eq: hered}
\text{$(H,N\cap G')$ (resp., $(H\cap G',N)$) is a special pair in $G'$.}
\end{equation}
\end{enumerate}
\end{lemma}

\begin{proof}

Parts \ref{part: hnzhzn} and \ref{part: nzn} are clear. Note that $N/Z(N)\simeq NH/Z(N)H$ is finite.

Clearly, $N\le N_{\max}$. On the other hand, if $N''\ge N$ then $N''$ satisfies \eqref{eq: zgzhnh}
while if $N''\le N_{\max}$ then $N''$ centralizes $H$ and $\pr(N'')$ is central in $\bas{G}$.
Part \ref{part: nmax} follows.

Part \ref{part: commd} is also straightforward. Note that $G/Z_G(N')$ is finite since $H\le Z_G(N')$.

The relation \eqref{eq: zgn'} holds since $H\le Z_G(N')\le Z_G(N\cap H)=NH$;
\eqref{eq: znhzn} and \eqref{eq: pdgnh} follow from part \ref{part: commd}, \eqref{eq: zgznnh}, \eqref{eq: zgnznh} and \eqref{eq: zgzhnh};
\eqref{eq: znznghn} follows from \eqref{eq: pdgnh}; \eqref{eq: zgnzngh} follows from \eqref{eq: zgnznh} and \eqref{eq: znznghn};
\eqref{eq: zgzngzhg} follows from \eqref{eq: zgnzngh}.

Part \ref{part: all exsts} follows from the fact that $[N,G]\subset A$ and $\chi^g\chi^{-1}=\chi([g,\cdot])$, $g\in G/Z_G(N')$
ranges over all characters of $N'/Z_{N'}(G)$, because of the group isomorphism
\[
G/Z_G(N')\rightarrow\Hom(N'/Z_{N'}(G),A),\ \ g\mapsto[g,\cdot].
\]
For \eqref{eq: repspsichi} we use \eqref{eq: extpsi}.

Part \ref{part: stabtau} follows from the fact that $G_\psi=NH$.

Part \ref{part: hered} is straightforward.
\end{proof}

\subsection{Lagrangian induction} \label{sec: LI}
Assume that $H$ is a special group in $G$.
The main result of this section is to relate the representation theory of $G$ and $H$.

Given two subgroups $H_1,H_2$ of $G$, we say that characters $\chi_i$ of $H_i$, $i=1,2$
are \emph{consistent} if they agree on the intersection $H_1\cap H_2$.
In the case where $H_1$ commutes with $H_2$, this condition implies that the character $\chi_1\chi_2$ of the group $H_1H_2$
is well defined (and extends both $\chi_i$).

\begin{proposition} \label{prop: lind}
Assume that $(H,N)$ is a special pair in $G$ with $[N:Z(N)]=d^2$, $d\ge0$.
Let $\chi$ and $\psi$ be consistent genuine characters of $N\cap Z(G)=Z_N(G)$ and $N\cap H$
and denote by $\varphi$ their common restriction to $Z_{N\cap H}(G)$.
Then,
\begin{enumerate}
\item There is an equivalence of categories (``Lagrangian induction'')
\begin{equation} \label{def: LIND}
\LInd_{H,\psi}^{G,\chi}=\LInd_{H,\psi,N}^{G,\chi}:\Reps_\psi(H)\rightarrow\Reps_\chi(G).
\end{equation}
Up to natural equivalence, this functor does not depend on additional choices.
\item \label{part: indcmp} For any $\tau\in\Reps_\psi(H)$ we have a natural isomorphism
\[
\Ind_H^G\tau\simeq d\cdot\oplus_{\chi'}\LInd_{H,\psi}^{G,\chi'}\tau,
\]
where $\chi'$ ranges over the characters of $N\cap Z(G)$ that are consistent with $\psi$
(the number of which is $[Z_N(G):Z_{N\cap H}(G)]=[Z(N):N\cap H]$).
\item \label{part: resdcmp} For any $\pi\in\Reps_\chi(G)$ we have a natural isomorphism,
\[
\Res_H^G\pi\simeq d\cdot\oplus_{\psi'}\LRes_{H,\psi'}^{G,\chi}\pi,
\]
where $\psi'$ ranges over the characters of $N\cap H$ that are consistent with $\chi$ (the number of which is
$[N\cap H:Z_{N\cap H}(G)]=[Z(N):Z_N(G)]$) and
$\LRes_{H,\psi'}^{G,\chi}$ (``Lagrangian restriction'') is an inverse to $\LInd_{H,\psi'}^{G,\chi}$.
\item \label{part: contr} We have $(\LInd_{H,\psi}^{G,\chi}\tau)^\vee=\LInd_{H,\psi^{-1}}^{G,\chi^{-1}}\tau^\vee$ for any
$\tau\in\Reps_\psi(H)$.
\item \label{part: equconj}
$\LInd_{H,\psi^g}^{G,\chi}(\tau^g)\simeq\LInd_{H,\psi}^{G,\chi}\tau$ for any $\tau\in\Reps_\psi(H)$ and $g\in G$.
Thus, for any $\pi\in\Irr_\chi(G)$, the equivalence classes of $\LRes_{H,\psi'}^{G,\chi}(\pi)$,
where $\psi'$ ranges over the characters of $N\cap H$ that extend $\varphi$ (i.e., are consistent with $\chi$), form
a $G/NH$-orbit in $\Irr_\varphi(H)$ under conjugation.

\item \label{part: twistw} For any character $\omega$ of $\bas{G}$ we have
\begin{equation} \label{eq: twistomega}
\LInd_{H,\psi\omega\rest_H}^{G,\chi\omega\rest_{Z_N(G)}}(\tau\cdot\omega\rest_H)\simeq(\LInd_{H,\psi}^{G,\chi}\tau)\cdot\omega
\end{equation}
for any $\tau\in\Reps_\psi(H)$.
In particular, if $\omega$ is trivial on $HZ(N)=HZ_N(G)$, then
\[
(\LInd_{H,\psi}^{G,\chi}(\tau))\cdot\omega\simeq\LInd_{H,\psi}^{G,\chi}(\tau).
\]
Thus, for any $\pi\in\Reps_\chi(G)$, up to isomorphism $\pi\cdot\omega$ depends only
on the restriction of $\omega$ to $HZ(N)$.

\item \label{part: twistab}
Suppose that $\Gamma=G/H$ is abelian.
Then, for any $\tau\in\Irr_\psi(H)$, the equivalence classes of $\LInd_{H,\psi}^{G,\chi'}(\tau)$,
where $\chi'$ ranges over the characters of $Z_N(G)$ that extend $\varphi$ (i.e., are consistent with $\psi$), form
a $\PD{\Gamma}$-orbit in $\Irr_\varphi(G)$ under twisting.

\item \label{part: corres}
Suppose that $\Gamma$ is abelian.
Then, we have a natural bijection between the $\Gamma$-orbits in $\Irr_\varphi(H)$ under conjugation and
the $\PD{\Gamma}$-orbits in $\Irr_\varphi(G)$ under twisting.

\item \label{part: fd}
Suppose that $N\cap H$ is cocompact in $Z(H)$ and let $\tau\in\Irr_\psi(H)$ and $\pi=\LInd_{H,\psi,N}^{G,\chi}\tau$.
Then, $\pi$ is essentially square-integrable if and only if $\tau$ is essentially square-integrable. In this case,
\begin{subequations}
\begin{equation} \label{eq: fdlind1}
d_\tau^{(H\cap N)\bs H}=d^{-1}\cdot p_*(d_\pi^{Z_N(G)\bs G}\rest_{Z_N(G)\bs NH})
\end{equation}
where $p_*$ is the pushforward of measures with respect to the projection
$p:Z_N(G)\bs NH\rightarrow N\bs NH\simeq (H\cap N)\bs H$ and
$\rest$ denotes restriction of measures.
Equivalently,
\begin{equation} \label{eq: fdlind2}
d_\tau^{Z_{H\cap N}(G)\bs H}=
([N:N\cap H][Z_N(G):Z_{N\cap H}(G)])^{\frac12}\cdot
d_\pi^{Z_{N\cap H}(G)\bs G}\rest_{Z_{H\cap N}(G)\bs H}.
\end{equation}
\end{subequations}
\end{enumerate}
\end{proposition}

\begin{proof}
The functor $\LInd_{H,\psi}^{G,\chi}$ is defined using a choice of a Lagrangian subgroup $L$ of $N$ and a character $\theta$ of $L$
that extends the character $\chi\psi$ of $Z(N)=(N\cap Z(G))(N\cap H)$.

Note that by Lemma \ref{lem: specp} part \ref{part: hnzhzn}, we have $L\cap H=N\cap H=Z(N)\cap H$ and $Z(LH)\supset L$.
By Lemma \ref{lem: centext}, we have an equivalence (and in fact, an isomorphism) of categories
\[
\Reps_\psi(H)\rightarrow\Reps_{\theta}(LH),\ \ \tau\mapsto\ex{\tau}{\theta}.
\]

On the other hand, since $\theta$ is genuine, $G_\theta=Z_G(L)=Z_N(L)H=LH$ by Lemma \ref{lem: specp} part \ref{part: all exsts}
and \eqref{eq: zgn'} since $L$ is Lagrangian.
Therefore, by Lemma \ref{lem: distcent} and \eqref{eq: repspsichi} we have an equivalence of categories
\[
\Ind_{LH}^G:\Reps_{\theta}(LH)\rightarrow\Reps_{\theta}(G)=\Reps_{\chi\psi}(G)=\Reps_\chi(G).
\]
The sought-after functor is the composition of the two equivalences above:
\[
\LInd_{H,\psi}^{G,\chi}\tau=\Ind_{LH}^G\ex{\tau}{\theta}.
\]
Recall that by Lemma \ref{lem: heisen} part \ref{part: st}, $N$  acts transitively (by conjugation) on the set of characters of $L$ extending $\chi\psi$.
Since
\[
\Ind_{LH}^G\ex{\tau}{\theta}\simeq\Ind_{LH}^G(\ex{\tau}{\theta})^g=\Ind_{LH}^G\ex{\tau}{\theta^g},\ \ g\in N
\]
we infer that up to a natural equivalence, $\LInd_{H,\psi}^{G,\chi}$ is independent of the choice of $\theta$.

To show independence of $L$, suppose that $L'$ is another Lagrangian subgroup of $N$ and
$\theta'$ is a character of $L'$ extending $\chi\psi$.
By the above, we may assume without loss of generality that $\theta'$ is consistent with $\theta$.
In this case, $\ex{\tau}{\theta}(\gamma)=\ex{\tau}{\theta'}(\gamma)$ for every $\gamma\in L\cap L'$, and
\[
T_{L,L'}:\Ind_{LH}^G\ex{\tau}{\theta}\rightarrow\Ind_{L'H}^G\ex{\tau}{\theta'},\ \
f\mapsto\sum_{\gamma\in (L\cap L')\bs L'}\ex{\tau}{\theta'}(\gamma)^{-1}(f(\gamma\cdot))
\]
defines an intertwining operator.
Moreover, it is clear from the definition \eqref{def: tau[chi]} that the operators $\ex{\tau}{\theta}(\gamma)$, $\gamma\in L$
and $\ex{\tau}{\theta'}(\gamma')$, $\gamma'\in L'$ pairwise commute. Therefore,
\begin{gather*}
T_{L',L}\circ T_{L,L'}f(g)=\sum_{\gamma'\in (L\cap L')\bs L'}\ex{\tau}{\theta'}(\gamma')^{-1}
\sum_{\gamma\in (L\cap L')\bs L}\ex{\tau}{\theta}(\gamma)^{-1}f(\gamma'\gamma g)\\=
\sum_{\gamma'\in (L\cap L')\bs L'}\ex{\tau}{\theta'}(\gamma')^{-1}
\sum_{\gamma\in (L\cap L')\bs L}\psi([\gamma,\gamma'])^{-1}f(\gamma'g)=
\#((L\cap L')\bs L)f(g).
\end{gather*}
Thus, up to natural equivalence, $\LInd_{H,\psi}^{G,\chi}$ is independent of the choice of $L$.

Finally, using \eqref{eq: totind},
\[
\Ind_H^G\tau=\Ind_{LH}^G\Ind_H^{LH}\tau=\oplus_{\chi'}\Ind_{LH}^G\ex{\tau}{\chi'}
\]
where $\chi'$ ranges over the characters of $L$ extending $\psi$.
As we have already noted, up to isomorphism $\Ind_{LH}^G\ex{\tau}{\chi'}$ depends only on the restriction of $\chi'$ to $Z(N)$.
Part \ref{part: indcmp} follows since $d=[L:Z(N)]$.

Part \ref{part: resdcmp} is similar:
an inverse $\LRes_{H,N}^{G,\chi}$ is given by $\Res_H^{LH}(\Res_{LH}^G\pi)^{(\theta)}$.

Part \ref{part: contr} is clear from \eqref{eq: contrind}, \eqref{eq: tauxhicont} and the construction.

The first statement of part \ref{part: equconj} is straightforward.
The other statement follows from Lemma \ref{lem: specp} part \ref{part: all exsts}.

The relation \eqref{eq: twistomega} is straightforward and it implies the rest of part \ref{part: twistw}.
Part \ref{part: twistab} follows from part \ref{part: twistw} and the fact that every character of $Z_N(G)/Z_{N\cap H}(G)$ can be extended to a character of $\Gamma$
(since $\Gamma$ is abelian by assumption).

Part \ref{part: corres} follows from parts \ref{part: equconj} and \ref{part: twistab}.

Before proving part \ref{part: fd} we make a straightforward remark.
\begin{subequations}
Suppose that we have a commutative triangle of $\ell$-groups
\[
\begin{tikzcd}
H_1\arrow[r,hook]\arrow[dr,twoheadrightarrow,"p'"]&H_2\arrow[d,twoheadrightarrow,"p"]\\
&G
\end{tikzcd}
\]
with $[H_2:H_1]<\infty$. Then, for a Haar measure $dh_2$ on $H_2$ we have
\begin{equation} \label{eq: 1iota2}
p'_*(dh_2\rest_{H_1})=[H_2:H_1]^{-1}\cdot p_*(dh_2) =[\Ker p:\Ker p']^{-1}\cdot p_*(dh_2).
\end{equation}
Dually, suppose that we have a commutative triangle of $\ell$-groups
\[
\begin{tikzcd}
G_1\arrow[r,twoheadrightarrow,"q"]&G_2\\
H\arrow[u,hook]\arrow[ur,hook]
\end{tikzcd}
\]
with finite $\Ker q$. Then, for a Haar measure $dg_1$ on $G_1$ we have
\begin{equation} \label{eq: 1p2}
(q_*(dg_1))\rest_H=(\#\Ker q)\cdot dg_1\rest_H.
\end{equation}
\end{subequations}

Suppose now that $N\cap H$ is cocompact in $Z(H)$ and let $\tau\in\Irr_\psi(H)$ and $\pi=\LInd_{H,\psi,N}^{G,\chi}\tau$.
Note that $Z_N(G)$ is cocompact in $Z(G)$ since by \eqref{eq: zgzngzhg}, $Z_N(G)\bs Z(G)\simeq Z_{N\cap H}(G)\bs Z_H(G)$ is a closed subgroup
of $(N\cap H)\bs Z(H)$.
Therefore, by Lemma \ref{lem: distcent} and Lemma \ref{lem: centext} part \ref{part: sqrintext}
the essential square-integrability of $\tau$, $\ex{\tau}{\theta}$ and $\pi$ are equivalent.
Moreover, in this case $d_{\ex{\tau}{\theta}}^{L\bs HL}=d_\tau^{(L\cap H)\bs H}$ under the isomorphism $L\bs HL\simeq (L\cap H)\bs H=(N\cap H)\bs H$
while $d_{\ex{\tau}{\theta}}^{Z_L(G)\bs HL}=d_\pi^{Z_L(G)\bs G}\rest_{Z_L(G)\bs HL}$.
Note that $Z_L(G)=Z_N(G)$.
Applying the relation \eqref{eq: 1iota2} to
\[
\begin{tikzcd}
HL/Z_N(G)\arrow[r,hook]\arrow[dr,twoheadrightarrow,"p'"]&HN/Z_N(G)\arrow[d,twoheadrightarrow,"p"]\\
&HN/N=HL/L
\end{tikzcd}
\]
we obtain \eqref{eq: fdlind1} since $[HN:HL]=[N:L]=d$. In order to deduce \eqref{eq: fdlind2}, consider the commutative diagram
\[
\begin{tikzcd}
G/Z_{N\cap H}(G)\arrow[r,twoheadrightarrow,"q"]&G/Z_N(G)\\
H/Z_{N\cap H}(G)\arrow[u,hook]\arrow[r,hook]\arrow[ur,hook]\arrow[rd,twoheadrightarrow,"p''"]&
HN/Z_N(G)\arrow[u,hook]\arrow[d,twoheadrightarrow,"p"]\\
&H/(N\cap H)=HN/N
\end{tikzcd}
\]
Using \eqref{eq: 1iota2} and \eqref{eq: 1p2}, for any Haar measure $dg$ on $G/Z_{N\cap H}(G)$ we have
\begin{gather*}
p_*((q_*(dg))\rest_{HN/Z_N(G)})=
[HN:Z_N(G)H]\cdot p''_*((q_*(dg))\rest_{H/Z_{N\cap H}(G)})=\\
[N:Z(N)]\cdot p''_*((q_*(dg))\rest_{H/Z_{N\cap H}(G)})=
[N:Z(N)][Z_N(G):Z_{H\cap N}(G)]\cdot p''_*(dg\rest_{H/Z_{H\cap N}(G)}).
\end{gather*}
Note that
\[
d^2=[N:Z(N)]=[N:Z_N(G)(H\cap N)]=\frac{[N:N\cap H]}{[Z_N(G)(N\cap H):N\cap H]}=
\frac{[N:N\cap H]}{[Z_N(G):Z_{N\cap H}(G)]}.
\]
The relation \eqref{eq: fdlind2} therefore follows from \eqref{eq: fdlind1}. The proof of the proposition is complete.
\end{proof}

\begin{remark} \label{rem: ggorbit}
Let $\Gamma'=Z_N(G)/Z_{N\cap H}(G)=Z_N(G)H/H\le\Gamma$.
Suppose that $\Gamma$ is abelian, so that the restriction map $\PD{\Gamma}\rightarrow\PD{\Gamma'}$ is surjective.
Let $\pi$ be a genuine irreducible representation of $G$ and $\omega\in\PD{\Gamma}$.
Then, up to isomorphism the twist of $\pi$ by $\omega$ depends only on the restriction of $\omega$ to $\Gamma'$.
Thus, we will use the notation $\pi\cdot\omega$ for $\omega\in\PD{\Gamma'}$.
The $\PD{\Gamma}$-orbit of $\pi$ in $\Irr(G)$ coincides with its $\PD{\Gamma'}$-orbit.
\end{remark}

\begin{remark}
As was pointed out above, up to natural equivalence the functor $\LInd_{H,\psi}^{G,\chi}$ does not depend on the choice of
a pair $(L,\theta)$ consisting of a Lagrangian $L$ of $N$ and a character $\theta$ of $L$ extending $\chi\psi$.
Although it will not be consequential for the purpose of this paper, it would be desirable to have a \emph{canonical} functor.
A natural way to do that would be to define for any two pairs $(L_i,\theta_i)$, $i=1,2$ as above (possibly with additional data)
a functorial isomorphism
\[
T_{(L_1,\theta_1)}^{(L_2,\theta_2)}:\Ind_{L_1H}^G\ex{\tau}{\theta_1}\rightarrow
\Ind_{L_2H}^G\ex{\tau}{\theta_2},\ \ \tau\in\Reps_\psi(H)
\]
such that for any three pairs $(L_i,\theta_i)$, $i=1,2,3$ we have
\[
T_{(L_1,\theta_1)}^{(L_3,\theta_3)}=T_{(L_2,\theta_2)}^{(L_3,\theta_3)}\circ T_{(L_1,\theta_1)}^{(L_2,\theta_2)}.
\]
A closely related problem was considered in \cite{MR2801175} (and the references therein), although the setup of [ibid.] unfortunately
excludes groups with nonabelian $2$-part.
\end{remark}

\subsection{Compatibility}

The following lemma is elementary.
\begin{lemma} \label{lem: diam}
Let $D$ be an abelian $\ell$-group. Let $A_1,A_2,B_1,B_2$ be subgroups of $D$ such that
\[
A_2\le A_1,\ B_2\le B_1,\ D=A_1B_2=A_2B_1,\ A_1\cap B_2=A_2\cap B_1(=A_2\cap B_2),.
\]
Let $\psi_2$ and $\chi_2$ be consistent characters of $A_2$ and $B_2$.
Denote by $X$ (resp., $Y$) the set of extensions of $\psi_2$ (resp., $\chi_2$) to a character of $A_1$ (resp., $B_1$).
Let $Z$ be the set of pairs $(\psi_1,\chi_1)\in X\times Y$ of consistent characters.
Then,
\begin{enumerate}
\item The set $Z$ is in bijection with set of extensions of $\chi_2\psi_2$ to a character of $D$.
\item The set $Z$ is the graph of a bijection between $X$ and $Y$.
\end{enumerate}
\end{lemma}

We continue to assume that $G$ is an $A$-covering of $\bas{G}$.
We have the following compatibility result.

\begin{lemma} \label{lem: compn1n2}
Suppose that $(H,N_1)$ and $(H,N_2)$ are two special pairs in $G$ such that $N_2\le N_1$. Then,
\begin{enumerate}
\item $N_1H=N_2H$, so that $N_1=(N_1\cap H)N_2$.
\item \label{part: zn1n1hzn2} $Z(N_1)=(N_2\cap Z(G))(N_1\cap H)=(N_1\cap Z(G))(N_2\cap H)$.
\item \label{part: zn1gzng2nhg} $Z_{N_1}(G)=Z_{N_2}(G)Z_{H\cap N_1}(G)$.
\item \label{part: n1n21}
Let $\psi_1$ and $\chi_1$ be consistent genuine characters of $N_1\cap H$ and $N_1\cap Z(G)$.
Let $\psi_2$ (resp., $\chi_2$) be the restriction of $\psi_1$ (resp., $\chi_1$) to $N_2\cap H$ (resp., $N_2\cap Z(G)$).
Then, $\LInd_{H,\psi_1,N_1}^{G,\chi_1}$ is the restriction of $\LInd_{H,\psi_2,N_2}^{G,\chi_2}$
to $\Reps_{\psi_1}(H)$.
\item \label{part: chi12}
Let $\psi_2$ and $\chi_2$ be consistent genuine characters of $N_2\cap H$ and $N_2\cap Z(G)$.
Denote by $X$ (resp., $Y$) the set of extensions of $\psi_2$ (resp., $\chi_2$) to a character of $N_1\cap H$ (resp., $N_1\cap Z(G)$).
Then, the set $Z$ of pairs $(\psi_1,\chi_1)\in X\times Y$ of consistent characters (which is in bijection with set of extensions
of $\chi_2\psi_2$ to a character of $Z(N_1)$) is the graph of a bijection between $X$ and $Y$.
Moreover, under the decompositions
\[
\Reps_{\psi_2}(H)=\oplus_{\psi_1\in X}\Reps_{\psi_1}(H),\ \ \Reps_{\chi_2}(G)=\oplus_{\chi_1\in Y}\Reps_{\chi_1}(G),
\]
(cf.\ \eqref{eq: extpsi}) we have
\[
\LInd_{H,\psi_2,N_2}^{G,\chi_2}=\oplus_{(\psi_1,\chi_1)\in Z}\LInd_{H,\psi_1,N_1}^{G,\chi_1}:\oplus_{\psi_1\in X}\Reps_{\psi_1}(H)
\rightarrow\oplus_{\chi_1\in Y}\Reps_{\chi_1}(G).
\]
\end{enumerate}
\end{lemma}

\begin{remark}
It follows from Lemma \ref{lem: compn1n2}, together with Lemma \ref{lem: specp} part \ref{part: nmax},
that for any special pair $(H,N)$, we can recover $\LInd_{H,\psi,N}^{G,\chi}$
from $\LInd_{H,\psi',N_{\max}}^{G,\chi'}$ as we vary over
the consistent characters $\psi'$ and $\chi'$ of $N_{\max}\cap H$ and $Z(G)$ extending $\psi$ and $\chi$.
For this reason, we often suppress the subgroup $N$ from the notation $\LInd_{H,\psi,N}^{G,\chi}$.
(In any case, the subgroups $N\cap H$ and $N\cap Z(G)$, and consequently $NH$ and $Z(N)$, are recovered from
the domains of $\psi$ and $\chi$.)
\end{remark}

\begin{proof}
First note that using \eqref{eq: znznghn},
\[
Z(N_2)=(N_2\cap H)Z_{N_2}(G)\le (N_1\cap H)Z_{N_1}(G)=Z(N_1).
\]
Also, since $N_1\le Z_G(H\cap N_2)=N_2H$, we have $N_1=N_2(N_1\cap H)$.
Together with \eqref{eq: znznghn} (for $N_2$) it follows that
\[
Z(N_1)=Z(N_2)(N_1\cap H)=(N_2\cap Z(G))(N_1\cap H).
\]
By \eqref{eq: pdgnh} we also have
\[
(N_1\cap H)/Z_{N_1\cap H}(G)\simeq (N_2\cap H)/Z_{N_2\cap H}(G)
\]
since $HN_1=HN_2$. Hence, $N_1\cap H\le (N_2\cap H)(N_1\cap Z(G))$.
Part \ref{part: zn1n1hzn2} follows once again from \eqref{eq: znznghn} (applied to $N_1$).

Part \ref{part: zn1gzng2nhg} follows from the relation $Z(G)=Z_{N_2}(G)Z_H(G)$ \eqref{eq: zgzngzhg}.

Let $L_2$ be a Lagrangian subgroup of $N_2$ and let $L_1$ be a Lagrangian subgroup of $N_1$ containing $L_2$.
By \eqref{eq: zgn'} we have
\[
L_1\le Z_G(L_2)=Z_{N_2}(L_2)H=L_2H.
\]
Hence, $HL_1=HL_2$.

Let $\chi_1$ and $\psi_1$ be consistent characters of $N_1\cap Z(G)$ and $N_1\cap H$.
Let $\theta_1$ be a character of $L_1$ extending the character $\chi_1\psi_1$ of $Z(N_1)$.
Let $\theta_2$ be the restriction of $\theta_1$ to $L_2$.
Then, for any $\tau\in\Reps_{\psi_1}(H)$ we have
\[
\ex{\tau}{\theta_1}=\ex{\tau}{\theta_2}
\]
as representations of $HL_1=HL_2$.
Part \ref{part: n1n21} now follows from the definition of $\LInd$.

Part \ref{part: chi12} follows from Lemma \ref{lem: diam} together with parts \ref{part: zn1n1hzn2} and \ref{part: n1n21}.
\end{proof}

Next, we study the relation between different special subgroups.

\begin{lemma} \label{lem: seesaw}
Let $(H_1,N_1)$, $(H_2,N_2)$ be two special pairs in $G$.
Assume that $H_1\le H_2$ and $N_2\le N_1$.

Let $N_1'=N_1\cap H_2$, $N_{1,2}'=N_1'N_2$, $N_1''=Z_{N_1}(H_2)=N_1\cap Z(H_2)$, $N_{1,2}''=N_1''N_2$.
Then,
\begin{enumerate}
\item \label{part: n1'n2} The groups $N_1'$ and $N_2$ commute and their intersection is $N_2\cap H_2=Z(N_1')\cap Z(N_2)$.
\item \label{part: spec1} $(H_1,N_1')$ is a special pair in $H_2$ with $H_1\cap N_1'=H_1\cap N_1$ and
$Z(N'_1)=N''_1(N_1\cap H_1)$.
\item \label{part: spec2} $(H_2,N_{1,2}'')$ is a special pair in $G$ with
\begin{subequations}
\begin{equation} \label{eq: n1''}
H_2\cap N_{1,2}''=N''_1=Z(H_2)\cap N'_1=Z_{N_1''}(G)(N_2\cap H_2)=Z_{N_1'}(G)(N_2\cap H_2)
\end{equation}
and
\begin{equation} \label{eq: zn12''}
Z_{N_{1,2}''}(G)=Z_{N_1}(G)=Z_{N_2}(G)Z_{N_1'}(G).
\end{equation}
\item \label{part: cois} $N'_{1,2}$ is coisotropic in $N_1$, i.e., $Z_{N_1}(N'_{1,2})=Z(N'_{1,2})$. Moreover,
\begin{equation} \label{eq: zn12'}
Z(N_{1,2}')=Z(N_1')Z(N_2)=Z(N_1)(N_2\cap H_2).
\end{equation}
\end{subequations}
\item \label{part: fnctrltn}
Let $\chi_1$ and $\psi_1$ be consistent genuine characters of $Z_{N_1}(G)$ and $N_1\cap H_1$.
Let $\psi_2$ be a character of $N_2\cap H_2$ that is consistent with the character $\chi_1\psi_1$ of $Z(N_1)=Z_{N_1}(G)(N_1\cap H_1)$.
Let $\chi'_1$ be the restriction of $\chi_1$ to $Z_{N'_1}(G)$ and let $\eta$ be the character $\chi'_1\psi_2$ of $N_1''$.
Then, we have a natural equivalence of functors
\[
\LInd_{H_1,\psi_1,N_1}^{G,\chi_1}=\LInd_{H_2,\eta,N''_{1,2}}^{G,\chi_1}\circ\LInd_{H_1,\psi_1,N_1'}^{H_2,\eta}:
\Reps_{\psi_1}(H_1)\rightarrow\Reps_{\chi_1}(G)
\]
Moreover, $\LInd_{H_2,\eta,N''_{1,2}}^{G,\chi_1}$ is the restriction of $\LInd_{H_2,\psi_2,N_2}^{G,\chi_2}$ to $\Reps_{\eta}(H_2)$, where
$\chi_2$ is the restriction of $\chi_1$ to $Z_{N_2}(G)$.
\item \label{part: n2h2zn1s} Suppose that $Z_{N_1'}(G)\le H_1$.
Then, $Z_{N_2\cap H_2}(N_1)=N_2\cap H_1$.
Therefore, if $\psi_i$ are consistent genuine characters of $H_i\cap N_i$, $i=1,2$ and
$\chi_1$ is a character of $Z_{N_1}(G)$ that is consistent with $\psi_1$,
then $\psi_2$ is consistent with $\chi_1\psi_1$. Moreover, in this case the character $\eta$ above
is the restriction of $\psi_1\psi_2$ to $N_1''=Z_{N_1'}(G)(N_2\cap H_2)\le(N_1\cap H_1)(N_2\cap H_2)$.
\end{enumerate}
\end{lemma}

\begin{proof}
Part \ref{part: n1'n2} is clear.

Part \ref{part: spec1} follows from \eqref{eq: hered} and \eqref{eq: znznghn}.

Clearly, $N''_{1,2}\le Z_G(H_2)$ and $\pr(N''_{1,2})\le Z(\bas{G})$ since $N''_{1,2}\le N_1$.
Also,
\[
Z_G(H_2\cap N''_{1,2})\le Z_G(N_2\cap H_2)=N_2H_2=N''_{1,2}H_2.
\]
Thus, $(H_2,N''_{1,2})$ is special.

The first equality in \eqref{eq: n1''} holds since $N_1''\le H_2$ and $H_2\cap N_2\le N_1''$.
The second equality is trivial.

Applying Lemma \ref{lem: specp} part \ref{part: commd} (with respect to $(H,N)=(H_1,N_1)$ and $N'=N_1''$) we have
\[
N_1''/Z_{N_1''}(G)\simeq\PD{G/Z_G(N_1'')}.
\]
Clearly, $N_2H_2\le Z_G(N_1'')$. On the other hand, since $N_1''\supset N_2\cap H_2$, we have $Z_G(N_1'')\le N_2H_2$.
Thus, $Z_G(N_1'')=N_2H_2$. Using \eqref{eq: pdgnh} (with respect to $(H,N)=(H_2,N_2)$) we conclude that
\[
N_1''/Z_{N_1''}(G)\simeq (N_2\cap H_2)/Z_{N_2\cap H_2}(G).
\]
Thus, $N_1''=Z_{N_1''}(G)(N_2\cap H_2)$. Finally, $Z_{N_1''}(G)=Z_{N_1'}(G)$.
This concludes the proof of \eqref{eq: n1''}.

To show \eqref{eq: zn12''}, recall that $Z(G)=Z_{N_2}(G)Z_{H_2}(G)$ by \eqref{eq: zgzngzhg}.
Hence,
\[
Z_{N_1}(G)=Z_{N_2}(G)Z_{N_1\cap H_2}(G)\le N''_{1,2}.
\]
This concludes the proof of part \ref{part: spec2}.

The first statement in part \ref{part: cois} holds since
\[
Z_{N_1}(N'_{1,2})\le Z_{N_1}(H_2\cap N_2)=N_1\cap Z_G(H_2\cap N_2)=N_1\cap N_2H_2=N'_{1,2}.
\]
The first equality in \eqref{eq: zn12'} holds since $N'_1$ and $N_2$ commute.
Clearly, $N_2\cap H_2\le Z(N_2)\le Z(N'_{1,2})$ and $Z(N_1)\le Z(N'_{1,2})$ by the above.
Conversely, $Z(N_2)=Z_{N_2}(G)(N_2\cap H_2)\le Z(N_1)(N_2\cap H_2)$ and by part \ref{part: spec1} and \eqref{eq: n1''},
\[
Z(N'_1)=N_1''(N_1\cap H_1)=Z_{N_1'}(G)(N_2\cap H_2)(N_1\cap H_1)\le Z(N_1)(N_2\cap H_2).
\]
Part \ref{part: cois} follows.

Let $L_1'$ be a Lagrangian subgroup of $N_1'$ and let $L_2$ be a Lagrangian subgroup of $N_2$.
(In particular, both $L_1'$ and $L_2$ contain $N_2\cap H_2$.)
Then, $L_1=L_1'L_2$ is a Lagrangian subgroup of $N'_{1,2}$, and hence of $N_1$ since $N'_{1,2}$ is coisotropic.
(Indeed, if $L$ is a Lagrangian subgroup of $N_1$ containing $L_1$, then
\[
L\le Z_{N_1}(L_1)\le Z_{N_1}(Z(N'_{1,2}))=Z_{N_1}(Z_{N_1}(N'_{1,2}))=N'_{1,2}
\]
by \eqref{eq: dcent}. Hence, $L=L_1$.)
Note that $H_1L_1\le H_2L_2$ and
\[
L_1\cap H_2=L_1'L_2\cap H_2=L_1'(L_2\cap H_2)=L_1'.
\]
Also, $L_1\supset Z(N'_{1,2})$.

Let $\chi=(\chi_1\psi_1)\cdot\psi_2$ be the character of $Z(N_{1,2}')=Z(N_1)(N_2\cap H_2)$ and let $\theta_1$ be an extension of $\chi$ to a character of $L_1$.
For any $\tau\in\Reps_{\psi_1}(H_1)$ consider
\[
\Ind_{H_1L_1}^G\ex{\tau}{\theta_1}=
\Ind_{H_2L_2}^G\Ind_{H_1L_1}^{H_2L_2}\ex{\tau}{\theta_1}.
\]
We claim that
\[
\Ind_{H_1L_1}^{H_2L_2}\ex{\tau}{\theta_1}=
\ex{\big(\Ind_{H_1L_1'}^{H_2}\ex{\tau}{\theta_1'}\big)}{\theta_2}
\]
where $\theta_1'$ (resp., $\theta_2$) is the restriction of $\theta_1$ to $L_1'$ (resp., $L_2$).
Indeed, it is enough to note that
\[
\Res_{H_2}^{H_2L_2}\Ind_{H_1L_1}^{H_2L_2}\ex{\tau}{\theta_1}=
\Ind_{H_1L_1'}^{H_2}\ex{\tau}{\theta_1'}
\]
since $H_2L_1=H_2L_2$ and
\[
H_2\cap H_1L_1=H_1(L_1\cap H_2)=H_1L_1'.
\]

It follows that
\[
\LInd_{H_2,\psi_2,N_2}^{G,\chi_2}\LInd_{H_1,\psi_1,N_1'}^{H_2,\eta}\tau=\LInd_{H_1,\psi_1,N_1}^{G,\chi_1}\tau.
\]
(Note that $N_1\cap H_1=N_1'\cap H_1$.)
Part \ref{part: fnctrltn} follows now from Lemma \ref{lem: compn1n2}.

Finally, note that $N_2\cap H_1\le N_2\cap H_2\cap Z(N_1)$. Conversely,
\[
N_2\cap H_2\cap Z(N_1)=N_2\cap H_2\cap Z_{N_1}(G)(N_1\cap H_1)=N_2\cap (N_1\cap H_1)Z_{N_1'}(G)
\]
which by our assumption is contained in $N_2\cap H_1$.
Part \ref{part: n2h2zn1s} follows.
\end{proof}

\subsection{Well-matched covering groups} \label{sec: wellmatch}

As before, suppose that $A$ is a finite cyclic group and $\bas{G}$ is an $\ell$-group.

We would like to compare, under suitable conditions, the representation theories of two $A$-coverings of $\bas{G}$.

\begin{definition}
Let
\[
\pr_i:G_i\rightarrow\bas{G},\ \ i=1,2
\]
be two $A$-coverings of $\bas{G}$.
We say that $G_1$ and $G_2$ are \emph{well matched} if $G_i$ admits a special pair $(H_i,N_i)$, $i=1,2$, such that the following conditions hold.
\begin{enumerate}
\item $\pr_1(H_1)=\pr_2(H_2)$, $\pr_1(N_1)=\pr_2(N_2)$ and $\pr_1(Z_{H_1\cap N_1}(G))=\pr_2(Z_{H_2\cap N_2}(G))$.
\item $H_1$ and $H_2$ are isomorphic as covering groups.
\end{enumerate}
In this case we will also say that the pairs $(H_i,N_i)$ are well matched.
\end{definition}

Assume that $(H_i,N_i)$ are well-matched special pairs in $G_i$, $i=1,2$.
We will write $\bas{H}=\pr_i(H_i)$ and $\bas{N}=\pr_i(N_i)$ and $\Gamma=\bas{G}/\bas{H}\simeq G_i/H_i$.
Fix an isomorphism of covering groups
\begin{equation} \label{def: iota12}
\iota_{H_1}^{H_2}:H_1\rightarrow H_2.
\end{equation}
This isomorphism induces a natural equivalence of categories
\begin{equation}\label{eqIH1H2}
I_{H_1}^{H_2}:\Reps(H_1)\rightarrow\Reps(H_2).
\end{equation}

\begin{definition} \
\begin{enumerate}
\item Let $\bas{B}$ be a subgroup of $\bas{H}$. Let $B_i=\pr_i^{-1}(\bas{B})$, $i=1,2$
so that $\iota_{H_1}^{H_2}(B_1)=B_2$.
We say that characters $\theta_i$ of $B_i$ are \emph{congruous} if $\theta_1=\theta_2\circ\iota_{H_1}^{H_2}$.

\item Let $\chi_i$ be genuine characters of $Z_{N_i}(G_i)$, $i=1,2$. We say that $\chi_1$ and $\chi_2$ are \emph{compatible}
if their restrictions to $Z_{H_i\cap N_i}(G_i)$ are congruous.
\end{enumerate}
\end{definition}

\begin{proposition} \label{prop: transferwm}
Suppose that $(H_i,N_i)$ are well-matched special pairs for $G_i$, $i=1,2$.
Let $\chi_i$ be compatible genuine characters of $Z_{N_i}(G_i)$, $i=1,2$.
Then, we have an equivalence of categories
\[
\trns_{G_1,H_1,\chi_1}^{G_2,H_2,\chi_2}:\Reps_{\chi_1}(G_1)\rightarrow\Reps_{\chi_2}(G_2).
\]
It satisfies
\[
(\trns_{G_1,H_1,\chi_1}^{G_2,H_2,\chi_2}\tau)^\vee\simeq\trns_{G_1,H_1,\chi_1^{-1}}^{G_2,H_2,\chi_2^{-1}}\tau^\vee,\ \ \tau\in\Reps_{\chi_1}(G_1)
\]
and for any character $\omega$ of $\bas{G}$
\[
(\trns_{G_1,H_1,\chi_1}^{G_2,H_2,\chi_2}\tau)\cdot\omega=\trns_{G_1,H_1,\chi_1\omega_1}^{G_2,H_2,\chi_2\omega_2}(\tau\cdot\omega)
\]
where $\omega_i$ is the restriction of $\omega\circ\pr_i$ to $Z_{N_i}(G)$.

Suppose that $\Gamma$ is abelian. Let $\varphi_i$ be congruous characters of $Z_{H_i\cap N_i}(G)$, $i=1,2$.
Then, we have a bijection between the sets of $\PD{\Gamma}$-orbits under twisting in $\Irr_{\varphi_i}(G_i)$, $i=1,2$.
\end{proposition}

\begin{proof}
Let $X$ be the set of pairs $(\psi_1,\psi_2)$ of congruous characters of $H_i\cap N_i$
such that $\psi_i$ and $\chi_i$ are consistent for $i=1,2$.
The set $X$ is non-empty by the condition on $\chi_i$. In fact, $X$ is the graph of a bijection
between the sets of characters of $H_i\cap N_i$ that are consistent with $\chi_i$, $i=1,2$.
By Lemma \ref{lem: specp} part \ref{part: all exsts} and \eqref{eq: zgzhnh}, $X$ is an orbit
of $\bas{G}/\bas{N}\bas{H}$ under diagonal action by conjugation.

Fixing $(\psi_1,\psi_2)\in X$ we define
\begin{equation} \label{def: trns}
\trns_{G_1,H_1,\chi_1}^{G_2,H_2,\chi_2}=\LInd_{H_2,\psi_2}^{G_2,\chi_2}\circ
I_{H_1,\psi_1}^{H_2,\psi_2}\circ\LRes_{H_1,\psi_1}^{G_1,\chi_1}
\end{equation}
where $I_{H_1,\psi_1}^{H_2,\psi_2}:\Reps_{\psi_1}(H_1)\rightarrow\Reps_{\psi_2}(H_2)$ is the restriction of the functor $I_{H_1}^{H_2}$ in \eqref{eqIH1H2} to $\Reps_{\psi_1}(H_1)$, which also induces an equivalence of categories.
In other words, we have a commutative diagram of equivalences of categories
\[
\begin{tikzcd}
\Reps_{\chi_1}(G_1) \arrow[d,"\LRes_{H_1,\psi_1}^{G_1,\chi_1}"'] \arrow[r, "\trns_{G_1,H_1,\chi_1}^{G_2,H_2,\chi_2}"] &
\Reps_{\chi_2}(G_2) \arrow[d,"\LRes_{H_2,\psi_2}^{G_2,\chi_2}"] \\
\Reps_{\psi_1}(H_1)\arrow[r,"I_{H_1,\psi_1}^{H_2,\psi_2}"] &\Reps_{\psi_2}(H_2)
\end{tikzcd}
\]
Using Proposition \ref{prop: lind} part \ref{part: equconj}, the functor $\trns_{G_1,H_1,\chi_1}^{G_2,H_2,\chi_2}$ is independent of the choice of $\psi_i$.

The claimed properties of $\trns_{G_1,H_1,\chi_1}^{G_2,H_2,\chi_2}$ follow directly from the corresponding properties
of $\LInd_{H_i,\psi_i}^{G_i,\chi_i}$ (Proposition \ref{prop: lind}).
\end{proof}

As before, we analyze the effect of $N_i$ in the construction.

\begin{lemma}
Suppose that $(H_i,N_i)$ and $(H'_i,N'_i)$ are two well-matched special pairs for $G_i$, $i=1,2$.
Assume that
\begin{enumerate}
\item $H_i\le H'_i$ and $N'_i\le N_i$, $i=1,2$.
\item The isomorphism $\iota_{H_1}^{H_2}$ of \eqref{def: iota12} is the restriction of $\iota_{H_1'}^{H_2'}$.
\item $Z_{N_i\cap H_i'}(G_i)\le H_i$, $i=1,2$.
\end{enumerate}
Let $\chi_i$ be compatible genuine characters of $Z_{N_i}(G_i)$, $i=1,2$.
Let $\chi'_i$ be the restriction of $\chi_i$ to $Z_{N'_i}(G_i)$.
Then, $\chi'_i$ are compatible genuine characters of $Z_{N'_i}(G_i)$, $i=1,2$ and
$\trns_{G_1,H_1,\chi_1}^{G_2,H_2,\chi_2}$ is the restriction of $\trns_{G_1,H_1',\chi'_1}^{G_2,H_2',\chi'_2}$ to
$\Reps_{\chi_1}(G_1)$.
\end{lemma}

\begin{proof}
The characters $\chi'_i$ are compatible since $Z_{N'_i\cap H'_i}(G_i)\le Z_{N_i\cap H_i}(G_i)$ by assumption.

Let $\psi_i$, $i=1,2$ be congruous characters of $H_i\cap N_i$ such that $\psi_i$ and $\chi_i$
are consistent for $i=1,2$. Then, by \eqref{def: trns}
\[
\trns_{G_1,H_1,\chi_1}^{G_2,H_2,\chi_2}\circ\LInd_{H_1,\psi_1,N_1}^{G_1,\chi_1}=
\LInd_{H_2,\psi_2,N_2}^{G_2,\chi_2}\circ I_{H_1,\psi_1}^{H_2,\psi_2}.
\]
Let $\psi_1'$ be a character of $H_1'\cap N'_1$ that is consistent with $\psi_1$.
Let $\psi_2'$ be the character of $H_2'\cap N'_2$ that is congruous to $\psi_1'$.
Then, $\psi_2'$ is consistent with $\psi_2$.
Let $\eta_i$ be the character $\psi_i\rest_{Z_{N_i\cap H_i'}(G_i)}\psi_i'$ of $Z_{N_i\cap H_i'}(G_i)(N'_i\cap H'_i)=
Z(H_i')\cap N_i$ (see \eqref{eq: n1''}).
Then, by Lemma \ref{lem: seesaw} parts \ref{part: fnctrltn} and \ref{part: n2h2zn1s},
\[
\LInd_{H_i,\psi_i,N_i}^{G_i,\chi_i}=\LInd_{H_i',\psi_i',N'_i}^{G_i,\chi_i'}\rest_{\Reps_{\eta_i}(H_i')}\circ\LInd_{H_i,\psi_i,N_i\cap H'_i}^{H_i',\eta_i},\ \ i=1,2.
\]
Using the commutative diagram
\[
\begin{tikzcd}
\Reps_{\psi_1}(H_1) \arrow[d,"\LInd_{H_1,\psi_1}^{H'_1,\eta_1}"'] \arrow[r, "I_{H_1,\psi_1}^{H_2,\psi_2}"] &
\Reps_{\psi_2}(H_2) \arrow[d,"\LInd_{H_2,\psi_2}^{H'_2,\eta_2}"] \\
\Reps_{\eta_1}(H'_1)\arrow[r,"I_{H_1',\eta_1}^{H'_2,\eta_2}"] &\Reps_{\eta_2}(H'_2)
\end{tikzcd}
\]
it follows that
\[
\trns_{G_1,H_1,\chi_1}^{G_2,H_2,\chi_2}\circ
\LInd_{H_1',\psi_1',N'_1}^{G_1,\chi_1'}\rest_{\Reps_{\eta_1}(H_1')}=
\LInd_{H_2',\psi_2',N'_2}^{G_2,\chi_2'}\rest_{\Reps_{\eta_2}(H_2')}\circ I_{H_1',\eta_1}^{H_2',\eta_2}.
\]
On the other hand, $\psi'_i$ is consistent with $\chi_i'$ and therefore, again by \eqref{def: trns}
\[
\trns_{G_1,H_1',\chi_1'}^{G_2,H_2',\chi_2'}\circ\LInd_{H_1',\psi_1',N_1'}^{G_1,\chi_1'}=
\LInd_{H_2',\psi_2',N_2'}^{G_2,\chi_2'}\circ I_{H_1',\psi_1'}^{H_2',\psi_2'}.
\]
The lemma follows.
\end{proof}

Finally, we address transitivity of this construction.

\begin{corollary} \label{cor: trns123}
Let
\[
\pr_i:G_i\rightarrow\bas{G},\ i=1,2,3
\]
be three $A$-coverings of $\bas{G}$. Let $\bas{H}\le\bas{H}'$ and $\bas{N}'\le\bas{N}$ be four subgroups of $\bas{G}$.
Let $H_i,H'_i,N_i,N'_i$ be the inverse images of $\bas{H},\bas{H}',\bas{N},\bas{N}'$ under $\pr_i$ in $G_i$, $i=1,2,3$.
Assume that
\begin{enumerate}
\item For $1\le i<j\le 3$ $(H_i,N_i)$ and $(H_j,N_j)$ are well-matched special pairs for $G_i$ and $G_j$.
\item $(H_2',N'_2)$ and $(H_3',N_3')$ are well-matched special pairs for $G_2$ and $G_3$.
\item $\iota_{H_1}^{H_3}=\iota_{H_2}^{H_3}\circ\iota_{H_1}^{H_2}$.
\item $\iota_{H_2}^{H_3}$ is the restriction of $\iota_{H_2'}^{H_3'}$.
\item $Z_{N_i\cap H_i'}(G_i)\le H_i$, $i=2,3$.
\end{enumerate}
Let $\chi_i$ be a genuine character of $Z_{N_i}(G)$, $i=1,2,3$ such that $\chi_i$ and $\chi_j$ are compatible for $i<j$.
Let $\chi'_i$ be the restriction of $\chi_i$ to $Z_{N_i'}(G)$, $i=2,3$.
Then,
\[
\trns_{G_1,H_1,\chi_1}^{G_3,H_3,\chi_3}=
\trns_{G_2,H_2',\chi'_2}^{G_3,H_3',\chi'_3}\rest_{\Reps_{\chi_2}(G_2)}\circ
\trns_{G_1,H_1,\chi_1}^{G_2,H_2,\chi_2}.
\]
\end{corollary}

\begin{proof}
Indeed, by the previous lemma, it is enough to check that
\[
\trns_{G_1,H_1,\chi_1}^{G_3,H_3,\chi_3}=
\trns_{G_2,H_2,\chi_2}^{G_3,H_3,\chi_3}\circ
\trns_{G_1,H_1,\chi_1}^{G_2,H_2,\chi_2}.
\]
This is immediate from the definition \eqref{def: trns}.
\end{proof}

\section{Metaplectic tensor product} \label{sec: mtp}

For the rest of the paper, let $F$ be a non-archimedean locally compact field of residue characteristic $p$.
The $\ell$-groups that will be considered henceforth are the $F$-points of reductive groups over $F$ (in the $p$-adic topology)
as well as central extensions thereof.

As pointed out in \cite{MR3053009} and other sources, the basic ingredients of the representation theory of $p$-adic groups
continue to hold for covering groups (see also \cite{2206.06905}).

Let $F^{\times}$ be the multiplicative group of $F$,
$\ordr_F$ the ring of integers of $F$ and $\abs{\cdot}=\abs{\cdot}_F$ the normalized absolute value on $F$.

Throughout, we fix a positive integer $n\ge1$.
We assume that  the cyclic group $\mu_n:=\mu_n(F)$ of the $n$-th roots of unity in $F^{\times}$ is of order $n$.
In particular, if $\chr(F)=p$, then $p\nmid n$.
We write $F^{\times n}$ for the finite-index open subgroup $\{x^n\mid x\in F^{\times}\}$ of $n$-powers in $F^{\times}$.
For $a,b\in F^{\times}$ we write $a\equiv_nb$ if $a/b\in F^{\times n}$.
We denote by $X_n(F^{\times})$ the finite group of characters $\chi$ of $F^\times$ such that $\chi^n=1$, i.e.,
the characters that factor through $F^{\times}/F^{\times n}$.

Denote by
\[
(\cdot,\cdot)_n:F^{\times}\times F^{\times}\rightarrow \mu_n
\]
the $n$-th order \emph{Hilbert symbol} (see \cite{MR1344916}*{XIII.\S 5}).
It is a bimultiplicative, antisymmetric pairing that descends to a non-degenerate pairing on
$F^{\times}/F^{\times n}\times F^{\times}/F^{\times n}$.

Moreover, for any integer $k$ and $x\in F^{\times}$ we have
\[
(x,y^k)_n=1\text{ for all }y\in F^{\times}\iff x^k\equiv_n1.
\]
Dually, we have
\begin{equation} \label{eq: kpwrhs}
(x,y)_n=1\text{ for all }y\in F^{\times}\text{ such that }y^k\equiv_n1\iff\exists z\in F^{\times}\text{ such that }x\equiv_n z^k.
\end{equation}

Note that if $p\ndiv n$, then $(x,y)_n=1$ for every $y\in \ordr_F^{\times}$ if and only if $x\in \ordr_F^{\times}F^{\times n}$.

\subsection{Kazhdan-Patterson covering groups}\label{subsectionKPgroup}

In this subsection, we recall the definition and properties of Kazhdan--Patterson covering groups following \cite{MR743816} and \cite{MR1670203}.

Let $r$ be a positive integer. We write $\bas{G}_r:=\GL_r(F)$ and
\[
\bas{Z}_r=Z(\bas{G}_r)=\{\lambda I_r\mid\lambda\in F^\times\},
\]
the center of $G_r$. Let $\nu$ be the character $\abs{\det\cdot}_F$ of $\bas{G}_r$.

We will consider $\mu_n$-coverings of $\bas{G}_r$. Recall that they are given by elements of $H^2(\bas{G}_r,\mu_n)$.
The basic example is the Hilbert symbol which defines a 2-cocycle on $F^{\times}$.
A 2-cocycle for a split simple simply connected $p$-adic group was considered by Matsumoto in \cite{MR240214}.
(See also \cite{MR0349811}*{\S11-12} for the special linear group.)
In the case at hand the Steinberg symbol is $(\cdot,\cdot)_n^{-1}$ (cf.\ \cite{MR1670203}).
Specializing to $\SL_{r+1}(F)$, let $\sigma^{(0)}$ be the pullback of this 2-cocycle to $\bas{G}_r$ via the embedding
\[
\bas{G}_r\rightarrow\SL_{r+1}(F),\quad g\mapsto\diag(\det(g)^{-1},g).
\]
More generally, for any $c\in\Z/n\Z$ let $\sigma^{(c)}$
be the product of $\sigma^{(0)}$ with the pullback of $(\cdot,\cdot)^c$ via $\det$, i.e.
\[
\sigma^{(c)}(g_1,g_2)=\sigma^{(0)}(g_1,g_2)\cdot(\det(g_1),\det(g_2))_n^{c},\quad g_1,g_2\in\bas{G}_r.
\]
These 2-cocycles were considered by Kazhdan and Patterson in \cite{MR743816} and explicated in \cite{MR1670203}.
We denote by $G_r$ the corresponding $n$-th fold cover of $\bas{G}_r$.
Note that the cohomology classes in $H^2(\bas{G}_r,\mu_n)$ for different $c$'s in $\Z/n\Z$
may coincide.

From now on we fix $c\in\Z$ and write $c'=2c+1$.

As a rule, for a subgroup $\bas{H}$ of $\bas{G}_r$, we often denote its inverse image $\pr^{-1}(\bas{H})$ in $G_r$ simply by $H$.

The commutator $[\cdot,\cdot]:G_r\times G_r\rightarrow G_r$ factors through $\bas{G}_r\times\bas{G}_r$.
We denote by
\[
[\cdot,\cdot]_{\incvr}:\bas{G}_r\times\bas{G}_r\rightarrow G_r
\]
the resulting map. (It should not be confused with the commutator in $\bas{G}_r$ itself.)
Note that if $g_1$ and $g_2$ commute in $\bas{G}_r$, then
\begin{equation}\label{eqcommutator}
[g_1,g_2]_{\incvr}=\sigma^{(c)}(g_1,g_2)\sigma^{(c)}(g_2,g_1)^{-1}\in\mu_n.
\end{equation}
Likewise, the conjugation action of $G_r$ on itself factors through $\bas{G}_r$.
Thus, we sometimes write $x^y$ for $x\in G_r$ and $y\in\bas{G}_r$

Let $\beta=(r_1,\dots,r_k)$ be a composition of $r$, i.e., $r_1,\dots,r_k$ are positive integers such that
$r_1+\dots+r_k=r$.
Let $\bas{G}_\beta$ be the standard Levi subgroup of $\bas{G}_r$ isomorphic to
$\bas{G}_{r_1}\times\dots\times\bas{G}_{r_k}$ via the block diagonal embedding.
We call $G_\beta$ the \emph{Kazhdan-Patterson covering group} of $\bas{G}_\beta$ (with respect to $n$ and $c$).
Of course, when $k=1$ and $\beta=(r)$ we have $G_\beta=G_r$.

We write $\bas{Z}_\beta=Z(\bas{G}_\beta)$ for the center of $\bas{G}_\beta$, isomorphic to $k$ copies of $F^\times$.
Let
\begin{gather*}
\bas{Z}_{\beta,\sml}=\{\diag(\lambda_1I_{r_1},\dots,\lambda_kI_{r_k})\mid\lambda_i\equiv_n1\text{ for all }i\},\\
\bas{Z}_{\beta,\lrg}=\{\diag(\lambda_1I_{r_1},\dots,\lambda_kI_{r_k})\mid\lambda_i^{r_i}\equiv_n1\text{ for all }i\}.
\end{gather*}
Note that by our convention $Z_\beta$ is $\pr^{-1}(\bas{Z}_\beta)$, rather than $Z(G_\beta)$ (which is smaller).

We record the following basic facts.

\begin{subequations}
\begin{lemma} \
\begin{enumerate}
\item For $r=1$ we have $\sigma^{(c)}(x,y)=(x,y)_n^{c}$ for $x,y\in F^{\times}$.
\item (\cite{MR1670203}*{Theorem 11}) The restriction of $\sigma^{(c)}$ to $\bas{G}_\beta$ is given by
\begin{equation}\label{eqblockdecomp}
\begin{aligned}
&\sigma^{(c)}(\diag(g_1,\dots,g_k),\diag(g_1',\dots,g_k'))=\\&
\bigg[\prod_{i=1}^k\sigma^{(c)}(g_i,g_i')\bigg]\cdot\bigg[\prod_{1\leq i<j\leq k}(\det(g_i),\det(g_j'))_n^{c+1}\cdot(\det(g_j),\det(g_i'))_n^{c}\bigg].    \end{aligned}
\end{equation}
In particular, for $x_1,\dots,x_r,x_1',\dots,x_r'\in F^{\times}$, we have
\begin{align*}
	\sigma^{(c)}(\diag(x_1,\dots,x_r),\diag(x_1',\dots,x_r'))=
	\bigg[\prod_{i=1}^r(x_i,x_i')_n^{c}\bigg]\cdot\bigg[\prod_{1\leq i<j\leq r}(x_i,x_j')_n^{c+1}\cdot(x_j,x_i')_n^{c}\bigg].
\end{align*}

\item (\cite{MR3442519}*{Lemma 3.9} and \cite{MR3038716}*{\S 2.1, Lemma 1})
For $z=\lambda I_r\in\bas{Z}_r$ and $g\in\bas{G}_r$, we have
\begin{equation}\label{eqwtzwtg}
[z,g]_{\incvr}=(\lambda,\det(g))_n^{rc'-1}.
\end{equation}
In particular,
\begin{equation} \label{eq: cntrofcfr}
Z(G_r)=\pr^{-1}(\{\lambda I_r\mid\lambda^{rc'-1}\equiv_n1\}).
\end{equation}

\item For any commuting elements $g=\diag(g_1,\dots,g_k), g'=\diag(g_1',\dots,g_k')\in\bas{G}_\beta$, we have
\begin{equation} \label{eq: commblocks}
[g,g']_{\incvr}=\prod_{i=1}^k[g_i,g_i']_{\incvr}\cdot\prod_{i\neq j}(\det(g_i),\det(g_j'))_n^{c'}.
\end{equation}

\item For $z=\diag(\lambda_1I_{r_1},\dots,\lambda_kI_{r_k})\in\bas{Z}_\beta$ and
$g=\diag(g_1,\dots,g_k)\in\bas{G}_\beta$, we have
\begin{equation}\label{eqwtzwtgblock}
	[z,g]_{\incvr}=\prod_{i=1}^k (\lambda_i,\det(g_i))_n^{r_ic'-1}\cdot\prod_{i\neq j}(\lambda_i,\det(g_j))_n^{r_ic'}=\prod_{i=1}^k(\lambda_i,\det(g)^{r_ic'}\det(g_i)^{-1})_n.
\end{equation}
Thus,
\begin{equation} \label{cntrbs}
Z(G_\beta)=Z(G_r)Z_{\beta,\sml}
\end{equation}
and
\begin{equation} \label{eq: zsmlrg}
Z_{\beta,\lrg}\cap Z(G_\beta)=Z_{\beta,\sml}.
\end{equation}
\end{enumerate}
\end{lemma}
\end{subequations}

Indeed, \eqref{eq: commblocks} follows from \eqref{eqcommutator} and \eqref{eqblockdecomp}
while \eqref{eqwtzwtgblock} follows from \eqref{eqwtzwtg} and \eqref{eq: commblocks}.
The relations \eqref{cntrbs} holds since $Z(G_\beta)\le Z_\beta$ and by \eqref{eqwtzwtgblock} and \eqref{eq: cntrofcfr}
\[
Z_{\beta,\sml}Z(G_r)\le Z(G_\beta),\ \ Z_{Z_\beta}(G_\beta\cap\pr^{-1}(\SL_r(F)))=Z_{\beta,\sml}Z_r\text{ and }Z_{Z_r}(G_\beta)=Z(G_r).
\]
Finally, by \eqref{cntrbs} and \eqref{eq: cntrofcfr}, we have
\[
Z_{\beta,\lrg}\cap Z(G_\beta)=Z_{\beta,\lrg}\cap Z(G_r)Z_{\beta,\sml}=Z_{\beta,\sml}(Z_{\beta,\lrg}\cap Z(G_r))=Z_{\beta,\sml}.
\]

From now on we fix a faithful character $\epsilon$ of $\mu_n$.
Let $H$ be a subgroup of $G_r$ containing $\mu_n$.
By a genuine representation of $H$ we will mean that $\mu_n$ acts by $\epsilon$
(i.e., an object of $\Reps_\epsilon(H)$).

For any representation $\pi$ of $H$ and a character $\chi$ of $F^{\times}$, we denote by $\pi\chi$
the twist of $\pi$ by the pullback of $\chi$ to $H$ via $\det\circ\pr$.
Clearly, if $\pi$ is genuine, then so is $\pi\chi$.

\begin{definition}\label{defequivclass}
Two representations $\pi,\pi'\in\Irr_{\epsilon}(G_\beta)$ are called \emph{weakly equivalent} if
they are in the same $X_n(F^{\times})$-orbit under twisting.
We denote by $[\pi]$ the weak equivalence class of $\pi$ and by
\[
\Irr_{\epsilon,\twst}(G_\beta)
\]
the set of $X_n(F^{\times})$-orbits in $\Irr_\epsilon(G_\beta)$.
\end{definition}

See Remark \ref{rem: realtwist} below for an equivalent definition.

Let $\bas{P}$ be a parabolic subgroup of $\bas{G}_r$ defined over $F$ and let $\bas{U}$ be the unipotent radical of $\bas{P}$.
Since by assumption $\chr F\ndiv n$, there is a unique lifting
\begin{equation} \label{eq: cansec}
\sctn_U:\bas{U}\rightarrow U
\end{equation}
i.e., a group homomorphism (necessarily continuous) such that $\pr\circ\sctn_U=\id_{\bas{U}}$.
(See \cite{MR1361168}*{Appendix I}.
Note that the proof in characteristic $0$ works also in positive characteristic not dividing $n$.)
In fact, we will only use that $\sctn_U$ is equivariant under $\bas{P}$-conjugation.
We will identify $\bas{U}$ with its image under $\sctn_U$, a subgroup of $U$.
Thus, if $\bas{M}$ is a Levi subgroup of $\bas{P}$, then we have a decomposition $P=M\ltimes\bas{U}$.

Let $W_r$ be the Weyl group of $\bas{G}_r$, which is isomorphic to the symmetric group on $r$ elements.
We identify $W_r$ with the subgroup of permutation matrices of $\bas{G}_r$.

\subsection{Metaplectic tensor product}\label{subsectionMTP}

Let $\beta=(r_1,\dots,r_k)$ be a composition of $r$.

In the linear case, $\bas{G}_\beta\simeq\bas{G}_{r_1}\times\dots\times\bas{G}_{r_k}$ and there is a multiexact, multiadditive multifunctor
\[
\Reps(\bas{G}_{r_1})\times\dots\times\Reps(\bas{G}_{r_k})\rightarrow\Reps(\bas{G}_\beta),
\]
given by the tensor product. It gives rise to a bijection\footnote{In contrast, it is hopeless to classify indecomposable representations
of groups as simple as $\Z\times\Z$ (cf.\ \cite{MR0254068}).}
\[
\Irr(\bas{G}_{r_1})\times\dots\times\Irr(\bas{G}_{r_k})\rightarrow\Irr(\bas{G}_\beta).
\]
On the other hand, for the covering case it is no longer true that the blocks $G_{r_i}$ commute in $G_\beta$.
Thus, $G_\beta$ is not isomorphic to $G^\beta:=(G_{r_1}\times\dots\times G_{r_k})/\Xi$ where
\[
\Xi=\{(\zeta_1,\dots,\zeta_k)\in\mu_n\times\dots\times\mu_n\mid\zeta_1\zeta_2\dots\zeta_k=1\}.
\]

Nevertheless, as we shall see below, the representation theories of $G_\beta$ and $G^\beta$ are essentially identical.

Let
\[
\bas{H}_r=\{g\in\bas{G}_r\mid\det g\equiv_n1\},
\]
and more generally
\[
\bas{H}_\beta=\{g=\diag(g_1,\dots,g_k)\in\bas{G}_\beta\mid g_i\in\bas{H}_{r_i}\text{ for all }i\}.
\]

\begin{lemma}\label{lemmaGbetaspecial}
The pair $(H_\beta,Z_\beta)$ is special in $G_\beta$. Moreover, $Z_\beta=Z_{G_\beta}(H_\beta)$.
\end{lemma}

\begin{proof}
For simplicity, let $G=G_\beta$, $H=H_\beta$ and $N=Z_G(H)$.

Let $\bas{G}_\beta^{\der}\simeq\SL_{r_1}(F)\times\dots\times\SL_{r_k}(F)$ denote the derived
group of $\bas{G}_\beta$.
Clearly, $Z_{\bas{G}_\beta}(\bas{H}_\beta)=\bas{Z}_\beta$
since $\bas{H}_\beta\supset\bas{G}_\beta^{\der}$ and $Z_{\bas{G}_\beta}(\bas{G}_\beta^{\der})=\bas{Z}_\beta$.
Hence,
\[
\pr(N)\subset\bas{Z}_\beta.
\]
On the other hand by \eqref{eqwtzwtgblock}, $\pr(N)\supset\bas{Z}_\beta$. Hence, $N=Z_\beta$.

It remains to show \eqref{eq: zgzhnh}.
Obviously, $Z_G(N\cap H)\supset NH$.
Suppose that $g=\diag(g_1,\dots,g_k)\in\pr(Z_G(N\cap H))$. By \eqref{eqwtzwtgblock},
for every $i$ we have $(\det g_i,\lambda_i)_n=1$ whenever $\lambda_i^{r_i}\equiv_n 1$.
It follows from \eqref{eq: kpwrhs} that $\det g_i\equiv_n\nu_i^{r_i}$ for some $\nu_i\in F^\times$.
Therefore, $g\in\bas{Z}_\beta\bas{H}_\beta=\pr(NH)$.

In conclusion, $(H,N)$ is special.
The lemma follows.
\end{proof}

\begin{remark} \label{rem: realtwist}
It follows from the lemma, together with Remark \ref{rem: ggorbit} that $\pi,\pi'\in\Irr_\epsilon(G_\beta)$ are weakly equivalent if
they are in the same $\PD{G_\beta/H_\beta}$-orbit under twisting.
Moreover, if $\omega\in\PD{G_\beta/H_\beta}$, then up to isomorphism $\pi\cdot\omega$ depends only
on the restriction of $\omega$ to the subgroup $Z(G_\beta)/Z_{H_\beta}(G_\beta)\simeq Z(G_r)/Z_{r,\sml}$.
(The last equality follows from \eqref{cntrbs} and \eqref{eq: zsmlrg}.)
\end{remark}

We consider $G^\beta=(G_{r_1}\times\dots G_{r_k})/\Xi$ as a covering group of $\bas{G}_\beta$:
\[
G^\beta\xrightarrow{\pr^\beta}\bas{G}_\beta.
\]
In general, for a subgroup $\bas{H}$ of $\bas{G}_\beta$, we will denote by $H^\beta$ its preimage in $G^\beta$ under $\pr^\beta$.
(In the case $\beta=(r)$, $H^{(r)}=H$.)

Recall that by \eqref{eqblockdecomp}, for any $s\le r$ and $0\le t\le r-s$ the pullback of the cocycle $\sigma^{(c)}$ defining $G_r$ to
$\bas{G}_s$ via $x\mapsto\diag(I_t,x,I_{r-s-t})$ is independent of $t$.
Thus, we can identify $G_{r_i}$ with the ``$i$-th block'' of $G_\beta$, as well as with the $i$-th block in $G^\beta$.
Of course, the different blocks do not commute with each other in $G_\beta$, only in $G^\beta$.
However, the images of $H_{r_i}$ in $G_\beta$ pairwise commute by \eqref{eq: commblocks}.
Thus, we get an isomorphism of covering groups of $\bas{H}_\beta$
\[
\iota_\beta:H_\beta\rightarrow H_\beta^\beta.
\]

For inductive arguments, we will also need to consider the relative case.
Given two compositions $\gamma$ and $\beta$ or $r$ we write $\gamma\prec\beta$ if $\bas{G}_\gamma\le\bas{G}_\beta$,
i.e., if $\gamma$ is a refinement of $\beta$.
In this case we can write $\gamma$ in the form $(\gamma_1,\dots,\gamma_k)$,
where $\gamma_i$ is a composition of $r_i$ for each $i=1,\dots,k$.

Let $\gamma$ and $\beta_1$ be compositions of $r$ such that $\gamma\prec\beta_1$. Define
\[
\bas{H}_{\gamma,\beta_1}=\bas{G}_\gamma\cap\bas{H}_{\beta_1}\le\bas{G}_{\beta_1}.
\]
Recall that according to our convention, for any additional composition $\beta_2$ such that $\beta_1\prec\beta_2$,
$H_{\gamma,\beta_1}^{\beta_2}$ is the inverse image under $\pr^{\beta_2}$ of $\bas{H}_{\gamma,\beta_1}$.
As before, the covering groups $H_{\gamma,\beta_1}^{\beta_2}$ of $\bas{H}_{\gamma,\beta_1}$ essentially do not depend on $\beta_2$.
In particular, $H_{\gamma,\beta}^\beta$ is naturally isomorphic to $H_{\gamma,\beta}=H_{\gamma,\beta}^{(r)}$ for $\gamma\prec\beta$.
We will denote by
\[
\iota_{\gamma,\beta}:H_{\gamma,\beta}^\beta\rightarrow H_{\gamma,\beta}
\]
the natural isomorphism for any $\gamma\prec\beta$.

As in Lemma \ref{lemmaGbetaspecial}, we have
\begin{lemma} \label{lem: specbeta12}
Let $\gamma$, $\beta_1$ and $\beta_2$ be compositions such that $\gamma\prec\beta_1\prec\beta_2$.
Let $G_i=G_\gamma^{\beta_i}$, $H_i=H_{\gamma,\beta_1}^{\beta_i}$, $N_i=Z_{\beta_1}^{\beta_i}$, $i=1,2$.
Then, $(H_i,N_i)$ are well-matched special pairs in $G_i$, $i=1,2$ and
\[
N_i\cap H_i=Z_{\beta_1,\lrg}^{\beta_i},\
Z_{N_i\cap H_i}(G_i)=Z_{\beta_1,\sml}^{\beta_i}.
\]
\end{lemma}

By Proposition \ref{prop: transferwm},9 we conclude
\begin{corollary}
For any compatible genuine characters $\chi_i$ of $Z_{N_i}(G_i)=Z(G_{\beta_1}^{\beta_i})$ we have an equivalence of categories
\begin{equation} \label{def: tnsrf}
\tnsr_{\gamma,\beta_1,\chi_1}^{\beta_2,\chi_2}:
=\trns_{G_1,H_1,\chi_1}^{G_2,H_2,\chi_2}:\Reps_{\chi_1}(G_\gamma^{\beta_1})
\rightarrow\Reps_{\chi_2}(G_\gamma^{\beta_2}).
\end{equation}
\end{corollary}

In particular taking $\beta_1=\beta$ and $\beta_2=(r)$ in the above lemma, we have the following special case.

\begin{corollary}
Let $\gamma$ and $\beta$ be two compositions of $r$ such that $\gamma\prec\beta$.
Let $\omega^\beta$ and $\omega_\beta$ be compatible characters of $Z(G^\beta)=(Z(G_{r_1})\times\dots\times Z(G_{r_k}))/\Xi$ and $Z(G_\beta)$.
Then, we have an equivalence of categories
\begin{equation} \label{def: tnsrbeta}
\tnsr_{\gamma,\beta,\omega^\beta}^{\omega_\beta}:\Reps_{\omega^\beta}(G_\gamma^\beta)\rightarrow\Reps_{\omega_\beta}(G_\gamma).
\end{equation}
It respects the contragredient.
\end{corollary}

For $\beta=\gamma$, we also write $\tnsr_{\beta,\omega^\beta}^{\omega_\beta}=\tnsr_{\beta,\beta,\omega^\beta}^{\omega_\beta}$.

The following remark will be used repeatedly.

\begin{remark} \label{rem: compat}
By \eqref{cntrbs}, given $\omega^\beta$, the choice of $\omega_\beta$ amounts to a choice of a character $\omega$ of $Z(G_r)$ that is prescribed
on $Z_{r,\sml}$.
More precisely, writing $\omega^\beta=\omega_1\otimes\dots\otimes\omega_k$ where $\omega_i$ is a genuine character of $Z(G_{r_i})$, $i=1,\dots,k$,
the condition on $\omega$ is
\begin{equation}\label{eqomegamulpiomega}
(\omega_1\rest_{Z_{r_1,\sml}}\otimes\dots\otimes\omega_k\rest_{Z_{r_k,\sml}})_{Z_{r,\sml}}=\omega\rest_{Z_{r,\sml}}.
\end{equation}
In this case we will say that $\omega$ is a \emph{compatible character} of $Z(G_r)$ with respect to $\omega^\beta$ or
$(\omega_1,\dots,\omega_k)$, or just a compatible character if $\omega^\beta$ is clear from the context.
\end{remark}

\begin{remark}
In the case $k=1$ (and $\beta=(r)$), $\omega^\beta$ and $\omega_\beta$ are characters of $Z(G_r)$ that coincide
on $Z(G_r)\cap H_r=Z_{r,\sml}$. We can extend $\omega_\beta(\omega^\beta)^{-1}$ to a character $\omega$ of the abelian group $G_r/H_r$.
Then, $\tnsr_{\beta,\omega^\beta}^{\omega_\beta}(\pi)\simeq\pi\cdot\omega$.
\end{remark}

By abuse of notation we will also write
\[
\tnsr_{\beta,\omega^\beta}^{\omega_\beta}:\Reps_{\omega_1}(G_{r_1})\times\dots\times\Reps_{\omega_k}(G_{r_k})\rightarrow\Reps_{\omega_\beta}(G_\beta)
\]
for the multiexact, multiadditive multifunctor which is the composition of \eqref{def: tnsrbeta} with the ordinary tensor product
\[
\otimes:\Reps_{\omega_1}(G_{r_1})\times\dots\times\Reps_{\omega_k}(G_{r_k})\rightarrow\Reps_{\omega^\beta}(G^\beta).
\]
We caution however that the latter commutes with the contragredient only in the admissible case.
At any rate, by Proposition \ref{prop: transferwm} and Remark \ref{rem: realtwist}, we get a bijection
\[
\Irr_{\epsilon,\twst}(G_{r_1})\times\dots\times\Irr_{\epsilon,\twst}(G_{r_k})\rightarrow\Irr_{\epsilon,\twst}(G_\beta).
\]
Moreover for each $i$ and $\pi_i\in\Reps_{\omega_i}(G_{\gamma_i})$, it is convenient to introduce
\begin{equation}\label{eqMTPdef}
(\pi_1\otimes\dots\otimes\pi_k)_\omega:=\tnsr_{\gamma,\beta,\omega^\beta}^{\omega_\beta}(\pi_1,\dots,\pi_k).
\end{equation}
Note that for brevity we omit $\omega^\beta=\omega_1\otimes\dots\otimes\omega_k$ from the notation since
it is encoded in the assumption on $\pi_i$.
We call \eqref{eqMTPdef} the \emph{metaplectic tensor product} with respect to $(\pi_1,\dots,\pi_k;\omega)$.

The following result is a consequence of Corollary \ref{cor: trns123}. It essentially says that the metaplectic tensor product is associative.

\begin{lemma}\label{lemmaMTPassociative}
Suppose that $\gamma,\beta_1,\beta_2,\beta_3$ are compositions such that $\gamma\prec\beta_1\prec\beta_2\prec\beta_3$.
Let $\chi_i$ be a genuine character of $Z(G_{\beta_1}^{\beta_i})$, $i=1,2,3$.
Assume that $\chi_i$ are pairwise compatible.
Let $\chi'_i$, $i=2,3$ be the restriction of $\chi_i$ to $Z(G_{\beta_2}^{\beta_i})$.
Then, we have an equivalence of functors
\[
\tnsr_{\gamma,\beta_2,\chi'_2}^{\beta_3,\chi'_3}\rest_{\Reps_{\chi_2}(G_2)}\circ\tnsr_{\gamma,\beta_1,\chi_1}^{\beta_2,\chi_2}=
\tnsr_{\gamma,\beta_1,\chi_1}^{\beta_3,\chi_3}.
\]
\end{lemma}

\begin{corollary}\label{cor: associativeMTP}
For $i=1,2,3$ let $\gamma_i$ be a composition of $r_i$, $\omega_i$ a genuine character of $Z(G_{r_i})$
and $\pi_i\in \Reps_{\omega_i}(G_{\gamma_i})$.
Let $\omega$, $\omega_{12}$ and $\omega_{23}$ be compatible characters of $Z(G_{r_1+r_2+r_3})$, $Z(G_{r_1+r_2})$ and $Z(G_{r_2+r_3})$ respectively.
Then,
\[
((\pi_1\otimes\pi_2)_{\omega_{12}}\otimes\pi_3)_\omega\simeq(\pi_1\otimes\pi_2\otimes\pi_3)_\omega\simeq
(\pi_1\otimes(\pi_2\otimes\pi_3)_{\omega_{23}})_\omega.
\]	
\end{corollary}

Indeed, we take $\gamma=(\gamma_1,\gamma_2,\gamma_3)$, $\beta_1=(r_1,r_2,r_3)$,
$\beta_2=(r_1+r_2,r_3)$ (or $(r_1,r_2+r_3)$) and $\beta_3=(r_1+r_2+r_3)$ in Lemma \ref{lemmaMTPassociative}.

\begin{remark}
The choice of characters $\omega_{12}$ and $\omega_{23}$ is immaterial.
Thus,  to simplify the notation we will henceforth write
$((\pi_1\otimes\pi_2)\otimes\pi_3)_\omega=((\pi_1\otimes\pi_2)_{\omega_{12}}\otimes\pi_3)_\omega$ and $(\pi_1\otimes(\pi_2\otimes\pi_3))_\omega=(\pi_1\otimes(\pi_2\otimes\pi_3)_{\omega_{23}})_\omega$.
\end{remark}

Finally, suppose that $\beta=(r_1,\dots,r_k)$ is a composition of $r$ and $\varsigma$ is a permutation of $\{1,\dots,k\}$.
Let $\beta'$ be the partition $(r'_1,\dots,r'_k)$ where $r'_i=r_{\varsigma(i)}$, $i=1,\dots,k$.
From the definition of $G^\beta$, we have a commutative diagram
\begin{equation} \label{eq: permblocks}
\begin{tikzcd}
G^\beta\arrow[r,"\varsigma"]\arrow[d,"\pr^\beta"]&G^{\beta'}\arrow[d,"\pr^{\beta'}"]\\
\bas{G}_\beta\arrow[r,"\varsigma"]&\bas{G}_{\beta'}
\end{tikzcd}
\end{equation}
where the horizonal arrows are permutation of the blocks of $G^\beta$ and $\bas{G}_\beta$.

Let $\omega^\beta$ be a character of $Z(G^\beta)$
and let $\omega^{\beta'}$ be the corresponding character of $Z(G^{\beta'})$ under \eqref{eq: permblocks}.
Let $\omega$ be a character of $Z(G_r)$ that is compatible with $\omega^\beta$ (or equivalently,
with $\omega^{\beta'}$) and let $\omega_\beta$ (resp., $\omega_{\beta'}$) be the corresponding character of $Z(G_\beta)$ (resp., $Z(G_{\beta'})$)
(see Remark \ref{rem: compat}).

Let $w\in W_r$ be the permutation matrix conjugating $\bas{G}_\beta$ to $\bas{G}_{\beta'}$
(and hence $G_\beta$ and $Z_\beta$ to $G_{\beta'}$ and $Z_{\beta'}$) such that
\[
{}^w\diag(g_1,\dots,g_k)=\diag(g_{\varsigma(1)},\dots,g_{\varsigma(k)}),\ \ \forall \diag(g_1,\dots,g_k)\in\bas{G}_\beta.
\]

Consider the embedding $j_{\beta,i}:G_{r_i}\rightarrow G_\beta$ in the $i$-th block.
Since $\pr\circ({}^wj_{\beta,\varsigma(i)})=\pr\circ j_{\beta',i}$, the embeddings
${}^wj_{\beta,\varsigma(i)}$ and $j_{\beta',i}$ differ by a twist by a homomorphism $\bas{G}_{r_i}\rightarrow\mu_n$,
i.e., by a character in $X_n(F^{\times})$.
In particular, ${}^wj_{\beta,\varsigma(i)}$ and $j_{\beta',i}$ coincide on $H_{r_i}$.
Since the groups $j_{\beta,i}(H_{r_i})$ pairwise commute, we get a commutative diagram
\[
\begin{tikzcd}
	H_{\beta}^\beta \arrow[d,"\iota_\beta"'] \arrow[r, "\varsigma"] &
	H_{\beta'}^{\beta'} \arrow[d,"\iota_{\beta'}"] \\
	H_{\beta}\arrow[r,"w"] & H_{\beta'}.
\end{tikzcd}
\]

Let $\gamma=(\gamma_1,\dots,\gamma_k)$ be a refinement of $\beta$ as before.
Let $\gamma'=(\gamma_1',\dots,\gamma_k')$ be the refinement of $\beta'$ given by $\gamma_i'=\gamma_{\varsigma(i)}$, $i=1,\dots, k$.
Thus, ${}^w\bas{G}_\gamma=\bas{G}_{\gamma'}$, ${}^wG_\gamma=G_{\gamma'}$,
$\varsigma(\bas{G}_\gamma)=\bas{G}_{\gamma'}$ and $\varsigma(G_\gamma^\beta)=G_{\gamma'}^{\beta'}$.
We have a commutative diagram
\[
\begin{tikzcd}
	H_{\gamma,\beta}^\beta \arrow[d,"\iota_{\gamma,\beta}"'] \arrow[r, "\varsigma"] &
	H_{\gamma',\beta'}^{\beta'} \arrow[d,"\iota_{\gamma',\beta'}"] \\
	H_{\gamma,\beta}\arrow[r,"w"] & H_{\gamma',\beta'}.
\end{tikzcd}
\]

From the definitions \eqref{def: tnsrf} and \eqref{def: trns}, we infer

\begin{proposition}\label{propweylinvariant}
The following diagram is commutative
\[
\begin{tikzcd}
	\Reps_{\omega^\beta}(G_\gamma^\beta) \arrow[d,"\tnsr_{\gamma,\beta,\omega^\beta}^{\omega_\beta}"'] \arrow[r, "\varsigma"] &
	\Reps_{\omega^{\beta'}}(G_{\gamma'}^{\beta'}) \arrow[d,"\tnsr_{\gamma',\beta',\omega^{\beta'}}^{\omega_{\beta'}}"] \\
	\Reps_{\omega_\beta}(G_\gamma)\arrow[r,"w"] & \Reps_{\omega_{\beta'}}(G_{\gamma'}).
\end{tikzcd}
\]
\end{proposition}

\section{Parabolic induction}

\subsection{Definition}

Let $\beta$ be a composition of $r$. As in the linear case,
following the notation of \cite{MR579172}*{\S 2.3},
we consider (normalized) parabolic induction and Jacquet functor
\[
i_\beta=i_{\bas{U}_\beta,1}:\Reps(G_\beta)\rightarrow\Reps(G_r),\ \ \
r_\beta=r_{\bas{U}_\beta,1}:\Reps(G_r)\rightarrow\Reps(G_\beta),
\]
defined with respect to the decomposition $P_\beta=G_\beta\ltimes\bas{U}_\beta$.
(As usual, we identify $\bas{U}_\beta$ with a subgroup of $P_\beta$.)
Thus, $i_\beta$ is the composition of the pullback $\Reps(G_\beta)\rightarrow\Reps(P_\beta)$,
twisting by $\modulus_{\bas{U}_\beta}^{\frac12}$ and $\Ind_{P_\beta}^{G_\beta}=\ind_{P_\beta}^{G_\beta}$,
while $r_\beta$ is the $\bas{U}_\beta$-coinvariants, twisted by $\modulus_{\bas{U}_\beta}^{-\frac12}$.
These functors commute with taking contragredient and preserve admissibility and finite length.

More generally, let $\beta$ and $\gamma$ be compositions of $r$ with $\gamma\prec \beta$.
Let $\ordPgp\gamma\beta=\bas{P}_\gamma\cap\bas{G}_\beta$, a standard parabolic subgroup of $\bas{G}_\beta$,
and let $\ordNgp\gamma\beta=\bas{U}_\gamma\cap G_\beta$ be its unipotent radical.
We define the normalized parabolic induction and Jacquet functor
\[
\Pind\gamma\beta:=i_{\ordNgp\gamma\beta,1}:\Reps(G_\gamma)\rightarrow\Reps(G_\beta),\ \ \
\Jac\beta\gamma:=r_{\ordNgp\gamma\beta,1}:\Reps(G_\beta)\rightarrow\Reps(G_\gamma)
\]
with respect to the decomposition $\Pgp\beta\gamma=G_\gamma\ordNgp\beta\gamma$.

Similarly, we define
\[
\Pindbar\gamma\beta:=i_{\ordNgpbar\gamma\beta,1}:\Reps(G_\gamma)\rightarrow\Reps(G_\beta),\ \ \
\Jacbar\beta\gamma:=r_{\ordNgpbar\gamma\beta,1}:\Reps(G_\beta)\rightarrow\Reps(G_\gamma)
\]
with respect to the opposite parabolic subgroup $\ordPgpbar\gamma\beta=\bas{P}_\gamma^-\cap G_\beta$ and its unipotent radical
$\ordNgpbar\gamma\beta=\bas{U}_\gamma\cap G_\beta$.

From now on we fix a composition $\beta=(r_1,\dots,r_k)$ of $r$.

Let $\delta$ and $\gamma$ be compositions of $r$ such that $\delta\prec\gamma\prec\beta$.
Define
\[
\Pindn\delta\gamma\beta=i_{\bas{U}_{\delta,\gamma},1}:\Reps(G_\delta^\beta)\rightarrow\Reps(G_\gamma^\beta)
\]
and similarly $\Jacn\gamma\delta\beta$, $\Pindbarn\delta\gamma\beta$ and $\Jacbarn\gamma\delta\beta$.

Let $\omega^\beta=\omega_1\otimes\dots\otimes\omega_k$ be a genuine character of $Z(G^\beta)$
and let $\omega$ be a compatible genuine character of $Z(G_r)$ (see \eqref{eqomegamulpiomega}).
Let $\omega_\beta$ be the character of $Z(G_\beta)$ that extends $\omega$ and is compatible with $\omega^\beta$.

\begin{proposition}\label{propcomatibleMTPparablicJacquet}
We have the following equivalences of functors
\begin{align*}
\Pind\delta\gamma\circ\tnsr_{\delta,\beta,\omega^\beta}^{\omega_\beta}=
\tnsr_{\gamma,\beta,\omega^\beta}^{\omega_\beta}\circ\Pindn\delta\gamma\beta&:
\Reps_{\omega^\beta}(G_\delta^\beta)\rightarrow\Reps_{\omega_\beta}(G_\gamma),\\
\Jac\gamma\delta\circ\tnsr_{\gamma,\beta,\omega^\beta}^{\omega_\beta}=
\tnsr_{\delta,\beta,\omega^\beta}^{\omega_\beta}\circ\Jacn\gamma\delta\beta&:
\Reps_{\omega^\beta}(G_\gamma^\beta)\rightarrow\Reps_{\omega_\beta}(G_\delta),\\
\Pindbar\delta\gamma\circ\tnsr_{\delta,\beta,\omega^\beta}^{\omega_\beta}=
\tnsr_{\gamma,\beta,\omega^\beta}^{\omega_\beta}\circ\Pindbarn\delta\gamma\beta&:
\Reps_{\omega^\beta}(G_\delta^\beta)\rightarrow\Reps_{\omega_\beta}(G_\gamma),\\
\Jacbar\gamma\delta\circ\tnsr_{\gamma,\beta,\omega^\beta}^{\omega_\beta}=
\tnsr_{\delta,\beta,\omega^\beta}^{\omega_\beta}\circ\Jacbarn\gamma\delta\beta&:
\Reps_{\omega^\beta}(G_\gamma^\beta)\rightarrow\Reps_{\omega_\beta}(G_\delta).
\end{align*}
\end{proposition}

\begin{proof}
Recall the subgroups $H_{\delta,\beta}$, $H_{\gamma,\beta}$ of $G_\gamma$ and
the subgroups $H_{\delta,\beta}^\beta$, $H_{\gamma,\beta}^\beta$ of $G_\gamma^\beta$ introduced in \S \ref{subsectionMTP}.

By Lemma \ref{lem: specbeta12}, $(H_{\delta,\beta}, Z_\beta)$ and $(H_{\delta,\beta}^\beta, Z_\beta^\beta)$ are well-matched
special pairs in $G_\delta$ and $G_\delta^\beta$. Similarly, when replacing $\delta$ by $\gamma$.
Let $\psi_1$ be a character of $Z_\beta\cap H_{\delta,\beta}=Z_\beta\cap H_{\gamma,\beta}=Z_{\beta,\lrg}$
that is consistent with $\omega_\beta$.
Let $\psi_2$ be the character of $Z_\beta^\beta\cap H_{\delta,\beta}^\beta=Z_\beta^\beta\cap H_{\gamma,\beta}^\beta=Z_\beta^{\lrg}$
that is congruous to $\psi_1$ with respect to $\iota_\beta$.

Consider the diagram
\[
\begin{tikzcd}
\Reps_{\omega^\beta}(G_\delta^\beta) \arrow[r,"\LRes_{H_{\delta,\beta}^\beta,\psi_2}^{G_\delta^\beta,\omega^\beta}"] \arrow[d,"i_{\delta,\gamma}^\beta"] &
\Reps_{\psi_2}(H_{\delta,\beta}^\beta) \arrow[r,"\iota_{\delta,\beta}"] \arrow[d,"i_{\bas{U}_{\delta,\gamma}}"] &
\Reps_{\psi_1}(H_{\delta,\beta}) \arrow[r,"\LInd_{H_{\delta,\beta},\psi_1}^{G_\delta,\omega_\beta}"] \arrow[d,"i_{\bas{U}_{\delta,\gamma}}"] &
\Reps_{\omega_\beta}(G_\delta)\arrow[d,"i_{\delta,\gamma}"] \\
\Reps_{\omega^\beta}(G_\gamma^\beta) \arrow[r,"\LRes_{H_{\gamma,\beta}^\beta,\psi_2}^{G_\gamma^\beta,\omega^\beta}"] &
\Reps_{\psi_2}(H_{\gamma,\beta}^\beta) \arrow[r,"\iota_{\gamma,\beta}"] &
\Reps_{\psi_1}(H_{\gamma,\beta}) \arrow[r,"\LInd_{H_{\gamma,\beta},\psi_1}^{G_\gamma,\omega_\beta}"] &
\Reps_{\omega_\beta}(G_\gamma)
\end{tikzcd}
\]
The square on the right is commutative by transitivity of induction. Similarly for the square on the left,
since $\LRes*$ is the inverse of $\LInd*$. The middle square is clearly commutative.
Hence, the whole diagram is commutative. This prove the first stated equivalence of functors.
The other ones are proved similarly.
\end{proof}

\subsection{Bernstein--Zelevinsky product}

As before, fix a composition $\beta=(r_1,\dots,r_k)$ of $r$.
Let $\pi_i\in\Reps_{\omega_i}(G_{r_i})$ where $\omega_i$ is a genuine character of $Z(G_i)$ for each $i=1,\dots,k$,
and let $\omega$ be a genuine compatible character of $Z(G_r)$.

We define the \emph{Bernstein--Zelevinsky product} by
\[
(\pi_1\times\pi_2\times\dots\times\pi_k)_\omega:=i_\beta((\pi_1\otimes\pi_2\otimes\dots\otimes\pi_k)_\omega).
\]

As in the linear case, it enjoys the following properties, which will be used below freely.

\begin{proposition} \label{prop: BZprod} \

\begin{enumerate}
\item \label{part: mea} The multifunctor
\[
(\times_{i=1}^k)_\omega:\Reps_{\omega_1}(G_{r_1})\times\dots\times\Reps_{\omega_k}(G_{r_k})\rightarrow
\Reps_\omega(G_r)
\]
is multiexact and multiadditive.
\item \label{part: BZcontra} If the $\pi_i$'s are admissible, then
\begin{equation} \label{eq: contrBZ}
(\pi_1\times\dots\times\pi_k)_\omega^\vee\simeq(\pi_1^{\vee}\times\dots\times\pi_k^{\vee})_{\omega^{-1}}.
\end{equation}
\item \label{part: cass} We have
\begin{equation} \label{eq: w0conjind}
(\pi_k\times\pi_{k-1}\times\dots\times\pi_1)_\omega\simeq i_{w_0(\beta)}(\,^{w_0}(\pi_1\otimes\pi_2\otimes\dots\otimes\pi_k)_\omega)\simeq\overline i_\beta((\pi_1\otimes\pi_2\otimes\dots\otimes\pi_k)_\omega),
\end{equation}
where $w_0\in W_r\subset \bas{G}_r$ denotes the longest element and $w_0(\beta):=(r_k,r_{k-1},\dots,r_1)$.
\item Let $k=3$. Then,
\begin{equation} \label{eq: BZassoc}
(\pi_1\times\pi_2\times\pi_3)_\omega\simeq ((\pi_1\times\pi_2)\times\pi_3)_\omega\simeq (\pi_1\times(\pi_2\times\pi_3))_\omega,
\end{equation}
where $(\pi_1\times\pi_2)$ and $(\pi_2\times\pi_3)$ denote the corresponding Bernstein-Zelevinsky products by omitting the subscript of characters.

\item \label{part: JHjac} (cf.\ \cite{MR874050}*{Lemma 5.4.(iii)})
For any $\pi\in\Reps^{\fl}_\omega(G_\beta)$ we have $\JH(i_\beta(\pi))=\JH(\overline i_\beta(\pi))$.

In particular, if $(\pi_1\times\pi_2)_\omega$ is irreducible, then $(\pi_1\times\pi_2)_\omega\simeq(\pi_2\times\pi_1)_\omega$.

\item \label{part: BZfrobrec}
Let $\pi\in\Reps_\omega(G_r)$ and $\pi_i\in\Reps_{\omega_i}(G_{r_i})$, $i=1,\dots,k$. Then,
\begin{subequations}	
\begin{equation} \label{eq: BZfrob1}
\Hom_{G_r}(\pi,(\pi_1\times\pi_2\times\dots\times\pi_k)_\omega)\simeq\Hom_{G_\beta}(r_\beta(\pi),(\pi_1\otimes\pi_2\otimes\dots\otimes\pi_k)_\omega).
\end{equation}
If $\pi$ and the $\pi_i$'s are admissible, then
\begin{equation} \label{eq: BZfrob2}
\Hom_{G_r}((\pi_k\times\pi_{k-1}\times\dots\times\pi_1)_\omega,\pi)\simeq\Hom_{G_\beta}((\pi_1\otimes\pi_2\otimes\dots\otimes\pi_k)_\omega,r_\beta(\pi)).
\end{equation}
\end{subequations}
\end{enumerate}
\end{proposition}

\begin{proof}
Part \ref{part: mea} follows from the multiexactness and multiadditivity of the metaplectic tensor product, together with the exactness of $i_\beta$.

Part \ref{part: BZcontra} follows the corresponding properties of $i_\beta$, the metaplectic tensor product and the ordinary tensor product.

For part \ref{part: cass}, the first isomorphism in \eqref{eq: w0conjind} follows from Proposition \ref{propweylinvariant}
while the second one follows from the fact that $P_\beta^-={}^{w_0}P_{w_0(\beta)}$.

The relation \eqref{eq: BZassoc} follows from Corollary \ref{cor: associativeMTP}, Proposition \ref{propcomatibleMTPparablicJacquet}
and transitivity of parabolic induction.

Part \ref{part: JHjac} is proved as in \cite{MR874050}*{Lemma 5.4.(iii)}.
The proof relies on the Langlands classification, which in the covering case is proved in \cites{MR3151110, MR3516189}.

Part \ref{part: BZfrobrec} follows from Frobenius reciprocity and Casselman's pairing.
\end{proof}

\begin{remark}
In fact, the relation \eqref{eq: BZfrob2} holds without the admissibility assumption.
In other words, Bernstein's second adjointness holds in the covering case as well, essentially with the same proof.
However, this is unnecessary for the purpose of this paper.
\end{remark}

\subsection{An irreducibility criterion}

Let $\beta=(r_1,\dots,r_k)$ be a composition of $r$ and let $\gamma$ be a refinement of $\beta$.
The following result is proved as in the linear case.

\begin{lemma}[cf.\ \cite{MR3178433}*{Lemme 2.5}]
Let $\pi\in\Reps^{\fl}_{\epsilon}(G_\beta)$ and $\sigma\in\Irr_{\epsilon}(G_\gamma)$. Suppose that the following conditions are satisfied.
\begin{enumerate}
\item $\pi$ is a subrepresentation of $\Pind\gamma\beta(\sigma)$ and a quotient of $\Pindbar\gamma\beta(\sigma)$.		
\item $\sigma$ occurs with multiplicity one in $\JH(\Jac\beta\gamma(\Pind\gamma\beta(\sigma)))$.
\end{enumerate}
Then, $\pi$ is irreducible.	
\end{lemma}

\begin{corollary} \label{cor: irredcrit}
Let $\pi\in\Reps^{\fl}_{\epsilon}(G_r)$ and $\sigma_i\in\Irr_{\epsilon}(G_{r_i})$ for each $i=1,\dots,k$. Suppose that
\begin{enumerate}
\item $\pi$ is a subrepresentation of $(\sigma_1\times\dots\times\sigma_k)_\omega$ and a quotient of $(\sigma_k\times\dots\times\sigma_1)_\omega$.		
\item $(\sigma_1\otimes\dots\otimes\sigma_k)_\omega$ occurs with multiplicity one in $\JH(r_\beta((\sigma_1\times\dots\times\sigma_k)_\omega))$.
\end{enumerate}
Then, $\pi$ is irreducible.	
\end{corollary}

\subsection{Cuspidal support}

Let  $\beta=(r_1,\dots,r_k)$ be a composition of $r$ and let $\pi\in\Reps_{\epsilon}(G_\beta)$.
Recall that by definition, $\pi$ is cuspidal if all its matrix coefficients are compactly supported modulo the center,
or equivalently all its proper Jacquet modules are trivial.
If $\pi\in\Irr_\epsilon(G_\beta)$, this is equivalent to $\pi$ not occurring as a subrepresentation of a proper parabolic induction.
Note that if $\pi=(\pi_1\otimes\dots\otimes\pi_k)_\omega$ with $\pi_i\in\Irr_{\epsilon}(G_{r_i})$ for $i=1,\dots,k$ and
$\omega$ a compatible character of $Z(G_r)$, then $\pi$ is cuspidal if and only if each $\pi_i$ is cuspidal.

A \emph{cuspidal pair} of $G_r$ consists of a standard Levi subgroup $G_\beta$ of $G_r$ and an irreducible cuspidal
representation $\rho$ of $G_\beta$. Two cuspidal pairs $(G_\beta,\rho)$ and $(G_{\beta'},\rho')$ are called
\emph{associated} if there exists $w\in W_r$ such that $wG_\beta w^{-1}=G_{\beta'}$ and $\,^{w}\rho\simeq\rho'$.
As in the linear case we have (cf.\  \cite{MR579172}*{\S 2})

\begin{proposition}\label{propcuspidalsupport}
Let $\pi\in\Irr_\epsilon(G_r).$
Then, there exists a cuspidal pair $(G_\beta,\rho)$ of $G_\beta$, unique up to association,
such that $\pi$ is a subrepresentation of $\Pind\beta{(r)}(\rho)$. Moreover, the following are equivalent.
\begin{enumerate}
\item $(G_\beta,\rho)$ and $(G_{\beta'},\rho')$ are associated.	
\item $\JH(\Pind\beta{(r)}(\rho))=\JH(\Pind{\beta'}{(r)}(\rho'))$.	
\item $\JH(\Pind\beta{(r)}(\rho))\cap\JH(\Pind{\beta'}{(r)}(\rho'))\neq \emptyset$.	
\item $\Hom_{G_r}(\Pind\beta{(r)}(\rho),\Pind{\beta'}{(r)}(\rho'))\neq 0$.
\end{enumerate}	
\end{proposition}

Let $(G_\beta,\rho)$ be a cuspidal pair.
Write $\rho=(\rho_1\otimes\dots\otimes\rho_k)_\omega$ with $\omega$ a compatible character of $Z(G_r)$.
Then, $\rho_i$ is an irreducible cuspidal representation of $G_{r_i}$ for each $i$.
Thus, the set of associated classes of cuspidal pairs $(G_\beta,\rho)$ corresponds bijectively to the set of pairs
$([\rho_1]+\dots+[\rho_k],\omega)$ where $[\rho_1]+\dots+[\rho_k]$ is a multiset of weak equivalence classes of irreducible cuspidal representations
and $\omega$ is a compatible character of $Z(G_r)$ with respect to $([\rho_1],\dots,[\rho_k])$.
(The latter makes sense since the notion of compatibility depends only on the weak equivalence classes of the $\rho_i$'s.)
We denote by
\[
\operatorname{Cusp}(\pi)=([\rho_1]+\dots+[\rho_k],\omega)
\]
the \emph{cuspidal support} of $\pi$. Sometimes it is convenient to omit $\omega$. In this spirit we denote by
\[
\WCusp(\pi)=[\rho_1]+\dots+[\rho_k]
\]
the \emph{weak cuspidal support} of $\pi$.
It depends only on the weak equivalence class of $\pi$. Therefore, we also write $\WCusp([\pi])$.

We view the set of multisets of weak equivalence classes of irreducible cuspidal representations as an ordered monoid.

\subsection{Geometric lemma}\label{subsectiongeolemma}

Let $\beta=(r_1,\dots,r_k)$ and $\gamma=(s_1,\dots,s_{l})$ be two compositions of $r$.
(Unlike before, we do not assume that $\gamma$ is a refinement of $\beta$.)
By definition, $\beta\cap\gamma$ is the common refinement of $\beta$ and $\gamma$, so that
$\bas{G}_{\beta\cap\gamma}=\bas{G}_{\beta}\cap\bas{G}_{\gamma}$.

Let $t_i=r_1+\dots+r_i$, $i=0,\dots,k$ and $u_j=s_1+\dots+s_j$, $j=0,\dots,l$.
Define
\begin{align*}
W^{\beta,\gamma}=\{w\in W_r\mid &\ w(i)<w(i+1)\ \text{for all}\ i\notin\{t_1,t_2,\dots,t_{k-1}\};\\
&\ w^{-1}(j)<w^{-1}(j+1)\ \text{for all}\ j\notin\{ u_1,u_2,\dots,u_{l-1}\}\}.
\end{align*}
The geometric lemma of Bernstein--Zelevinsky takes the following form in the case at hand.

\begin{proposition}[\cite{MR579172}*{Theorem 5.2}]\label{propgeolemma}
The functor
$\Jac{(r)}\gamma\circ \Pind\beta{(r)}:\Reps_{\epsilon}(G_\beta)\rightarrow\Reps_{\epsilon}(G_\gamma)$ is glued from the functors
$\Pind{w(\beta)\cap \gamma}\gamma\circ w\circ \Jac\beta{\beta\cap w^{-1}(\gamma)}$, where $w$ ranges over $W^{\beta,\gamma}$.

In particular,
$\Pind{\beta\cap\gamma}\gamma(\Jac\beta{\beta\cap\gamma}(\pi))$ is a quotient of $\Jac{(r)}\gamma(\Pind\beta{(r)}(\pi))$ for $\pi\in\Reps_{\epsilon}(G_\beta)$.
\end{proposition}

As usual, it is possible to give a ``coordinate version'' of the geometric lemma (or more precisely, after semisimplification) --
cf.\ \cite{MR584084}*{\S1.6}.
For instance, as in \cite{MR2527415}*{Proposition 2.1 and Corollaire 2.2} we can conclude the following.

\begin{lemma} \label{lem: wcuspmult1}
Let $\pi=(\pi_1\otimes\dots\otimes\pi_k)_\omega\in\Irr_\epsilon(G_\beta)$.
Assume that for every $i=1,\dots,k$, every composition $\beta_i=(b_{i1},b_{i2})$ of $r_i$ with at most two blocks,
and every irreducible subquotient
$\sigma=(\sigma_1\otimes\sigma_2)_{\omega_i}$ of $\Jac{(r_i)}{\beta_i}(\pi_i)$, we have
\[
\WCusp(\sigma_2)\nleq\sum_{i<j\leq k}\WCusp(\pi_j).
\]
Then, $\pi$ occurs with multiplicity one in $\JH(\Jac{(r)}\beta(\Pind\beta{(r)}(\pi)))$.
Moreover, $\Pind\beta{(r)}(\pi)$ (resp., $\Pindbar\beta{(r)}(\pi)$) has a unique irreducible subrepresentation (resp., quotient)
and it occurs with multiplicity one in $\JH(\Pind\beta{(r)}(\pi))=\JH(\Pindbar\beta{(r)}(\pi))$.
\end{lemma}

By Corollary \ref{cor: irredcrit}, we infer

\begin{corollary} \label{cor: irredcritsimple}
Let $\pi\in\Reps^{\fl}_{\epsilon}(G_r)$ and $\sigma_i\in\Irr_{\epsilon}(G_{r_i})$, $i=1,\dots,k$.
Assume that
\begin{enumerate}
\item $\pi$ is a subrepresentation of $(\sigma_1\times\dots\times\sigma_k)_\omega$ and a quotient of $(\sigma_k\times\dots\times\sigma_1)_\omega$.
\item $\WCusp(\sigma_i)\cap\WCusp(\sigma_j)=\emptyset$ for $1\leq i<j\leq k$.	
\end{enumerate}
Then, $\pi$ is irreducible.
\end{corollary}

\section{An analogue of a result of Olshanski} \label{sec: olshan}

In this section we will prove the fundamental irreducibility result for parabolic induction
from irreducible cuspidal representations in the corank one case, following Olshanski \cite{MR0499010} and Bernstein--Zelevinsky \cite{MR584084}
in the linear case. The main ingredient is the analysis of intertwining operators.

\subsection{Computation of residue of intertwining operator} \label{sec: inter}
Consider the partition $\beta=(r,r)$ of $2r$.
Let $\bas{G}=\bas{G}_{2r}$, $\bas{P}=\bas{P}_\beta$, $\bas{M}=\bas{G}_\beta\simeq\bas{G}_r\times\bas{G}_r$ and $\bas{U}=\bas{U}_\beta$.
As usual, $\pr:G\rightarrow\bas{G}$ is the Kazhdan-Patterson covering group of $\bas{G}$ and $P$, $M$, $U$ are the corresponding inverse images in $G$.

Let $(\pi,V)$ be a genuine admissible representation of $M$. Define
\[
I_P(\pi,s)=i_\beta(\pi\cdot(\nu^{s/2}\otimes\nu^{-s/2})),\ s\in\C
\]
where $\nu$ is the character $\abs{\det}\circ\pr$ of $G_r$.

The normalizer $N_G(M)$ of $M$ in $G$ contains $M$ as an index two subgroup. Fix $w\in N_G(M)- M$
and consider the intertwining operator
\[
M(w,s):I_P(\pi,s)\rightarrow I_P(\pi^w,-s).
\]
It is defined for $\Re s\gg0$ by
\[
M(w,s)f(g)=\int_{\bas{U}}f(w\bas{u}g)\ d\bas{u}
\]
and admits a meromorphic continuation to a rational function in $q^{-s}$ \cite{MR3053009}*{Th\'eor\`eme 2.4.1}.
(As usual, we view $\bas{U}$ as a subgroup of $U$ via the canonical lifting $\sctn_U$ (see \eqref{eq: cansec}).)
Moreover, if $\pi$ is tempered, then the integral above converges for $\Re s>0$.

Let $\Irr_\epsilon(M)^w$ be the set of $\pi\in\Irr_\epsilon(M)$ such that $\pi^w\simeq\pi$.

Fix a genuine character $\omega$ of $Z(G_r)$. Let $\omega^\beta$ be the character $\omega\otimes\omega$
of $Z(G^\beta)$ and fix a character $\omega_\beta$ of $Z(M)$ compatible with $\omega^\beta$.

\begin{remark}
By Proposition \ref{propweylinvariant}, the map
\[
\tau\mapsto\tnsr_{\beta,\omega^\beta}^{\omega_\beta}(\tau\otimes\tau)
\]
is a bijection between $\Irr_\omega(G_r)$ and $\Irr_{\omega_\beta}(M)^w$,
preserving cuspidality, square-irreducibility and temperedness.
\end{remark}

The following is an analogue of a result of Olshanski in the linear case \cite{MR0499010}.\footnote{In fact, Olshanski
considered inner forms of $\GL$ as well.}
It will be proved in \S\ref{sec: pfofthmintermain} below.

\begin{theorem} \label{thm: intermain}
Let $\rho\in\Irr_\omega^{\sqr}(G_r)$ and $\pi=\tnsr_{\beta,\omega^\beta}^{\omega_\beta}(\rho\otimes\rho)$.
Then, the intertwining operator $M(w,s)$ (which is holomorphic for $\Re s>0$) has a simple pole at $s=0$.
Moreover, let
\begin{equation} \label{def: M*w}
M^*(w)=\lim_{s\rightarrow0}(1-q^{-rns})M(w,s)
\end{equation}
where the Haar measure on $\bas{U}$ defining $M(w,s)$ depends on the formal degree of $\rho$ (see below).
Let $T:\pi^w\rightarrow\pi$ be an intertwining operator (uniquely determined up to a sign) such that $T^2=\pi(w^{-2})$.
Then, we have
\begin{equation} \label{eq: mainresint}
I_P(T,0)M^*(w)=\pm [Z(G_r):Z_{r,\sml}]^{-\frac12}\cdot\id_{I_P(\pi,0)}.
\end{equation}
\end{theorem}

The Haar measure on $\bas{U}$ defining $M(w,s)$ is specified as follows.
We take the usual Haar measure on $F^{\times}$ such that $\vol(\ordr^{\times})=1$
and its pushforward to $\bas{Z}_{r,\sml}$ via $\lambda\mapsto\lambda^nI_r$.
Together with the formal degree $d_\rho^{Z_{r,\sml}\bs G_r}$ of $\rho$,
which is a Haar measure on $Z_{r,\sml}\bs G_r\simeq\bas{Z}_{r,\sml}\bs\bas{G}_r$,
this determines a Haar measure $dx$ on $\bas{G}_r$. In turn, we obtain a Haar measure
$\abs{\det x}^r\ dx$ on the space of $r\times r$-matrices over $F$, which
we identify with $\bas{U}$ in the usual way.

\begin{remark} \label{rem: indepw}
The set $\Irr_\epsilon(M)^w$ and the validity of Theorem \ref{thm: intermain}
do not depend on the choice of $w\in N_G(M)- M$.
\end{remark}

\subsection{\texorpdfstring{$\theta$}{theta}-integrable representations} \label{sec: thetaint}

Consider the following situation.
Let $G$ be a unimodular $\ell$-group and let $\theta:G\rightarrow G$ be a measure preserving automorphism of $\ell$-groups.
Let $G^\theta$ be the fixed point subgroup of $\theta$ and assume that $G^\theta$ is unimodular as well.
We are also given a central, $\theta$-stable subgroup $B$ of $G$.

For any representation $(\pi,V)$ of $G$, let $\theta(\pi)$ be the representation on $V$ given by
\[
\theta(\pi)(g)=\pi(\theta(g)),\ g\in G.
\]
In particular, if $\theta(g)={}^xg$ for some $x\in G$, then $\theta(\pi)=\pi^x$.
\begin{definition}
We say that $\pi$ is $\theta$-integrable if the following conditions hold.
\begin{enumerate}
\item $\pi$ is admissible (but not necessarily irreducible).
\item $\pi$ has a central character, which is $\theta$-invariant.
\item The integral
\begin{equation} \label{eq: twstmc}
\int_{BG^\theta\bs G}\sprod{\theta(\pi)(g)v}{\pi^\vee(g)v^\vee}\ dg
\end{equation}
converges for all $v\in V$, $v^\vee\in V^\vee$.
\end{enumerate}
\end{definition}

Assume that $\pi$ is $\theta$-integrable.
Then, \eqref{eq: twstmc} defines a $G$-invariant pairing for the representation $\theta(\pi)\otimes\pi^\vee$.
Thus, we can write it as $\sprod{\twstoper_\pi^\theta v}{v^\vee}$, where $\twstoper_\pi^\theta:\theta(\pi)\rightarrow\pi$ is an intertwining operator that
depends on the choice of an invariant measure on $BG^\theta\bs G$.
If we want to emphasize it, we will write $\twstoper_\pi^{\theta,dg}$.
Note that in principle $\twstoper_\pi^\theta$ could be trivial even if $\theta(\pi)\simeq\pi$.

Clearly, if $\tau$ is a subrepresentation of $\pi$, then $\tau$ is $\theta$-integrable and $\twstoper_\tau^\theta$ is the restriction of $\twstoper_\pi^\theta$ to $\tau$.

\begin{example} \label{ex: fd}
Suppose that $G=H\times H$, $\theta:G\rightarrow G$ is the involution $\theta(h_1,h_2)=(h_2,h_1)$ and
$B=A\times A$ where $A$ is a cocompact subgroup of $Z(H)$.
Let $\pi=\sigma\otimes\sigma$ where $\sigma$ is an irreducible representation of $H$.
Then, $\pi$ is $\theta$-integrable if and only if $\sigma$ is essentially square-integrable,
in which case, by the Schur orthogonality relations \eqref{eq: Schur}
\[
\twstoper_\pi^{\theta,d_\sigma^{A\bs H}}(v_1\otimes v_2)=v_2\otimes v_1
\]
where we recall that $d_\sigma^{A\bs H}$ is the formal degree of $\sigma$ and we identify $A\bs H$ with $BG^\theta\bs G$.
\end{example}

Assume that $\theta$ is an involution. Then, $\pi$ is $\theta$-integrable if and only if
$\theta(\pi)$ is $\theta$-integrable, in which case $\twstoper_{\theta(\pi)}^\theta=\twstoper_\pi^\theta$.
More generally, we have the following.

\begin{lemma} \label{lem: thetapm}
Suppose that there exist elements $x,y\in G$ such that $\theta^{-1}(g)=\theta(g)^x$ for all $g\in G$
and $x=\theta(y)^{-1}y$. Then,
\begin{enumerate}
\item $y$ normalizes $G^\theta$.
\item Suppose that $\pi$ is $\theta$-integrable. Then, $\theta(\pi)$ is $\theta^{-1}$-integrable and
$\twstoper_{\theta(\pi)}^{\theta^{-1}}=c^{-1}\twstoper_\pi^\theta\circ\pi(x)$
where $c=\modulus_{G^\theta}(y)$.
\end{enumerate}
\end{lemma}

\begin{proof}
The first part is straightforward. Making a change of variables we get
\begin{align*}
\sprod{\twstoper_\pi^\theta\circ\pi(x)v}{v^\vee}&=
\int_{BG^\theta\bs G}\sprod{\pi(\theta(g)x)v}{\pi^\vee(g)v^\vee}\ dg\\=
\int_{BG^\theta\bs G}\sprod{\pi(x\theta^{-1}(g))v}{\pi^\vee(g)v^\vee}\ dg&=
\int_{BG^\theta\bs G}\sprod{\pi(xg)v}{\pi^\vee(\theta(g))v^\vee}\ dg\\=
\int_{BG^\theta\bs G}\sprod{\pi(yg)v}{\pi^\vee(\theta(yg))v^\vee}\ dg&=
c\int_{BG^\theta\bs G}\sprod{\pi(g)v}{\pi^\vee(\theta(g))v^\vee}\ dg\\=
c\int_{BG^\theta\bs G}\sprod{\pi(g)v}{(\theta(\pi))^\vee(g)v^\vee}\ dg&=
c\sprod{\twstoper_{\theta(\pi)}^{\theta^{-1}}v}{v^\vee}.
\end{align*}
The lemma follows.
\end{proof}

For the next proposition, let $H$ be a normal, $\theta$-stable subgroup of $G$ of finite index
and let $\theta_H$ be the restriction of $\theta$ to $H$.
Assume that $H$ contains $B$.
Let $\Gamma=H\bs G$ and $\Gamma^\theta=H\bs HG^\theta\simeq H^\theta\bs G^\theta$.
(Hopefully, this notation will not create ambiguity. We will not consider the automorphism on $\Gamma$ induced by $\theta$ and its fixed point subgroup
which contains $\Gamma^\theta$.)
For any representation $\tau$ of $H$ let
\[
\Gamma_\tau^\theta=\{x\in\Gamma\mid\theta(\tau^x)\simeq\tau^x\}.
\]
This is a (possibly empty) right $H\bs HG^\theta Z_G(H)$-invariant subset of $\Gamma$.

\begin{proposition} \label{prop: indfd}
Let $\tau$ be an admissible irreducible representation of $H$.
Assume that $\tau^\gamma$ is $\theta_H$-integrable for every $\gamma\in\Gamma$.
Then, $\pi$ is $\theta$-integrable.
Assume that
\begin{subequations}
\begin{equation}\label{eq: Gammathetarho}
\Gamma_\tau^\theta=H\bs HG^\theta Z_G(H)
\end{equation}
(and in particular, $\tau$ is $\theta$-invariant), and that
\begin{equation} \label{eq: ngtc}
H\cap G^\theta Z_G(H)=H^\theta Z(H).
\end{equation}
Let $\Omega$ be the (finite) set of cosets $Z_G(H)/Z(H) Z_{G^\theta}(H)$.\footnote{In fact, by \eqref{eq: ngtc} we have
$\Omega=Z_G(H)/Z_{HG^\theta}(H)$.}
Let $\Pi=\Ind_H^G\tau$. Then,
\begin{equation} \label{eq: inds}
\twstoper_\Pi^\theta\varphi(g)=\sum_{x\in\Omega}\twstoper_\tau^{\theta_H}\varphi(x\theta(x)^{-1}\theta(g))
\end{equation}
provided that the invariant measure on $BH^\theta\bs H$ is the restriction of the invariant measure on $BG^\theta\bs G$.

Moreover, assume that $\theta_H(\tau^\gamma)$ is $\theta_H^{-1}$-integrable for every $\gamma\in\Gamma$,
\begin{equation} \label{eq: zzzh}
\{z\in Z_G(H)\mid z\theta(z)^{-1}\in Z(Z_G(H))\}=Z(H)Z_{G^\theta}(H),
\end{equation}
and there exists a finite cyclic subgroup $A$ of $B$ such that
\begin{equation} \label{eq: abgen}
\text{$Z_G(H)/A$ is abelian and $\omega_{\tau}\rest_{A}$ is faithful.}
\end{equation}
Then, $\theta(\Pi)$ is $\theta^{-1}$-integrable and
\begin{equation} \label{eq: composes}
\twstoper_\Pi^\theta\twstoper_{\theta(\Pi)}^{\theta^{-1}}=\#\Omega\cdot
\Ind_H^G(\twstoper_{\tau}^{\theta_H}\twstoper_{\theta_H(\tau)}^{\theta_H^{-1}}).
\end{equation}
\end{subequations}
\end{proposition}

\begin{proof}
We will write the standard pairing on $\Pi\times\Pi^\vee$ as $\dsprod{\cdot}{\cdot}$, in order to distinguish it from the
standard pairing on $\tau\otimes\tau^\vee$.

Since $H$ is normal,
\begin{equation} \label{eq: str123}
\int_{BG^\theta\bs G}\dsprod{\Pi(\theta(g))\varphi}{\Pi^\vee(g)\varphi^\vee}\ dg=
\sum_{\gamma\in\Gamma^\theta\bs\Gamma}\int_{BH^\theta\bs H}\dsprod{\Pi(\theta(h\gamma))\varphi}{\Pi^\vee(h\gamma)\varphi^\vee}\ dh.
\end{equation}
Identify $\Pi^\vee$ with $\Ind_H^G\tau^\vee$ via the pairing
\[
\dsprod{\varphi}{\varphi^\vee}=\sum_{\gamma\in\Gamma}\sprod{\varphi(\gamma)}{\varphi^\vee(\gamma)}.
\]
Consider
\[
\int_{BH^\theta\bs H}\dsprod{\Pi(\theta(h))\varphi}{\Pi^\vee(h)\varphi^\vee}\ dh.
\]
It can be written as
\begin{equation} \label{eq: str345}
\begin{aligned}
\int_{BH^\theta\bs H}\sum_{\gamma\in\Gamma}\sprod{\varphi(\gamma\theta(h))}{\varphi^\vee(\gamma h)}\ dh&=
\int_{BH^\theta\bs H}\sum_{\gamma\in\Gamma}\sprod{\tau^\gamma(\theta(h))\varphi(\gamma)}{(\tau^\gamma)^\vee(h)\varphi^\vee(\gamma)}\ dh\\&=
\sum_{\gamma\in\Gamma}\int_{BH^\theta\bs H}\sprod{\tau^\gamma(\theta(h))\varphi(\gamma)}{(\tau^\gamma)^\vee(h)\varphi^\vee(\gamma)}\ dh.
\end{aligned}
\end{equation}
The convergence of the left-hand side of \eqref{eq: str123} follows from the convergence of the right-hand side of \eqref{eq: str345}.
This implies the first part of the proposition.

In order to prove \eqref{eq: inds}, it suffices to show that
\begin{equation} \label{eq: nts1}
\sum_{\gamma\in\Gamma}\int_{BH^\theta\bs H}\sprod{\tau^\gamma(\theta(h))\varphi(\gamma)}{(\tau^\gamma)^\vee(h)\varphi^\vee(\gamma)}\ dh
=\sum_{x\in \Omega}\sum_{\gamma\in\Gamma^\theta}
\sprod{\twstoper_\tau^{\theta_H}\varphi(x\theta(\gamma))}{\varphi^\vee(x\gamma)}.
\end{equation}
Note that the right-hand side makes sense.

Since $\tau$ is irreducible, the integral over $h$ on the left-hand side of \eqref{eq: nts1} will be non-zero only if
$\gamma\in\Gamma_\tau^\theta$. By assumption, this occurs only if $\gamma\in H\bs HG^\theta Z_G(H)$. Thus, we obtain (using \eqref{eq: ngtc})
\begin{align*}
&\sum_{\gamma\in H\bs HG^\theta Z_G(H)}\int_{BH^\theta\bs H}\sprod{\tau^\gamma(\theta(h))\varphi(\gamma)}{(\tau^\gamma)^\vee(h)\varphi^\vee(\gamma)}\ dh\\=
&\sum_{\gamma\in H^\theta Z(H)\bs G^\theta Z_G(H)}\int_{BH^\theta\bs H}\sprod{\tau^\gamma(\theta(h))\varphi(\gamma)}{(\tau^\gamma)^\vee(h)\varphi^\vee(\gamma)}\ dh\\=
&\sum_{\gamma\in H^\theta Z(H)\bs G^\theta Z_G(H)}\int_{BH^\theta\bs H}\sprod{\tau(\theta(h))\varphi(\gamma)}{\tau^\vee(h)\varphi^\vee(\gamma)}\ dh\\=
&\sum_{\gamma\in H^\theta Z(H)\bs G^\theta Z_G(H)}\sprod{\twstoper_\tau^{\theta_H}\varphi(\gamma)}{\varphi^\vee(\gamma)}.
\end{align*}
Finally, we can write
\begin{align*}
&\sum_{\gamma\in H^\theta Z(H)\bs G^\theta Z_G(H)}\sprod{\twstoper_\tau^{\theta_H}\varphi(\gamma)}{\varphi^\vee(\gamma)}
\\=&\sum_{\gamma\in Z_G(H)H^\theta\bs Z_G(H)G^\theta}\sum_{x\in Z(H)\bs Z_G(H)}
\sprod{\twstoper_\tau^{\theta_H}\varphi(x\gamma)}{\varphi^\vee(x\gamma)}
\\=&\sum_{\gamma\in Z_{G^\theta}(H)H^\theta\bs G^\theta}\sum_{x\in Z(H)\bs Z_G(H)}
\sprod{\twstoper_\tau^{\theta_H}\varphi(x\gamma)}{\varphi^\vee(x\gamma)}\\=
&\sum_{\gamma\in Z_{G^\theta}(H)H^\theta\bs G^\theta}\sum_{x\in \Omega}
\sum_{x'\in Z(H)^\theta\bs Z_{G^\theta}(H)}
\sprod{\twstoper_\tau^{\theta_H}\varphi(xx'\gamma)}{\varphi^\vee(xx'\gamma)}\\
=&\sum_{x\in \Omega}
\sum_{\gamma\in H^\theta\bs G^\theta}
\sprod{\twstoper_\tau^{\theta_H}\varphi(x\gamma)}{\varphi^\vee(x\gamma)}.
\end{align*}
This implies \eqref{eq: nts1}, and hence \eqref{eq: inds}.

To show \eqref{eq: composes}, we note that $\Omega$ is an abelian group, since by assumption, $Z_G(H)/A$ is abelian.
By the previous part,
\begin{align*}
\twstoper_\Pi^\theta\twstoper_{\theta(\Pi)}^{\theta^{-1}}\varphi(g)=&
\sum_{x,y\in\Omega}\twstoper_{\tau}^{\theta_H}\twstoper_{\theta_H(\tau)}^{\theta_H^{-1}}\varphi(x\theta(x^{-1}y)y^{-1}g)\\=
&\sum_{z\in\Omega}\sum_{x,y\in\Omega\mid x^{-1}y=z}
\twstoper_{\tau}^{\theta_H}\twstoper_{\theta_H(\tau)}^{\theta_H^{-1}}\varphi({}^x(\theta(z)z^{-1})g)\\=
&\sum_{z\in\Omega}\sum_{x,y\in\Omega\mid x^{-1}y=z}
\omega_\tau([x,\theta(z)z^{-1}])\twstoper_{\tau}^{\theta_H}\twstoper_{\theta_H(\tau)}^{\theta_H^{-1}}\varphi(\theta(z)z^{-1}g).
\end{align*}
Since $\omega_\tau\rest_{A}$ is faithful, the inner sum vanishes unless $\theta(z)z^{-1}\in Z(Z_G(H))$.
Hence, by \eqref{eq: zzzh} only $z=1$ contributes. We remain with
\[
\abs{\Omega}\cdot\twstoper_{\tau}^{\theta_H}\twstoper_{\theta_H(\tau)}^{\theta_H^{-1}}\varphi(g).
\]
The proposition follows.
\end{proof}

\subsection{Proof of Theorem \texorpdfstring{\ref{thm: intermain}}{123}} \label{sec: pfofthmintermain}

We go to the setup of \S\ref{sec: inter}.
By Remark \ref{rem: indepw} we may work with $w\in N_G(M)- M$ of our choice.

Identify $\bas{U}$ with the space $\Mat_r$ of $r\times r$-matrices over $F$ via $X\mapsto\sm{I_r}{X}{}{I_r}$.
Thus, we view $\sctn_U$ as a group embedding $\varsigma_+:\Mat_r\rightarrow U$.
Similarly, we identify the unipotent radical $\bas{U}^-$ of the parabolic subgroup opposite to $\bas{P}$ with $\Mat_r$ via
$X\mapsto\sm{I_r}{}{X}{I_r}$ and consider $\sctn_{U^-}$ as a group embedding
$\varsigma_-:\Mat_r\rightarrow U^-$.
From now on let
\[
w=\varsigma_+(-I_r)\varsigma_-(I_r)\varsigma_+(-I_r),
\]
so that $\bas{w}=\pr(w)=\sm{}{-I_r}{I_r}{}$.

Let $j_1:G_r\rightarrow G$ be the embedding in the upper left corner.
Let $j_2(x)={}^w j_1(x)$.
Note that by \eqref{eqwtzwtg} we have
\[
{}^w j_2(x)={}^{w^2}j_1(x)=(-1,\det\bas{x})_nj_1(x)\ \ x\in G_r.
\]
Hence, by \eqref{eq: commblocks},
\begin{align*}
{}^w(j_1(x)j_2(y))&=(-1,\det\bas{y})_nj_2(x)j_1(y)\\&=(-1,\det\bas{y})_n(\det\bas{x},\det\bas{y})_n^{c'}j_1(y)j_2(x),\ \ x,y\in G_r.
\end{align*}
In particular,
\begin{equation} \label{eq: thetaH}
{}^w(j_1(x)j_2(y))=j_1(y)j_2(x),\ \ x,y\in H_r.
\end{equation}

Let $\bas{\Delta}=\{\diag(g,g)\mid g\in\bas{G}_r\}$.

\begin{lemma} \label{lem: conjw} \
\begin{enumerate}
\item \label{part: centw} The centralizer of $w$ in $M$ is $\Delta$.
\item For any $x\in\bas{G}_r$ we have
\begin{equation} \label{eq: wx+}
w\varsigma_+(x)=\varsigma_+(-x^{-1})j_2(x)j_1(x)^{-1}\varsigma_-(x^{-1}).
\end{equation}
Note that $j_2(x)j_1(x)^{-1}=[w,j_1(x)]$ is well defined.
\item Let $\bas{J}_1=\diag(-I_r,I_r)$. Then,
\begin{equation} \label{eq: w2}
w^2=[\bas{w},\bas{J}_1]_{\incvr}.
\end{equation}
\end{enumerate}
\end{lemma}

\begin{proof}
Clearly, the centralizer of $w$ in $M$ is contained in $\Delta$, since the centralizer of $\bas{w}$ in $\bas{M}$ is $\bas{\Delta}$.
Conversely, for any $x\in\bas{G}_r$ and $y\in\Mat_r$ we have
\begin{align*}
{}^{j_1(x)}\varsigma_+(y)=\varsigma_+(xy),\quad & \varsigma_+(y)^{j_2(x)}=\varsigma_+(yx)	,\\
{}^{j_2(x)}\varsigma_-(y)=\varsigma_-(xy),\quad & \varsigma_-(y)^{j_1(x)}=\varsigma_-(yx).
\end{align*}
Hence, from the definition of $w$, for any $x\in\bas{G}_r$
\[
w^{j_1(x)j_2(x)}=
\varsigma_+(-I_r)^{j_1(x)j_2(x)}\varsigma_-(I_r)^{j_1(x)j_2(x)}\varsigma_+(-I_r)^{j_1(x)j_2(x)}=
\varsigma_+(-I_r)\varsigma_-(I_r)\varsigma_+(-I_r)=w.
\]
It follows that the centralizer of $w$ contains $\Delta$.

Moreover,
\begin{align*}
w\varsigma_+(x)&=wj_1(x)\varsigma_+(I_r)j_1(x)^{-1}=j_2(x)w\varsigma_+(I_r)j_1(x)^{-1}\\&=
j_2(x)\varsigma_+(-I_r)\varsigma_-(I_r)j_1(x)^{-1}=\varsigma_+(-x^{-1})j_2(x)j_1(x)^{-1}\varsigma_-(x^{-1}).
\end{align*}

Finally,
\[
{}^{\bas{J}_1}w={}^{\bas{J}_1}\varsigma_+(-I_r){}^{\bas{J}_1}\varsigma_-(I_r){}^{\bas{J}_1}\varsigma_+(-I_r)=
\varsigma_+(I_r)\varsigma_-(-I_r)\varsigma_+(I_r)=w^{-1}.
\]
Hence,
\[
w^2=[\bas{w},\bas{J}_1]_{\incvr}.
\]
The lemma follows.
\end{proof}

Let $\pi$ be an admissible representation of $M$.
We compute $M(w,s)$ on vectors in $I_P(\pi,s)$ that are supported in the big cell $PU^-$.

Fix $v\in V$ and a Schwartz function $\Phi$ on $\Mat_r$.
Then, there is a unique vector $f_s$ in $I_P(\pi,s)$ such that
\begin{equation} \label{eq: defs}
f_s(\varsigma_-(x))=\Phi(x)v,\ \ \ x\in\Mat_r.
\end{equation}

It follows from \eqref{eq: wx+} that
\begin{align*}
M(w,s)f_s(e)&=\int_{\Mat_r}f_s(w\varsigma_+(x))\ dx=
\int_{\bas{G}_r}f_s(w\varsigma_+(x))\abs{\det x}_F^r\ dx\\
&=\int_{\bas{G}_r}f_s(j_2(x)j_1(x)^{-1}\varsigma_-(x^{-1}))\abs{\det x}_F^r\ dx\\&=
\int_{\bas{G}_r}f_s(j_2(x^{-1})j_1(x)\varsigma_-(x))\abs{\det x}_F^{-r}\ dx\\&=
\int_{\bas{G}_r}\abs{\det x}_F^s\Phi(x)\cdot \pi(j_2(x^{-1})j_1(x))v\ dx.
\end{align*}
Hence, for any $v^\vee\in\pi^\vee$ we have
\begin{align*}
\sprod{M(w,s)f_s(e)}{v^\vee}&=\int_{\bas{G}_r}\abs{\det x}_F^s\Phi(x)\sprod{\pi(j_2(x^{-1})j_1(x))v}{v^\vee}\ dx\\&=
\int_{\bas{G}_r}\abs{\det x}_F^s\Phi(x)\sprod{\pi(j_1(x))v}{\pi^\vee(j_2(x))v^\vee}\ dx.
\end{align*}
Assume from now on that $\pi$ has a central character, which is invariant under conjugation by $w$.
Then, we can write the above as
\begin{equation} \label{eq: withTate}
\int_{\bas{Z}_{r,\sml}\bs\bas{G}_r}\big(\abs{\det x}_F^s\int_{F^{\times}}\abs{\lambda}_F^{rns}\Phi(\lambda^nx)\ d\lambda\big)
\sprod{\pi(j_1(x))v}{\pi^\vee(j_2(x))v^\vee}\ dx.
\end{equation}
From now on we assume that the integral
\[
\int_{Z_{r,\sml}\bs G_r}\sprod{\pi(j_1(x))v}{\pi^\vee(j_2(x))v^\vee}\ dx
\]
converges. The inner integral in \eqref{eq: withTate} can be written as a sum over $\chi\in X_n(F^{\times})$ of Tate integrals with respect to $\Phi$ and $\chi$.
In particular, it is a rational function in $q^s$ with at most a simple pole at $s=0$ whose residue is $\Phi(0)$.
Moreover, let $C'=\max\abs{\Phi}$ and suppose that the support of $\Phi$ is contained in the ball $\{x\in\Mat_r(F)\mid\norm{x}\le C\}$
where $\norm{x}$ is the maximum of the absolute values of the coordinate of $x$. Then, for any $s>0$
and $x\in\bas{G}_r$
\begin{gather*}
\abs{\det x}^s\int_{F^{\times}}\abs{\lambda}^{rs}\abs{\Phi(\lambda x)}\ d\lambda
\le
C'\abs{\det x}^s\int_{\abs{\lambda}\le C\norm{x}^{-1}}\abs{\lambda}^{rs}\ d\lambda\\=
C'\cdot (1-q^{-rs})^{-1}(C\abs{\det x}\norm{x}^{-r})^s\le C'(1-q^{-rs})^{-1}C^s.
\end{gather*}
In particular, the double integral \eqref{eq: withTate} converges for $\Re s>0$.
By a lemma of Rallis (cf.\ \cite{MR1159430}*{Lemma 4.1}) it follows that $M(w,s)$ is holomorphic for $\Re s>0$
and has at most a simple pole at $s=0$. (Note that the argument of [ibid.]\ is valid
for any admissible representation, not necessarily irreducible.)
Moreover, if $M^*(w)$ is as in \eqref{def: M*w}, then we may take the residue at $s=0$ in \eqref{eq: withTate}
inside the inner integral and obtain
\begin{subequations}
\begin{equation} \label{eq: m*}
\sprod{M^*(w)f_0(e)}{v^\vee}=\Phi(0)\cdot
\int_{\bas{Z}_{r,\sml}\bs\bas{G}_r}\sprod{\pi(j_1(x))v}{\pi^\vee(j_2(x))v^\vee}\ dx.
\end{equation}

Let $\theta$ be the automorphism of $M$ (of order two or four, depending on whether or not $-1\equiv_n1$) given by $\theta(m)={}^wm$.
Then, by Lemma \ref{lem: conjw} part \ref{part: centw}, $M^\theta=\Delta$ and hence
\begin{equation} \label{eq: Mmtheta}
M=M^\theta j_1(G_r).
\end{equation}
Therefore, we can write the integral above as
\begin{equation} \label{eq: m*theta}
\int_{Z_{\beta,\sml}M^\theta \bs M}\sprod{\pi(m)v}{\pi^\vee(\theta(m))v^\vee}\ dm=
\int_{Z_{\beta,\sml}M^\theta \bs M}\sprod{\pi(\theta^{-1}(m))v}{\pi^\vee(m)v^\vee}\ dm.
\end{equation}
\end{subequations}

We will apply the above discussion in the following situation.
Let $\rho$ be a genuine, irreducible, essentially square-integrable representation of $G_r$.
Fix a character $\psi$ of $Z_{\beta,\lrg}$ that is consistent with $\omega_\rho$.
Since $(H_r,Z_r)$ is a special pair in $G_r$, we may write $\rho=\LInd_{H_r,\psi,Z_r}^{G_r,\omega_\rho}\sigma$ where
$\sigma\in\Irr_\psi(H_r)$.
Moreover, by Proposition \ref{prop: lind} part \ref{part: fd} we have
\begin{equation} \label{eq: drhosigma}
d_\sigma^{Z_{r,\sml}\bs H_r}=a\cdot d_\rho^{Z_{r,\sml}\bs G_r}\rest_{Z_{r,\sml}\bs H_r}
\end{equation}
where $a=([Z_r:Z_{r,\lrg}][Z(G_r):Z_{r,\sml}])^{\frac12}$.

Consider the irreducible representation $\sigma\otimes\sigma$ of $H_\beta$
and its extension $\tau=\ex{(\sigma\otimes\sigma)}{\omega_\beta}$ to the group $H=Z(G_{2r})H_\beta=Z(M)H_\beta$
(cf.\ Lemma \ref{lem: centext}).
Let $\intrchng_\tau$ be the intertwining operator (cf.\ \eqref{eq: thetaH})
\[
\intrchng_\tau:\tau^w\rightarrow\tau,\quad v_1\otimes v_2\mapsto v_2\otimes v_1.
\]
Let $\Pi=\Ind_H^M\tau$. Consider the intertwining operator
\[
\TwstOper_\Pi:\theta(\Pi)\rightarrow\Pi,\quad\quad\TwstOper_\Pi\varphi(g)=
\sum_{x\in Z_\beta/Z_{\beta,\lrg}Z_\beta^\theta}\intrchng_\tau\varphi(x\theta(x)^{-1}\theta(g)).
\]

Let $b=\#(Z_\beta/Z_{\beta,\lrg}Z_\beta^\theta)=[Z_r:Z_{r,\lrg}]$.
We claim that
\begin{equation} \label{eq: intrltnf}
I_P(\TwstOper_\Pi,0)M^*(w)=a^{-1}b\cdot\id_{I_P(\Pi,0)}.
\end{equation}
By Rallis's lemma, it is enough to show that
\begin{equation} \label{eq: intsclr}
(I_P(\TwstOper_\Pi,0)M^*(w)f)(e)=a^{-1}b\cdot f(e)
\end{equation}
for every $f\in I_P(\Pi,0)$ supported in $PU^-$.
Note that the left-hand side is
\[
\TwstOper_\Pi(M^*(w)f)(e).
\]
We take the Haar measure on $\bas{U}$ as specified following the statement of Theorem \ref{thm: intermain}.
Recall that by \eqref{eq: m*} and \eqref{eq: m*theta}, taking $f=f_0$ with $f_s$ as in \eqref{eq: defs} we have
\[
(M^*(w)f)(e)=\Phi(0)\cdot\twstoper_{\theta(\Pi)}^{\theta^{-1},d_\rho^{Z_{r,\sml}\bs G_r}} v
\]
in the notation of \S\ref{sec: thetaint}, where we identify $Z_{\beta,\sml}M^\theta\bs M$ with $G_r/Z_{r,\sml}$,
Note that $\Pi$ is $\theta$-integrable and $\theta(\Pi)$ is $\theta^{-1}$-integrable by the first part of Proposition \ref{prop: indfd} together with Example \ref{ex: fd}.
Thus,
\begin{equation} \label{eq: int345}
(I_P(\TwstOper_\Pi,0)M^*(w)f)(e)=\Phi(0)\cdot\TwstOper_\Pi\twstoper_{\theta(\Pi)}^{\theta^{-1},d_\rho^{Z_{r,\sml}\bs G_r}} v.
\end{equation}
We will compute the right-hand side using Proposition \ref{prop: indfd} applied with $G=M$,
$B=Z_{\beta,\sml}$, $A=\mu_n$ and $H$, $\theta$ as above.

\begin{lemma}
The conditions \eqref{eq: Gammathetarho}, \eqref{eq: ngtc}, \eqref{eq: zzzh} and \eqref{eq: abgen} of Proposition \ref{prop: indfd}
are satisfied.
\end{lemma}

\begin{proof}
In order to show \eqref{eq: Gammathetarho}, it is enough by \eqref{eq: Mmtheta} to show that
if $x\in G_r$ and $\theta(\tau^{j_1(x)})\simeq\tau^{j_1(x)}$, then $x\in H_rZ_r$.
Note that since $\theta(\tau)\simeq\tau$,
\[
\theta(\tau^{j_1(x)})\simeq\tau^{j_1(x)}\iff\tau^{j_2(x)}\simeq\tau^{j_1(x)}\iff
\sigma^x\simeq\sigma.
\]
Applying \eqref{eq: stabtau} with respect to the special pair $(H_r,Z_r)$ of $G_r$ we conclude that $x\in H_rZ_r$ as required.

By Lemma \ref{lemmaGbetaspecial}, \eqref{eq: znznghn} and \eqref{cntrbs} we have
\[
Z_M(H)=Z_\beta\text{ and }Z(Z_\beta)=Z(M)Z_{\beta,\lrg}=Z(G_{2r})Z_{\beta,\lrg}=Z(H).
\]

Next, we show \eqref{eq: ngtc}.
Since $H=Z(G_{2r})H_\beta$ and $Z(G_{2r})\le M^\theta$ we have
\[
H\cap M^\theta Z_\beta=Z(G_{2r})(H_\beta\cap M^\theta Z_\beta)
\]
and therefore, it is enough to show that
\[
H_\beta\cap M^\theta Z_\beta=H_\beta^\theta Z_{\beta,\lrg}.
\]
Since $H_\beta=(H_\beta\cap\Delta)j_1(H_r)=H_\beta^\theta j_1(H_r)$ we have
\[
H_\beta\cap M^\theta Z_\beta=H_\beta^\theta(j_1(H_r)\cap M^\theta Z_\beta),
\]
and it is easy to see that $j_1(H_r)\cap M^\theta Z_\beta=j_1(H_r)\cap Z_\beta=j_1(Z_{r,\lrg})\le Z_{\beta,\lrg}$.
The relation \eqref{eq: ngtc} follows.

We show \eqref{eq: zzzh}. Since $Z_\beta=(\Delta\cap Z_\beta)j_1(Z_r)=Z_\beta^\theta j_1(Z_r)$, it is enough to show that
\[
\{z\in Z_r\mid j_1(z)\theta(j_1(z))^{-1}\in Z(H)\}=Z_{r,\lrg}.
\]
Suppose that $z'=j_1(z)\theta(j_1(z))^{-1}\in Z(H)$. Write $z'=z_1z_2$ where $z_1\in Z(G_{2r})$ and $z_2\in Z_{\beta,\lrg}$.
Write $\bas{z}_1=\lambda I_{2r}$, $\bas{z}_2=\diag(\mu I_r,\nu I_r)$. Then, $\lambda\mu=\lambda^{-1}\nu^{-1}$ and
$\mu^r\equiv_n\nu^r\equiv_n1$. It follows that $\lambda^{2r}\equiv_n1$. On the other hand,
by \eqref{eq: cntrofcfr}, $\lambda^{2rc'-1}\equiv_n1$, and hence $\lambda\equiv_n1$.
Thus, $z\in Z_{r,\lrg}$ since $\bas{z}=\lambda\mu I_r$, as claimed.

Finally, condition \eqref{eq: abgen} is clear since $\rho$ is genuine.
\end{proof}

We can therefore apply Proposition \ref{prop: indfd}.

Note that $\theta_H$ acts identically on $Z(G_{2r})$ and by ``switching the blocks'' on $H_\beta$ (by \eqref{eq: thetaH}).
Thus, by Example \ref{ex: fd} we have
\[
\twstoper_\tau^{\theta_H,d_\sigma^{Z_{r,\sml}\bs H_r}}=\intrchng_\tau
\]
where we identify $BH^\theta\bs H$ with $H_r/Z_{r,\sml}$.
Hence, by \eqref{eq: inds} and \eqref{eq: drhosigma} we get
\[
\TwstOper_\Pi=\twstoper_\Pi^{\theta,a\cdot d_\rho^{Z_{r,\sml}\bs G_r}}.
\]
where we identify $BG^\theta\bs G$ with $G_r/Z_{r,\sml}$.
Also, since $\intrchng_\tau$ is an involution,
\[
\twstoper_\tau^{\theta_H,d_\sigma^{Z_{r,\sml}\bs H_r}}\twstoper_{\theta_H(\tau)}^{\theta_H^{-1},d_\sigma^{Z_{r,\sml}\bs H_r}}=
\id_\tau.
\]
Together with \eqref{eq: int345} and \eqref{eq: composes} we deduce that
\begin{gather*}
I(\TwstOper_\Pi,0)M^*(w)f(e)
=\Phi(0)\cdot\twstoper_\Pi^{\theta,a\cdot d_\rho^{Z_{r,\sml\bs G_r}}}\twstoper_{\theta(\Pi)}^{\theta^{-1},d_\rho^{Z_{r,\sml}\bs G_r}} v\\
=a^{-1}b\cdot\Phi(0)\cdot\Ind_H^M(
\twstoper_\tau^{\theta_H,d_\sigma^{Z_{r,\sml}\bs H_r}}\twstoper_{\theta_H(\tau)}^{\theta_H^{-1},d_\sigma^{Z_{r,\sml}\bs H_r}})v
=a^{-1}b\cdot\Phi(0)\cdot v=a^{-1}b\cdot f(e).
\end{gather*}
This implies \eqref{eq: intsclr} since the space of sections in $I_P(\Pi,0)$ supported in $PU^-$ is spanned
by $f_0$ where $f_s$ is as in \eqref{eq: defs}.
We infer \eqref{eq: intrltnf}.

By Lemma \ref{lem: thetapm} (which is applicable with $x=w^2$ and $y=J_1$ by \eqref{eq: w2}) we have
\[
\twstoper_{\theta(\Pi)}^{\theta^{-1}}=\twstoper_\Pi^\theta\circ\Pi(w^2).
\]
Thus, as before
\begin{gather*}
\TwstOper_\Pi^2\circ\Pi(w^2)=
(\twstoper_\Pi^{\theta,a\cdot d_\rho^{Z_{r,\sml}\bs G_r}})^2\circ\Pi(w^2)=
\twstoper_\Pi^{\theta,a\cdot d_\rho^{Z_{r,\sml}\bs G_r}}
\twstoper_{\theta(\Pi)}^{\theta^{-1},a\cdot d_\rho^{Z_{r,\sml}\bs G_r}}\\=
b\cdot\Ind_H^M
(\twstoper_\tau^{\theta_H,d_\sigma^{Z_{r,\sml}\bs H_r}}\twstoper_{\theta_H(\tau)}^{\theta_H^{-1},d_\sigma^{Z_{r,\sml}\bs H_r}})
=b\cdot\id_\Pi.
\end{gather*}

Recall that $\Pi$ is a semisimple, isotypic representation of $M$ of type $\tnsr_{\beta,\omega^\beta}^{\omega_\beta}(\rho\otimes\rho)$.
Also, $\TwstOper_\Pi$ preserves any subrepresentation of $\Pi$, since it is proportional to $\twstoper_\Pi^\theta$.
Thus, if $\pi$ is any irreducible constituent of $\Pi$, then the restriction $T$ of $b^{-\frac12}\cdot\TwstOper_\Pi$ to $\pi$
satisfies $T^2=\pi(w)^{-2}$.
Observe that $a^{-1}b^{\frac12}=[Z(G_r):Z_{r,\sml}]^{-\frac12}$.
Theorem \ref{thm: intermain} follows.

\subsection{Bernstein-Zelevinsky product for two cuspidal representations}\label{subsectioncuspidalinduction}

We can now state the fundamental irreducibility result for the Bernstein--Zelevinsky product for two cuspidal representations.

\begin{proposition}\label{prop: BZPtwocuspidals} \
\begin{enumerate}
\item Let $\rho_i$ be irreducible genuine cuspidal representations of $G_{r_i}$, $i=1,2$.
Then, $(\rho_1\times\rho_2)_\omega$ is irreducible unless $r_1=r_2$ and $[\rho_1\nu^s]=[\rho_2]$
for some real number $s$.
\item Let $\rho$ be a genuine irreducible cuspidal representation of $G_r$.
Then, there is a unique positive real number $s_\rho$ such that $(\rho\times\rho\nu^{s_\rho})_\omega$ is reducible.
Moreover, for $s\in\R$, $(\rho\times\rho\nu^s)_\omega$ is reducible if and only if $s=\pm s_\rho$.
\end{enumerate}	
\end{proposition}

The first statement follows from \cite{MR2567785}*{Th\'eor\`eme VII 1.3}
whose proof works directly in our case in view of Proposition \ref{propweylinvariant}.

To prove the second part, we may assume without loss of generality that $\rho$ is unitary.
	
The irreducibility of $(\rho\times\rho)_\omega$ is deduced from Theorem \ref{thm: intermain}
by a standard result.
(See Proposition 6.3 by Savin in \url{https://doi.org/10.48550/arXiv.1706.05145} for the covering case.)
Alternatively, we can use the argument of \cite{MR3573961}*{Lemma A.5} (which is special for $\GL$).

The existence of $s>0$ such that $(\rho\times\rho\nu^s)_\omega$ (or equivalently,
$\pi_s:=(\rho\nu^{-s/2}\times\rho\nu^{s/2})_\omega$) is reducible also follows from Theorem \ref{thm: intermain}.
Indeed, assume on the contrary that $\pi_s$ is irreducible for all $s>0$.
Then, by a standard argument (cf.\ \cite{MR584084}*{\S 1.11}) $\pi_s$ would be
unitarizable for all $s>0$, in contradiction to the fact that for $s$ large,
the matrix coefficients of $\pi_s$ are not bounded.

It remains to show the uniqueness of $s_\rho$.
(Clearly, $-s_\rho$ is then the unique negative real number $s$ such that $(\rho\times\rho\nu^s)_\omega$ is reducible.)
We will return to this point in the next section.

For the time being, by abuse of notation we write
\[
\nu_\rho=\nu^s
\]
for \emph{some} $s>0$ such that  $(\rho\times\rho\nu^s)_\omega$ is reducible.

\begin{remark}
Let $\pi$ be the irreducible, essentially square-integrable representation of $\bas{G}_r$ corresponding
to $\rho$ under the metaplectic correspondence introduced in \cite{MR876160}.
By Bernstein--Zelevinsky theory in the linear case, $\pi$ corresponds to a segment, say of length $m$.
One can prove that $s_\rho=\frac{m}{n}$. Details will be given elsewhere.
\end{remark}

\section{Segments, multisegments, standard modules and classification} \label{sec: seg}

In this section, we complete the analogue of the classification scheme of Bernstein--Zelevinsky
for the metaplectic groups $G_r$, $r\ge0$.
We have already developed the necessary ingredients in the previous sections.
The rest of the argument follows \cite{MR3573961}*{Appendix A} (and the references therein), so we will only sketch it.
The argument of [ibid.] uses among other things Casselman's criterion for square-integrability (cf.\ \cite{MR3151110}*{Theorem 3.4}) and repeated application of Corollary \ref{cor: irredcrit}
and Lemma \ref{lem: wcuspmult1} (see Corollary \ref{cor: irredcritsimple} in the simplest case).

\subsection{Segments}

The following definition is motivated by the linear case.

\begin{definition}
A \emph{segment} is a set of the form
\[
[a,b]_\rho=\{\rho\nu_\rho^a,\rho\nu_\rho^{a+1},\dots,\rho\nu_\rho^{b}\}
\]
where $\rho$ is an irreducible cuspidal representation of $G_{r_0}$, $r_0>0$ and $a,b$ are integers such that $a\leq b$.
\end{definition}

Let $\Delta=[a,b]_\rho$ be a segment.
Define $\deg \Delta=r_0(b-a+1)$ and $\Delta^{\vee}=[-b,-a]_{\rho^{\vee}}$.

For any compatible character $\omega$ of $Z(G_{\deg\Delta})$ let
\[
\Pi(\Delta)_\omega=(\rho\nu_\rho^a\times\rho\nu_\rho^{a+1}\times\dots\times\rho\nu_\rho^{b})_\omega,\ \overleftarrow{\Pi}(\Delta)_\omega=(\rho\nu_\rho^{b}\times\rho\nu_\rho^{b-1}\times\dots\times\rho\nu_\rho^a)_\omega\in\Reps_\omega(G_{\deg\Delta}).
\]

As in the linear case we have

\begin{proposition}
The representations $Z(\Delta)_\omega:=\soc(\Pi(\Delta)_\omega)$ and $L(\Delta)_\omega:=\cos(\Pi(\Delta)_\omega)$
are irreducible and occur with multiplicity one in $\JH(\Pi(\Delta)_\omega)$.
They satisfy the following properties.
\[
L(\Delta)_\omega=\soc(\overleftarrow{\Pi}(\Delta)_\omega)\text{ and }Z(\Delta)_\omega=\cos(\overleftarrow{\Pi}(\Delta)_\omega).
\]
\[
\Pi(\Delta)_\omega^{\vee}=\overleftarrow{\Pi}(\Delta^{\vee})_{\omega^{-1}},\
Z(\Delta)_\omega^{\vee}=Z(\Delta^{\vee})_{\omega^{-1}}\text{ and }
L(\Delta)_\omega^{\vee}=L(\Delta^{\vee})_{\omega^{-1}}.
\]
\begin{align*}
r_{(r_0,\dots,r_0)}(Z(\Delta)_\omega)&=(\rho\nu_\rho^a\otimes\rho\nu_\rho^{a+1}\otimes\dots\otimes\rho\nu_\rho^{b})_\omega\\
r_{(r_0,\dots,r_0)}(L(\Delta)_\omega)&=(\rho\nu_\rho^{b}\otimes\rho\nu_\rho^{b-1}\otimes\dots\otimes\rho\nu_\rho^a)_\omega.
\end{align*}
For any integer $0\leq s\leq\deg\Delta$ we have
\begin{align*}
r_{(s,\deg\Delta-s)}(Z(\Delta)_\omega)&=
\begin{cases} (Z([a,a+s_0-1]_\rho)\otimes Z([a+s_0,b]_\rho))_\omega\quad &\text{if}\ s=s_0r_0, \\
	0 \quad & \text{if}\ r_0\ndiv s.
\end{cases}\\
r_{(s,\deg\Delta-s)}(L(\Delta)_\omega)&=
\begin{cases} (L([b-s_0+1,b]_\rho)\otimes L([a,b-s_0]_\rho))_\omega\quad &\text{if}\ s=s_0r_0, \\
	0 \quad & \text{if}\ r_0\ndiv s.
\end{cases}
\end{align*}
Moreover, $L(\Delta)_\omega$ is essentially square integrable.
\end{proposition}

Recall that hitherto, the character $\nu_\rho=\nu^{s_\rho}$ was chosen so that $(\rho\times\rho\nu_\rho)_\omega$
is reducible. At this point we can follow the argument of \cite{MR3573961}*{Lemma A.3}
to show the uniqueness of $s_\rho$, completing the proof of Proposition \ref{prop: BZPtwocuspidals}.

\begin{definition}
We say that two segments $\Delta_1,\Delta_2$ are weakly equivalent if we can write
$\Delta_i=[a,b]_{\rho_i}$, $i=1,2$ with $\rho_1$ weakly equivalent to $\rho_2$.
\end{definition}

The terminology is justified by the fact that the weak equivalence classes of $Z(\Delta)_\omega$ and $L(\Delta)_\omega$
depend only on the weak equivalence class of $\Delta$ (and not on $\Delta$ itself or on $\omega$).

Let $\Delta_1$ and $\Delta_2$ be two segments.
As in the ordinary case, we say that $\Delta_1$ and $\Delta_2$ are \emph{linked} if $\Delta_1\cup\Delta_2$
forms a segment and neither $\Delta_1\subset\Delta_2$ nor $\Delta_2\subset\Delta_1$.
In this case, we may write $\Delta_1=[a_1,b_1]_\rho$ and $\Delta_2=[a_2,b_2]_\rho$ for an irreducible cuspidal representation $\rho$ of $G_{r_0}$
and integers $b_1\geq a_1$, $b_2\geq a_2$ and moreover, either
$a_2>a_1$, $b_2>b_1$ and $b_1+1\geq a_2$ (in which case we say that $\Delta_1$ \emph{precedes} $\Delta_2$) or the symmetric condition.

The conditions above do not depend only on the weak equivalence class of $\Delta_1$ and $\Delta_2$.
We will say that $\Delta_1$ and $\Delta_2$ are \emph{weakly linked} if $\Delta'_1$ and $\Delta_2$ are linked
for some $\Delta'_1$ in the weak equivalence class of $\Delta_1$.
This notion will be more useful than linking itself.
Similarly, we say that $\Delta_1$ weakly precedes $\Delta_2$ if $\Delta'_1$ precedes $\Delta_2$
for some $\Delta'_1$ in the weak equivalence class of $\Delta_1$.
Of course, $\Delta_1$ and $\Delta_2$ are weakly linked if and only if either $\Delta_1$ weakly precedes $\Delta_2$
or $\Delta_2$ weakly precedes $\Delta_1$.

The crucial ingredient for the classification theorem is the following.

\begin{proposition} \label{prop: 2seg}
Let $\Delta_1$ and $\Delta_2$ be two segments and let $\omega$ be a compatible character of $Z(G_r)$
where $r=\deg\Delta_1+\deg\Delta_2$. Let
\[
\Pi=(Z(\Delta_1)\times Z(\Delta_2))_\omega,\ \ \Lambda=(L(\Delta_1)\times L(\Delta_2))_\omega.
\]
Then, $\Pi$ and $\Lambda$ are either both of length 2 or both irreducible, depending
on whether or not $\Delta_1$ and $\Delta_2$ are weakly linked.

Moreover, if $\Delta_1$ weakly precedes $\Delta_2$, then
\[
\soc\Pi\simeq (Z(\Delta_1\cap\Delta_2)\times Z(\Delta_1\cup\Delta_2))_\omega,\ \ \
\cos\Lambda\simeq (L(\Delta_1\cap\Delta_2)\times L(\Delta_1\cup\Delta_2))_\omega.
\]
\end{proposition}

For the unlinked case cf.\ \cite{MR3573961}*{Lemma A.7} which is proved by induction on $\deg\Delta_1+\deg\Delta_2$.
The special case of the irreducibility of the representations
\[
(Z([0,1]_\rho)\times \rho)_\omega,\ (Z([0,1]_\rho)\times \rho\nu_\rho)_{\omega'},\
(L([0,1]_\rho)\times \rho)_\omega,\ (L([0,1]_\rho)\times \rho\nu_\rho)_{\omega'},
\]
where $\rho$ is an irreducible cuspidal representation, is considered separately (cf.\ \cite{MR3573961}*{Lemma A.6}).

For the linked case see \cite{MR3573961}*{Lemma A.9} which relies on \cite{MR584084}*{\S 2 and \S 4}.

\subsection{Multisegments and classification}

By definition, a multisegment is a multiset of segments.

As usual, we view the set of multisegments as an ordered monoid.
Thus, we write a typical multisegment as
\[
\m=\sum_{i=1}^k\Delta_i
\]
where $\Delta_i$ is a segment.
We write $\deg \m=\sum_{i=1}^k\deg\Delta_i$ and
\[
\m^{\vee}=\Delta_1^\vee+\dots+\Delta_k^\vee.
\]

We say that two multisegments $\m$ and $\m'$ are weakly equivalent, denoted $\m\sim\m'$, if
we can write $\m=\Delta_1+\dots+\Delta_k$ and $\m'=\Delta'_1+\dots+\Delta'_k$
where $\Delta_i$ and $\Delta'_i$ are weakly equivalent segments for every $i$.

Clearly, the set of weak equivalence classes of multisegments is in bijection with the
set of multisets of weak equivalence classes of segments.

We can now state the classification theorem, which is analogous to the linear case
\cite{MR584084}*{Theorem 6.1}.

\begin{theorem}\label{thm: main}
Let $\m=\Delta_1+\dots+\Delta_k$ be a multisegment and $\omega$ a compatible character of $Z(G_{\deg\m})$.
Assume that the $\Delta_i$'s are enumerated such that $\Delta_i$ does not weakly precede $\Delta_j$ for every $1\leq i<j\leq k$.
(This is always possible.)

Then, up to isomorphism, the representations
\[
\zeta(\m)_\omega=(Z(\Delta_1)\times \dots\times Z (\Delta_k))_\omega\text{ and }
\lambda(\m)_\omega=(L(\Delta_1)\times \dots\times L (\Delta_k))_\omega
\]
depend only on $\m$ and $\omega$ and not on the enumeration of the $\Delta_i$'s subject to the above constraint.

Moreover, the representation
\[
Z(\m)_\omega=\soc(\zeta(\m)_\omega),\ \ \ (\text{resp., }L(\m)_\omega:=\cos(\lambda(\m)_\omega))
\]
is irreducible and occurs with multiplicity one in $\JH(\zeta(\m)_\omega)$ (resp., $\JH(\lambda(\m)_\omega)$).

We have $Z(\m)_\omega^{\vee}\simeq Z(\m^{\vee})_{\omega^{-1}}$ and $L(\m)_\omega^{\vee}\simeq L(\m^{\vee})_{\omega^{-1}}$.	

Let $\m'$ be a multisegment and $\omega'$ a compatible character of $Z(G_{\deg\m'})$.
Then,
\[
Z(\m)_\omega\simeq Z(\m')_{\omega'}\iff L(\m)_\omega\simeq L(\m')_{\omega'}\iff
\m\sim\m'\text{ and }\omega=\omega'.
\]

Finally, any genuine irreducible representation of $G_r$ with central character $\omega$
is of the form $Z(\m)_\omega$ (resp., $L(\m)_\omega$)
for a multisegment $\m$ such that $\omega$ is compatible with $\m$.

In particular, this gives rise to a bijection between the weak equivalence classes of irreducible
representations of $G_r$, $r\ge0$ and weak equivalence classes of multisegments
(i.e., multisets of weak equivalence classes of segments).
\end{theorem}

The proof relies on the special case of two segments (Proposition \ref{prop: 2seg}).

The first two parts are proved as in \cite{MR3049700}*{\S 5.2 and \S 5.3}.
An important technical step is to show that if the $\Delta_i$'s are pairwise unlinked, then the induced representations
$\zeta(\m)_\omega$ and $\lambda(\m)_\omega$ are irreducible.
This is proved as in \cite{MR3573961}*{Lemma A.7} by induction on $k$.

The proof of uniqueness and exhaustion parts follows \cite{MR3049700}*{\S 5.5, 5.6, 5.7}.
We omit the details.

\begin{remark}
When studying covering groups of classical groups one encounters a different covering of $\GL_r(F)$,
namely the pullback of Matsumoto's covering for the symplectic group of rank $r$ via the embedding of $\GL_r$ as a Levi subgroup.\footnote{In fact,
the resulting class in $H^2(\GL_r(F),\mu_n)$ is twice the class of $\sigma^{(0)}$ plus the class of $(\det g_1,\det g_2)_n$
\cite{1902.00880}*{Proposition 20}.}
This covering is simpler than the Kazhdan--Patterson covering (coming from the embedding of $\GL_r$ in $\SL_{r+1}$).
In particular, different blocks in the preimage of a Levi subgroup commute.
Therefore, the tensor product is the ordinary one and the classification theorem and its proof are as in the linear case
\cite{MR3573961}*{Appendix A}. We omit the details.
\end{remark}

%\bibliographystyle{alpha}
%\bibliography{../Bibfiles/all}
\def\cprime{$'$}

\end{document}